\documentclass[10pt,twoside]{amsart}
\pdfoutput=1
\allowdisplaybreaks
\usepackage{amsmath,amssymb
}
\usepackage{amsfonts}
\usepackage{amsthm,amscd}
\usepackage{amsmath}
\usepackage{amssymb}
\usepackage{amstext}
\usepackage{epsfig, color}
\usepackage{latexsym}
\usepackage{amscd,graphics}

\begin{document}
\newcommand{\comments}[1]{\marginpar{\footnotesize #1}} 
\newtheorem{innercustomgeneric}{\customgenericname}
\providecommand{\customgenericname}{}
\newcommand{\newcustomtheorem}[2]{%
  \newenvironment{#1}[1]
  {%
   \renewcommand\customgenericname{#2}%
   \renewcommand\theinnercustomgeneric{##1}%
   \innercustomgeneric
  }
  {\endinnercustomgeneric}
}
\newcustomtheorem{customthm}{Theorem}
\newcustomtheorem{customconj}{Conjecture}
\newcustomtheorem{customprop}{Proposition}
\newcustomtheorem{customlemma}{Lemma}

\newtheorem{proposition}{Proposition}[section]
\newtheorem{lemma}[proposition]{Lemma}
\newtheorem{sublemma}[proposition]{Sublemma}
\newtheorem{theorem}[proposition]{Theorem}
\newtheorem{conj}[proposition]{Conjecture}

\newtheorem{maintheorem}{Main Theorem}
\newtheorem{corollary}[proposition]{Corollary}

\newtheorem{ex}[proposition]{Example}

\theoremstyle{remark}

\newtheorem{remark}[proposition]{Remark}
\theoremstyle{definition}
\newtheorem{definition}[proposition]{Definition}

\newcommand{\ovB}{\bar{B}}
\def\Erg{\mathrm{Erg}\, }
\def\cone{\mathbf{C}} 
\def\real{\mathbb{R}}
\def\sphere{\mathbf{S}^{d-1}}
\def\integer{\mathbb{Z}}
\def\complex{\mathbb{C}}
\def\BBB{\mathbb{B}}
\def\supp{\mathrm{supp}}
\def\var{\mathrm{var}}
\def\sgn{\mathrm{sgn}}
\def\sp{\mathrm{sp}}
\def\TSR{\mathrm{TSR}}
\def\MT{\mathrm{MT}}
\def\id{\mathrm{id}}
\def\intt{\mathrm{int}}
\def\Imm{\mathrm{Image}}
\def\cc{\Subset}
\def\D{\mathrm {d}}
\def\I{i}
\def\E{e}
\def\Lip{\mathrm{Lip}}
\def\AA{\mathcal{A}}
\def\BB{\mathcal{B}}
\def\CC{\mathcal{C}}
\def\DD{\mathcal{D}}
\def\EE{\mathcal{Z}}
\def\FF{\mathcal{F}}
\def\GG{\mathcal{G}}
\def\HH{\mathcal{H}}
\def\II{\mathcal{I}}
\def\JJ{\mathcal{J}}
\def\KK{\mathcal{K}}
\def\LL{\mathcal{L}}
\def\LLL{\mathbb{L}}
\def\MM{\mathcal{M}}
\def\NN{\mathcal{N}}
\def \OO{\mathcal {O}}
\def \PP{\mathcal {P}}
\def \QQ{\mathcal {Q}}
\def \RR{\mathcal {R}}
\def\SS{\mathcal{S}}
\def\TT{\mathcal{T}}
\def\UU{\mathcal{U}}
\def\VV{\mathcal{V}}
\def\WW{\mathcal{W}}
\def\XX{\mathcal{X}}
\def\YY{\mathcal{Y}}
\def\ZZ{\mathcal{Z}}
\def\FFF{\mathbb{F}}
\def\PPP{\mathbb{P}}

\title[Fractional susceptibility function]{Fractional susceptibility functions for  the quadratic family: \\
Misiurewicz--Thurston parameters}

\date{March 22, 2022}

\author{Viviane Baladi and Daniel Smania}
\address{Laboratoire de Probabilit\'es, Statistique et Mod\'elisation (LPSM), CNRS, 
	 Sorbonne Universit\'e, Universit\'e
	de Paris, 4 Place Jussieu, 75005 Paris, France}
\email{baladi@lpsm.paris}

\address{Depto de Matem\'atica \\
   ICMC/USP,
               CP 668 \\ S\~ao Carlos-SP \\ CEP 13560-970 \\ Brazil}
\email{smania@icmc.usp.br} 

\thanks{
DS was partially supported by CNPq 306622/2019-0, CNPq 430351/2018-6 and FAPESP
Projeto Tem\'atico 2017/06463-3. We are grateful to the Brazilian-French Network in Mathematics for supporting VB's visit to DS in S\~ao Carlos  in 2015 and DS's visit to VB in Paris in  2017. This work was started when VB was working at IMJ-PRG. VB is grateful
to the Knut and Alice Wallenberg Foundation for invitations to Lund University 
in 2018, 2019, and 2020. The visit of VB to S\~ao Carlos in 2019 was supported by FAPESP Projeto Te\'matico 2017/06463-3. The visits of DS to Paris in 2018 and 2019 and VB's research are supported
by the European Research Council (ERC) under the European Union's Horizon 2020 research and innovation programme (grant agreement No 787304).
 We thank Magnus Aspenberg,   Genadi Levin, Tomas Persson, and Julien Sedro for
useful comments. We are very grateful to Clodoaldo Ragazzo for pointing out to us 
that Abel was the first to notice Lemma~\ref{halff+}.}

\begin{abstract} 
For $f_t(x)= t-x^2$   the quadratic family, we define the fractional    susceptibility
function $\Psi^\Omega_{\phi,t_0}(\eta, z)$ of $f_t$,  associated to a $C^1$
observable $\phi$ at a stochastic parameter ${t_0}$.
We  also define an approximate, ``frozen,'' fractional susceptibility function $\Psi^\mathrm{fr}_{\phi,t_0}(\eta, z)$  such that
$\lim_{\eta \to 1} \Psi^{\mathrm{fr}}_{\phi,t_0}(\eta, z)$ is the susceptibility function $\Psi_{\phi,t_0}(z)$ studied by Ruelle.
If ${t_0}$ is Misiurewicz--Thurston, we show that  $\Psi^\mathrm{fr}_{\phi,t_0}(1/2, z)$ has a pole at $z=1$ for generic
$\phi$  if 
$\JJ_{1/2}(t_0)\ne 0$, where  $\JJ_\eta(t)=\sum_{k=0}^\infty   \sgn(Df^k_{t}(c_1))|D f^k_{t}(c_1)|^{-\eta}$, with $c_1=t$ the critical value of $f_{t}$.
We introduce  ``Whitney'' fractional integrals
$I^{\eta,\Omega}$ and derivatives $M^{\eta,\Omega}$ on suitable sets $\Omega$.
We formulate conjectures on $\Psi^\Omega_{\phi,t_0}(\eta, z)$ and 	$\JJ_\eta(t)$, supported by our results on  $M^{\eta,\Omega}$ and  $\Psi^{\mathrm{fr}}_{\phi,t_0}(1/2, z)$, for the former, and numerical experiments, for the latter.
In particular, we expect that 	
$\Psi^\Omega_{\phi,t_0}(1/2, z)$ is singular at $z=1$ for
Collet--Eckmann  ${t_0}$ and  generic $\phi$.
 
 We view this work as a step towards the resolution
of the paradox that $\Psi_{\phi,t_0}(z)$ is holomorphic at $z=1$ for 
Misiurewicz--Thurston $f_{t_0}$ \cite{Ru0, JR}, despite lack of linear
response \cite{BBS}.
\end{abstract}


\maketitle
\tableofcontents

\section{Introduction}

For real\footnote{The map for $t=2$ is the full parabola $2-x^2$ on $[-2, 2]$, which can only
be perturbed by taking $t<2$. For $t=1$, we get a half-parabola on $[0, 1]$.}
 parameters $t\in (1,2)$, we consider the quadratic family
$$
f_t(x)=t-x^2 \, , \,\, \, \, x \in [-2,2] \, .
$$ 
 The critical point  is $c=c_{0,t}=0$, the critical value is
$c_{1,t}=t<|a_t|$ where $a_t:=\frac{-1-\sqrt{1+4t}}{2}\in (-2,0) $
satisfies   $f_t(a_t)=a_t=f_t(-a_t)$. More generally, we denote the postcritical points 
by $c_{k,t}=f^k_t(c)$ for $k\ge 0$.

We are  interested in the set $\SS$ of (so-called stochastic) parameters $t$ for which $f_t$ admits
an absolutely continuous invariant probability measure $\mu_t=\rho_t dm$. 
The set $\SS$ contains   the Collet--Eckmann (CE) parameters $t$, i.e.
those $t$ such that there exist
$\lambda_c >1$ and $K_0 \ge 1$ with
\begin{equation}\label{00}
|D f_t^k (c_{1,t}) |\ge  \lambda_c ^k\, , 
\quad \forall k \ge K_0 \, .
\end{equation}

Linear response is the study of differentiability of the map $t\mapsto \mu_t$, on suitable
subsets of $\SS$, in a suitable topology in the image, viewing $\mu_t$ as
a Radon measure or a distribution of higher order by introducing
smooth observables $\phi$.
In the simpler setting of families
$t\mapsto F_t$ of smooth expanding (or mixing smooth hyperbolic) maps
with $\partial_t F_t=X_t\circ F_t$,
the map
$$t\mapsto \RR_\phi(t):=\int \phi \,  \D\mu_t$$ 
is differentiable, and the derivative
$\partial _s \RR_\phi(s)$ at $s=t$ is, by
\cite{Ru00}, the value at $z=1$ of the susceptibility function, which is
the power series  (see also   \cite[\S 1]{ABLP})
\begin{equation*}
\Psi_\phi(z):=\Psi_{\phi,t}(z)=
\sum_{k=0}^\infty z^k  \int    (\phi\circ  F_t^k)' \, (X_t \rho_t) \, \D m \, .
\end{equation*}

Returning to the quadratic family $f_t$,
it is well known since the work of Thunberg \cite{Thun}
(see \cite[Theorem 1.30]{DT} for a more recent statement)
that\footnote{As a Radon measure, say --- using distributions
of higher order does not help.} $t\mapsto \mu_t$ is severely discontinuous if one does
not restrict to  Collet--Eckmann parameters with bounded constants.
However, this map is continuous when restricted to a suitable
(large) set of good parameters \cite{Tsu1}. More recently,
\cite[Cor 1.6]{BBS} showed that at almost every Collet--Eckmann 
parameter $t$, and for every $1/2$ H\"older observable
$\phi$, the function $\RR_\phi(s)$ is $\eta$-H\"older
for all $\eta <1/2$ at $s=t$, in the sense of Whitney,
on a set $\Omega_{<1/2}=\Omega_{<1/2}(t)$ of Collet--Eckmann parameters having $t$ as a density point.

One of the purposes of the present work is to reconcile two apparently contradictory results:
In 2005, Ruelle \cite{Ru0} considered the  full unimodal map $f_t$ 
 (and more generally,  Chebyshev polynomials $f_t$ of degree $D\ge 2$). He
  showed that the susceptibility function  (note that
$X_t:=\partial_t f_t 
\circ f_t^{-1} \equiv 1
$ for the quadratic family: Appendix~\ref{vanX} discusses
the condition that $X_t$ vanishes at endpoints)
  \begin{equation}\label{Rudef}
\Psi_\phi(z)=\Psi_{\phi,t}(z)=
  \sum_{k=0}^\infty z^k  \int    (\phi\circ  f_t^k)' \,  \rho_t \, \D m
\end{equation}
 admits a
meromorphic  extension to $\complex$. Ruelle also obtained
the remarkable fact that the residue of the
possible pole at $z=1$ vanishes
(for all observables $\phi \in C^1$).
Soon thereafter, with Jiang \cite{JR},  they generalised this result to the set $\MT$ of Misiurewicz--Thurston  parameters,
i.e., those $t$ for which there exist $L\ge 1$ 
and $P\ge 1$ with $y=f_t^L(c)$  periodic of minimal period $P$, with
$|D f^P_t (y)|>1$.
This raised the hope that $s\mapsto \RR_\phi(s):=\int \phi(x) \rho_s(x) \, \D m$ could be differentiable (in the sense
of Whitney, on an appropriate subset of $\SS$)
at $t\in \MT$, with $\partial_s \RR_\phi(s)|_{s=t}=\Psi_{\phi}(1)$.
In\footnote{In the decade between 2005 and 2015, the hope that $\RR_\phi(s)$ could be differentiable in the sense of Whitney
had already been diminished by the papers \cite{BS0} and \cite{Ru}.}
 2015, however, with Benedicks and Schnellmann \cite{BBS}, one of us showed
that for 
any mixing $t\in \MT$, there exist $\phi \in C^{\infty}$,  
 and  a set $\Omega_{1/2}=\Omega_{1/2}(t)\subset \SS$ containing $t$ as an accumulation
point such that 
\begin{equation}\label{LU}
	0 <\liminf_{\substack{\delta \to 0\\ t+\delta \in \Omega_{1/2}}}\frac{|\RR_\phi(t+\delta)-\RR_\phi(t)|}{ \sqrt {|\delta|}}\le \limsup_{\substack{\delta \to 0\\ t+\delta \in \Omega_{1/2}}}\frac{|\RR_\phi(t+\delta)-\RR_\phi(t)|}{ \sqrt {|\delta|}}< \infty \, .
\end{equation}
A hard open question is whether $t$ is a Lebesgue density
point of $\Omega_{1/2}$: In the affirmative, \eqref{LU} would not be
compatible with Whitney-differentiability  of $\RR_\phi(t)$ at $t$
in any natural sense, a strict paradox.
Otherwise, the bounds \eqref{LU}, may be compatible with differentiability in the sense
of Whitney,  although this would still be counter-intuitive.

\smallskip

Aiming to shed\footnote{Another goal is to give
a probabilistic analysis (analogous to  the central limit theorem of
de Lima--Smania \cite{dLS} in the piecewise expanding setting) of the breakdown of
$C^{1/2}$ regularity of the acim in transversal families of smooth 
unimodal maps with a quadratic critical point.} 
some light on this puzzling state of affairs, we introduce 
below, for $\Re \eta \in (0,1)$ and an appropriate positive measure set
$\Omega\subset \SS$, a two-variable fractional susceptibility function $\Psi^\Omega_{\phi,t}(\eta,z)$
in \S\ref{ladeff}.
The idea is to replace ordinary derivatives by fractional derivatives (Marchaud derivatives are convenient, in particular because they vanish on constants).
The main hurdle is that $\Omega$ has positive measure
but does not contain any nontrivial interval:
Despite the vast existing literature on fractional derivatives, we
did not find any suitable notion of fractional derivatives on such sets
(we propose a definition $M^{\eta,\Omega}$ in Section~\ref{Whitneyfrac}).
In Conjectures~\ref{laconj} and \ref{laconj+}, we formulate 
expected properties of $\Psi^\Omega_{\phi,t}(\eta,z)$.
We also introduce in \S\ref{ladeff} an approximate,  ``frozen'' fractional susceptibility function $\Psi^{\mathrm{fr}}_{\phi,t}(\eta,z)$, where the dynamics is frozen at
a parameter $t$ (so that $\Omega$ does not appear and ordinary Marchaud
derivatives can be used), and we study its properties in Theorem~\ref{WTF}
for $t\in \MT$. (We expect that the techniques of the proof can be extended to TSR parameters
defined in \eqref{sr}, see Remarks \ref{BS12exp} and \ref{CEinfty} and Footnote~\ref{foot0}.)

We next briefly discuss the organisation of the paper and  key points
in the proof of our main rigorous result, Theorem~\ref{WTF}.
Section~\ref{sec2} contains the definitions of the fractional susceptibility
functions. (Another approximate  fractional susceptibility function,
the response function  $\Psi^{\mathrm{rsp}}_{\phi,t}(\eta,z)$, is useful
to prove Theorem~\ref{WTF}.) Sections~\ref{Abel} and~\ref{mainsec} are  devoted to preparatory
material on fractional integrals and derivatives. 
We mention here that the case of piecewise expanding maps
\cite{Ba, BS0, BMS, ABLP} is easier, because the
invariant density appearing there is a sum of a nice function with a countable sum
of Heaviside functions.  For the quadratic maps, the invariant
density \eqref{Ruellerho} involves a sum of quadratic spikes. The fact, used in \cite{Ba, BS0, BMS},
that the derivative of a Heaviside function is a Dirac mass is mirrored in the
present work by Abel's remark that the one-sided half-integral 
of a quadratic spike is a Heaviside,
so that its one-sided Marchaud half derivative is a Dirac mass
(see  Lemmas~\ref{halff+} and~\ref{AbelMarchaud}). However,  {\it one-sided}
derivatives do not seem appropriate to define 
reasonable fractional susceptibility functions. The {\it two-sided} half  integrals, respectively derivatives, of quadratic spikes (Lemmas~\ref{halff} and ~\ref{AbelMarchaud}) involve an additional logarithmic, respectively\footnote{It is
natural that the half derivative of $\mathbf{1}_{x>c_k}(x-c_k)^{-1/2}$  involves $(x-c_k)^{-1}$, but we found no good reference for the computation.} polar, term.
The corresponding ``iterated pole'' is one of the 
features  of Theorem~\ref{WTF} 
in Section~\ref{rigorr}  (see Lemma~\ref{A2}).

An unexpected ingredient
of   Theorem~\ref{WTF} is a new half-transversality condition
$\JJ_{1/2}(t)\ne 0$ (see \eqref{eq.transversal}). Conjecture~\ref{conj2} on sums $\JJ_{\eta}(t)$ in \S\ref{conjB}
 is backed up
by our numerical results in \S\ref{algo'}. 

Finally, in \S\ref{bonusOmega} and \S\ref{five}, we 
introduce and study fractional
Whitney--Riemann--Liouville integrals $I^{\eta,\Omega}$ and Whitney--Marchaud  derivatives
$M^{\eta,\Omega}$ (in particular a ``Whitney version''
of Abel's remark) which support our conjectures on $\Psi^\Omega_\phi(\eta,z)$
and $\Psi^{\mathrm{fr}, \Omega}_\phi(\eta,z)$.
More precisely, as a stepping stone between the frozen function $\Psi^{\mathrm{fr}}_\phi(\eta,z)$ and   $\Psi^\Omega_\phi(\eta,z)$, we introduce
yet another approximate ``semifreddo'' function $\Psi^{\Omega,\mathrm{sf}}_\phi(\eta,z)$
in \S\ref{five}.
We expect that the approximate susceptibility functions $\Psi^{\mathrm{fr}}_\phi(\eta,z)$,  $\Psi^{\mathrm{rsp}}_\phi(\eta,z)$,
and  $\Psi^{\Omega, \mathrm{sf}}_\phi(\eta,z)$  have the same qualitative behaviour as $\Psi^\Omega_\phi(\eta,z)$ (Remark~\ref{conjC}).
Proposition~\ref{propE}, proved in \S\ref{dealing},   shows that the approximate functions $\Psi^{\mathrm{fr}}_\phi(\eta,z)$ and  $\Psi^{\mathrm{rsp}}_\phi(\eta,z)$
   tend to $\Psi_\phi(z)$
as $\eta\to 1$ as formal powers series in $z$ (i.e., convergence of the coefficients of each individual $z^k$).

\bigskip

In the remainder of this Introduction, we flesh out 
the synopsis given above.

\subsection{Conjecture~\ref{laconj} on the fractional susceptibility function
$\Psi^\Omega_{\phi}(\eta, z)$}

We say that a parameter $t\in (1,2)$ is TSR  if  $f_t$ is Collet--Eckmann and satisfies Tsujii's \cite[(WR)]{Tsu} condition, i.e.,
\begin{equation}\label{sr}
\lim_{\eta\rightarrow 0^+} \liminf_{n\rightarrow \infty} 
\frac{1}{n} 
\sum_{\substack{1\leq j\leq n  \\ |f_t^j(c) -c|< \eta }} \ln |f_t'(f_t^j(c))|=0
\, .
\end{equation}
TSR is  implied by polynomial recurrence and implies
Benedicks--Carleson exponential recurrence (see e.g. \cite[Proposition 2.2 and references]{BS2},
also for a topological definition of TSR).
Tsujii constructed in \cite[Theorem 1 (I)]{Tsu}  a positive measure subset  $\Omega\subset \SS$  of 
 TSR parameters  
  such that,
setting $\Omega^c = \real\setminus \Omega$, and letting $m$ denote Lebesgue measure,
	\begin{equation}\label{tscond0}
	\lim_{\delta \to 0}
	\frac{m([t-\delta, t+ \delta] \cap  \Omega^c)}{\delta^\beta} = 0 \, ,
\quad \forall t \in \Omega \, ,
	\end{equation}
for all $\beta <2$ (in particular each $t\in \Omega$ is a Lebesgue density point of
$\Omega$).

\smallskip
The transfer operator associated to  $f_t$ 
is defined on $L^1([-2,2],dm)$ by setting 
\begin{equation*}
\LL_t \varphi(x)=
\sum_{f_{t}(y)=x}\frac{ \varphi(y)}{|Df_{t}(y)|} =\mathbf{1}_{x<t}
\frac{ \varphi(\sqrt{t-x}) + \varphi(-\sqrt{t-x})}{2 \sqrt{t-x}} \, .
\end{equation*}
The dual of $\LL_t$ fixes Lebesgue measure restricted
to $I_t:=  [a_t,-a_t]$ so that $f_t(I_t)\subset (I_t)$.

\smallskip
 For $t\in (1,2)$ a fixed TSR
parameter, it is  convenient to extend $s\mapsto f_s$ as a Lipschitz map
to the whole line as follows: choosing $\epsilon=\epsilon(t)>0$ such that 
$[t-\epsilon,t +\epsilon]\subset ( 1, 2)$, and such that\footnote{Recall that $\supp(\rho_t)=[c_{2,t}, c_{1,t}]$.} $[c_{2,t}, c_{1,t}]
\subset \intt(\cap_{\tau\in [t-\epsilon,t+\epsilon]} I_\tau)=: I_{t,\epsilon}$,
set 
\begin{equation}\label{extendd}
f_\tau=f_{t}\, \mbox{ if } |\tau-t|<\epsilon\, , 
f_\tau=f_{t-\epsilon} \mbox{ for all } 
\tau < t-\epsilon \, , \mbox{ and  } f_\tau=f_{t+\epsilon} \mbox{ for all }
\tau > t+\epsilon
\, .
\end{equation}
Then, for $\Omega\subset \TSR$ having
$t$ as a Lebesgue density point, and $\phi$ a compactly supported $C^1$ function,  the  {\it fractional susceptibility function} $\Psi^\Omega_\phi(\eta,z)=\Psi^\Omega_{\phi,t,\epsilon}(\eta,z)
$ for the quadratic family
at $t$ is the function of two complex variables $\eta$ and $z$ 
$$ 
\Psi^\Omega_\phi(\eta,z)=
\frac{\eta}{2 \Gamma(1-\eta)}
\sum_{k=0}^\infty  z^k  \int
\int_{\delta\in \real \cap (\Omega-t)} \phi ( f_{t+\delta}^k (x)) 
\nonumber  
\frac{(\LL_{t+ \delta}-\LL_{t} )  \rho_{t}(x) }{|\delta|^{1+\eta}} \sgn(\delta) \, \D \delta
\D x \, ,
$$
(writing $\D x=\D m(x)$, $\D \delta=\D m(\delta)$),  in the
sense of formal power series in $z$, for fixed $\eta$
with $\Re \eta \in (0,1)$. (Motivation and details are
given in \S\ref{ladeff}.)

For $\Omega$ satisfying
\eqref{tscond0} for some $\beta>1$, we  define  a ``Whitney--Marchaud'' 
	fractional derivative $M^{\eta,\Omega}$ in \S\ref{five}.
For $\eta\in (0,1)$, Proposition~\ref{natomega} in \S\ref{five} gives  conditions
on $g$ and $\Omega$ ensuring that
\begin{equation*}
		\lim_{\zeta \uparrow \eta}
		\biggl ( \frac{\Gamma(1-\zeta)}{\eta\cdot  \Gamma(\eta-\zeta)}	M^{\zeta,\Omega} g (t) \biggr ) =	\lim_{\delta \to 0, t+\delta \in \Omega}
	\frac{g(t+\delta)-g(t)}{\sgn(\delta) {|\delta|^\eta}}
		\, .
\end{equation*}

\smallskip

We can now state our main conjecture\footnote{The threshold  for $\eta$ below
is $1/2$;  for families with criticality $d$ the
expected  threshold
is  $1/d$.}:

\begin{customconj}{A}
\label{laconj}
For  almost every mixing\footnote{Some results of
\cite{BBS} require polynomial
recurrence. We expect that this is an artefact of the method used
there, but maybe TSR must be strengthened to polynomial recurrence.}  $t\in \TSR$, there exist
$\bar\lambda_t>1$,  $\epsilon>0$,
and a set $\Omega=\Omega(t)\subset \TSR$ containing $t$ and satisfying 
\eqref{tscond0}  for  all
$\beta <2$, such that,
 for any 
compactly supported
$C^1$ function $\phi$, and any $N\ge 1$, the following  holds:

\begin{itemize}

 \item[i.] For any $\eta$ with  $0<\Re \eta < 1/2$, there exists a
 disc\footnote{All discs in the present work are centered at the origin.} 
 $D_\eta$ of radius $>1$ such that 
  $\Psi^{\Omega}_{\phi,t}(\eta,z)$  is holomorphic in $\{(\eta,z)\mid 0<\Re \eta < 1/2\, , \,
z\in D_\eta\}$. 

\item[ii.] For any real $0<\eta <1/2$, we have the
{\it fractional response formula}
\begin{equation}\label{axiom}
\Psi^\Omega_{\phi, t}(\eta,1)=
M^{ \eta, \Omega}_{s=t} \int \phi (x)\rho_s(x) \, \D x \, .
\end{equation}

 \item[iii.] The power series
 $\Psi^{\Omega}_{\phi,t}(1/2,z)$ defines a holomorphic function in the
open unit disc. For a generic $C^N$ function $\tilde \phi$: 
the unit circle is a natural boundary for this function;
the limit  as $z\in (0,1)$ tends to $1$ of
$\Psi^{\Omega}_{\tilde \phi,t}(1/2,z)$ 
does not exist; the limit  as $z\in (0,1)$ tends to $1$ of 
$(z-1)\Psi^{\Omega}_{\tilde \phi,t}(1/2,z)$, if it exists, does not vanish. 

\item[iv.] For any $\eta$ with  $\Re \eta\in ( 1/2,1)$  there exists 
a disc $D_\eta$ with radius in $(1/\bar \lambda_t, 1)$ such that the function $\Psi^{\Omega}_\phi(\eta,z)$ is 
holomorphic in 
$\{(\eta,z)\mid 0<\Re \eta < 1/2\, , \,
z\in D_\eta\}$. 
For any $\eta$ with  $\Re \eta\in ( 1/2,1)$
and any generic $C^N$ function $\tilde \phi$  we have that
$\Psi^{\Omega}_{\tilde \phi,t}(\eta,z)$ 
has a
singularity in the open unit disc.

\item[v.] We have
$
\lim_{\eta \uparrow 1} \Psi^{\Omega}_{\phi,t}(\eta, z)=\Psi_{\phi,t}(z) 
$ as formal power series in $z$ (recall \eqref{Rudef}).
\end{itemize}
\end{customconj}

\smallskip

For families of  piecewise expanding maps, a more precise version
of [iii] for the ordinary susceptibility function  $\Psi_{\phi,t}(z)$  (similar to Conjecture~\ref{laconj+} below) was established  \cite[Theorem 1]{BMS},
using results  in \cite{Ba, BS0}. (We expect that other results of \cite{BMS}, on the iterated logarithm law e.g.,
can be adapted to the quadratic family.)

Also in the piecewise expanding setting, the analogue of [i] and [ii] in Conjecture~\ref{laconj},  replacing $1/2$ 
by $1$, and taking $\Omega$ to be a neighbourhood of $t$, has been established
in\footnote{The
weighted Marchaud derivatives in  \cite{ABLP}  could be useful to understand
the logarithmic factors appearing in \cite{BBS}.} \cite{ABLP}.

\smallskip
We explain next how the conjectured properties of $\Psi^{\Omega}_\phi(\eta, z)$ reflect  the behaviour
described in \cite{BBS} of the 
absolutely continuous invariant measure and may also contribute to resolve the paradox\footnote{The ``averaging'' response studied in \cite[\S3]{WG1} and  \cite[(16)]{WG2}
does not resolve the paradox, see Appendix~\ref{wormell}.}
arising from comparing the results of \cite{Ru0} and \cite{JR} with those of  \cite{BBS}.
(The fractional susceptibility function being holomorphic in two variables also raises the hope to use tools such as
Hartog's extension theorem.)

First, the $\eta$ H\"older
upper bounds of \cite{BBS} on $\Omega_{<1/2}$
together with Proposition~\ref{natomega}  and 
[ii] in Conjecture~\ref{laconj} would imply that, 
if   $\Omega_{<1/2}$ satisfies
\eqref{tscond0} for some $\beta>1$,
 \begin{equation*}
\lim_{\zeta \to\eta} \frac{\Psi^{\Omega}_{\phi, t}(\zeta,1)}
{\Gamma(\eta-\zeta)}=\frac{1}{\Gamma(1-\eta)}
\lim_{\delta \to 0, t+\delta \in \Omega}
\frac{ \RR_\phi (t+\delta)- \RR_\phi (t)  }
{\sgn(\delta)|\delta|^\eta}
=0 \, ,\,\, \forall \eta\in (0,1/2)\, .
\end{equation*}

Next, if [ii] in Conjecture~\ref{laconj} could be established at any $\eta\in[ 1/2,1)$ for which
 either side of \eqref{axiom} is well defined, then we would have for any  $t$ at which $\RR_\phi(t)$
is $\Omega$-Whitney $1/2$ differentiable (Definition~\ref{relevant}, Proposition~\ref{natomega} in \S\ref{five})
\begin{equation}\label{**}
\lim_{\zeta \uparrow 1/2} \frac{\Psi^{\Omega}_{\phi, t}(\zeta,1)}
{\Gamma(1/2-\zeta)}=\frac{1}{\Gamma(1/2)}
\lim_{\delta \to 0, t+\delta \in \Omega}
\frac{ \RR_\phi (t+\delta)- \RR_\phi (t)  }
{\sgn(\delta) |\delta|^{1/2}} \, .
\end{equation}
(If $t\in \MT$, \cite{BBS} furnishes  upper and lower bounds on $( \RR_\phi (t)- \RR_\phi (t+\delta_n)  )/|\delta_n|^{1/2}$,
for suitable sequences $\delta_n \to 0$,    but the existence of the limit in the right-hand side of
\eqref{**} is not known.)

If the last claim of [iii] in Conjecture~\ref{laconj} holds we also expect that, for  $\Omega$ satisfying Tsujii's
condition \eqref{tscond0} for all $\beta <2$, and generic $\phi$,
\begin{equation}\label{firstst}
\lim_{\zeta \uparrow 1/2} \frac{\Psi^{\Omega}_{\phi, t}(\zeta,1/2+\zeta)}
{\Gamma(1/2-\zeta)}\ne 0\, , \qquad
\lim_{\eta \uparrow 1/2}	\frac{\Psi^\Omega_{\phi,t}(\eta,1) }{ \Gamma(1/2-\eta)}
\ne 0 \, .
\end{equation}
In view of Proposition~\ref{natomega} in  \S\ref{five}, the above inequality would establish that
$\RR_\phi$ is not $\Omega$-Whitney $\eta$-differentiable if $\eta >1/2$
(Definition~\ref{relevant}).
In particular,  the ordinary susceptibility function 
\eqref{Rudef} at $z=1$
could not be interpreted as a derivative.
(The singularity  of $\Psi_t(z)$ in the open unit interval could
then be 
a ``scar'' of the singularity at $z=1$ of $\Psi^{\Omega}_{\phi,t}(\eta,z)$
for some $\eta<1$, presumably $\eta=1/2$.)
The inequalities \eqref{firstst} could be useful to determine
whether $t$ is a density point of the set $\Omega_{1/2}$ in \eqref{LU}.

\begin{remark}[Tangential families] In view of the linear response
result in \cite{BS2},  replacing the quadratic family by a ``tangential'' family
$\tilde f_\tau$ of smooth  unimodal maps all topologically conjugated to a TSR map
$\tilde f_t$, we expect that, taking $\Omega$ a small enough
neighbourhood of $t$,
claims [i] and [ii] in Conjecture~\ref{laconj}, hold,
replacing $1/2$ by $1$, and, in addition,
$$
\lim_{\eta \uparrow 1} \Psi^\Omega_{\phi,\tilde f_t} (\eta,1)
 = \lim_{ \tau \to t}
 \frac{\RR_\phi(t)  -\RR_\phi (\tau)  }{t-\tau} \, .
$$

It would be interesting, but more challenging, to investigate whether ``tangentiality'' of a family
$\tilde f_t$
at a single point $t_0$ 
implies some additional (Whitney) regularity of the response at $t_0$.
\end{remark}

\subsection{Fractional transversality $\JJ_\eta$. Conjectures~\ref{conj2} and \ref{laconj+}}

It is well known that 
all\footnote{\label{foot1}In fact all ``summable'' parameters, i.e. those
	for which $\JJ({t})$ is absolutely convergent, are transversal,
see \cite[Cor~1.b]{Levin} and \cite[Cor~A.4]{A}.}
Collet--Eckmann parameters  $t_1$
are transversal (see \cite[Theorem 3]{Tsu2})
in the sense of Tsujii
\cite{Tsu} (see also Appendix~\ref{vanX}), i.e.
\begin{equation}
\label{eq.transversal}
\JJ({t}):=\sum_{j=0}^\infty\frac{\partial_\tau f_\tau(c_{j,\tau})|_{\tau=t}}{D f_{t}^j(c_{1,t})}=\sum_{j=0}^\infty\frac{1}{D f_{t}^j (c_{1,t})}\neq0\,.
\end{equation}

\smallskip
To state  Conjecture~\ref{laconj+}
and the fractional transversality
condition  appearing in Theorem~\ref{WTF} (see \S\ref{SS3}), setting $\sgn(x)=\frac{x}{|x|}$
for  $x\in \real_*$, and 
$\sgn(0)=0$, we let 
\begin{equation}\label{sigmak}
s_0=1\, ,\quad s_k:=s_{k,t}=\sgn(Df^k_t(c_{1,t})) \in \{ -1, +1\} 
\, , \quad k \ge 1 \, .
\end{equation}
Then, we define,
for  $t>1$ such that $f^k_{t}(c)\ne c$ for all $k\ge 1$, 
and $\eta >0$, 
\begin{equation}
\label{JMT}
\JJ_\eta(t)=\sum_{k=0}^\infty   \frac{s_{k,t}}{|D f^k_{t}(c_{1,t})|^\eta} \, ,
\end{equation}
whenever the sum converges absolutely, and in this case we say that $t$ satisfies the {\it $\eta$-summability condition.}   
Note that the parameter
$t=2$ (the full quadratic map) satisfies $\JJ_{1/2}(2)=0$.
We expect that $t=2$ is the only
$1/2$-summable parameter where  the {\it fractional transversality condition}
$\JJ_{1/2}(t)\ne 0$ fails: This is the main claim 
of {\bf  Conjecture~\ref{conj2},} supported by numerics,  in Section~\ref{numerics}.

\smallskip
Now, if $f_t$ is Collet--Eckmann,  setting 
$u_t=-{\rho_t(0) \cdot \sqrt \pi }/{2}\ne 0$, we put
\begin{equation}\label{defUU}
\UU_{ 1/2}(z)=\UU_{ 1/2, t}(z):=u_t \cdot \sum_{k=0}^\infty \frac{s_k}{z^k \sqrt{|Df_t^k(c_1)|}}\, .
\end{equation}
The function $\UU_{ 1/2}(z)$ 
is  holomorphic outside of the disc of radius $1/\sqrt \lambda_c$,
with   $\UU_{ 1/2,t}(1)\ne 0$ if and only if 
  $\JJ_{1/2}(t)\ne 0$. We shall also need the power series 
\begin{equation}\label{defUU+}
\UU^+_{ 1/2}(z)=\UU^+_{ 1/2, t}(z):=
u_t \cdot \sum_{k=0}^\infty \frac{1}{z^k \sqrt{|Df_t^k(c_1)|}}\, .
\end{equation}
The function
$\UU^+_{ 1/2}(z)$
is  holomorphic outside of the disc of radius $1/\sqrt \lambda_c$,
with  
\begin{equation*}
\JJ^+_{1/2}(t):=\sum_{k=0}^\infty   \frac{1}{\sqrt{D f^k_{t}(c_{1,t})|}}= \frac{\UU^+_{ 1/2, t}(1)}{u_t} \, \ne 0
\, .
\end{equation*}
Next, following \cite{BMS}  (where this function was denoted $\sigma_\phi$) 
we set, for   $\phi\in C^0$,
\begin{align}\label{defSig}\Sigma_\phi(z)= \Sigma_{\phi,t}(z):=\sum_{\ell=1}^\infty \phi(c_{\ell,t})z^{\ell-1}\, .
\end{align}
($\Sigma_\phi(z)$  is holomorphic in
the open unit disc. If $t \in \MT$, then $\Sigma_\phi(z)$ is rational.)

Recall that if $\phi:\real \to \complex$ is  $C^0$, compactly supported,
 and $C^1$ at $y\in \real$, the Hilbert transform of $\phi$ at $y$
is defined by the  Cauchy principal value
(see also \S\ref{prel})
\begin{equation}\label{Hilb}
(\HH \phi)(y):=\frac{1}{\pi} p.v. \int \frac{\phi(x)}{y-x} \, \D x \, .
\end{equation}
Then, for $\phi$ a $C^1$ function, we define a formal power series
\begin{align}\label{defSigH}
 \Sigma^\HH_\phi(z)= \Sigma^\HH_{\phi,t}(z):= \sum_{\ell=1}^\infty s_{\ell,t}\cdot \HH(\mathbf{1}_{I_t} \phi)(c_{\ell,t})z^{\ell-1}\, .
\end{align}
Finally, for $r>0$, $q>1$, and  a bounded sequence $\tilde\psi_t(\ell)$ of functions in 
the Sobolev space 
$H^r_q[-2,2]=\{ \varphi\mid \mathbf{1}_{[-2,2] }\cdot \varphi \in H^r_q(\real)\}$,
 we introduce the formal power series  
\begin{align}\label{defSigH2}
 \Sigma^{\tilde\psi_t}_t(z)=\sum_{\ell=1}^\infty s_{\ell,t}\cdot \tilde\psi_t(\ell)   z^{\ell-1} \, .
\end{align}

\smallskip
We can now state the announced complements to [iii] in Conjecture~\ref{laconj}:

\begin{customconj}{A+}\label{laconj+}
For $t$, $\Omega$, and $\phi$ as in Conjecture~\ref{laconj}, we have 
\begin{align*}
\Psi^{\Omega}_{\phi,t}(1/2,z)&=\UU_{1/2,t}(z) \Sigma_{\phi,t}(z) 
+ \WW^{\Omega}_{\phi, 1/2, t}(z)
+  \VV^{\Omega}_{\phi, 1/2, t}(z)\, ,
\end{align*}
with $\VV^{\Omega}_{\phi,1/2, t}(z)$ holomorphic in an open annulus $\AA$ containing $\SS^1$.  Moreover, there exist  $r>0$, $q>1$, and
functions
$\tilde\psi_t (\ell)\in H^r_q[-2,2]$, with $\int_{I_t} \tilde\psi_t(\ell)\, \D m=0$, 
such that 
\begin{align*}
 \WW^{\Omega}_{\phi, 1/2, t}(z)= \UU^{+}_{1/2,t}(z) \bigl [ \Sigma^\HH_{\phi,t}(z)
+  \sum_{k=0}^\infty z^k \int
(\phi \circ f_t^k)  \cdot  \Sigma^{\tilde \psi_t}_t(z)  \, \D m\bigr ]\, .
\end{align*}
Finally, 
$\Sigma^{\tilde \psi_t}_t(z)$, and, for generic $\tilde \phi \in C^N$ (any $N\ge 1$) the functions
 $\Sigma_{\tilde \phi,t}(z)$ and $\Sigma^\HH_{\tilde \phi,t}(z)$
  are holomorphic
in the open unit disc and have a natural boundary on
$\SS^1$.
\end{customconj}

\begin{remark}[Approximate Susceptibility Functions]\label{conjC}
We expect that claims [i], [iii], and [iv] (but not [ii])
of Conjecture~\ref{laconj}, as well as the claims of
 Conjecture~\ref{laconj+}, hold for the three approximate fractional
susceptibility functions $\Psi^\mathrm{fr}_\phi(\eta,z)$, $\Psi^{\mathrm{rsp}}_\phi(\eta,z)$, and $\Psi^{\Omega, \mathrm{sf}}_\phi(\eta,z)$, keeping the
same functions $\UU_{1/2}$, $\UU^+_{1/2}$,  $\Sigma_\phi$, $\Sigma^\HH_\phi$, and  $\Sigma^{\tilde \psi}$, and 
replacing $\WW^{\Omega,reg}_{\phi, 1/2}(z)$ and
$\VV^{\Omega}_{\phi, 1/2}(z)$ by suitable $\WW^{*}_{\phi, 1/2}(z)$ and
$\VV^{*}_{\phi, 1/2}(z)$,  for $*=\mathrm{fr}$,
$\mathrm{rsp}$, and $(\Omega,\mathrm{sf})$, respectively. 
Claim [v] for the  approximate fractional susceptibility
functions $\Psi^\mathrm{fr}_\phi(\eta,z)$ and $\Psi^{\mathrm{rsp}}_\phi(\eta,z)$ is the content of Proposition~\ref{propE}.
\end{remark}
\subsection{Frozen  and response susceptibilities:  Theorem~\ref{WTF} and Proposition~\ref{propE}}\label{SS3}

We move to the rigorous results.
To keep this ``proof of concept'' paper short, 
we will focus on the countable subset $\MT\subset \SS$ of Misiurewicz--Thurston (MT)
parameters. This toy model setting allows us to present  new ideas  with the least possible technicalities.  In addition, the ``paradox'' discussed above occurs at MT parameters \cite{BBS}.

We
shall mostly study  here an approximate fractional susceptibility
function, the {\it frozen
	fractional susceptibility function} (Definition~\eqref{defr})
\begin{align*}
\Psi^\mathrm{fr}_\phi(\eta,z)=\Psi^\mathrm{fr}_{\phi,t}(\eta,z)
&= \sum_{k=0}^\infty  z^k \int
 (\phi \circ f_{t}^k) (x) 
M^\eta_s (\LL_s \rho_{t}(x))|_{s=t} \,
\D x\, ,
\end{align*}
where $M^\eta_s$ is the two-sided Marchaud fractional\footnote{We recall definitions in \S\ref{defMar}. A good introduction to fractional derivatives is the book \cite{MR}. See also the  short introduction \cite{OM} and the treatise
\cite{sam}.}  derivative of order $\eta$ and $\phi$ is $C^1$
and supported in $[-2,2]$.

Sedro \cite{Sedro} has recently proved  item [i] of
Conjecture~\ref{laconj} for $\Psi^\mathrm{fr}_\phi(\eta,z)$ for  Misiurewicz parameters.

{\it Our main rigorous result, {\bf Theorem~\ref{WTF},}  stated in Section~\ref{state}, furnishes
the analogue of  Conjecture~\ref{laconj+} for $\Psi^{\mathrm{fr}}(1/2,z)$,
considering  parameters $t\in \MT$.} In the MT case, the functions
$\Sigma_\phi$, $\Sigma^\HH_\phi$ and $\Sigma^{\tilde\psi}$ are rational
and the singularities of $\Psi^{\mathrm{fr}}_\phi(1/2,z)$ on the unit circle are simple
poles. (We also expect this to hold for $\Psi^{\Omega}_{\phi,t}(1/2,z)$
if $t\in \MT$.)

\smallskip

We also introduce (Definition~\ref{defrsp}) a {\it response fractional
susceptibility function} by taking the Marchaud derivative with respect to
$x$
\begin{align*}
\Psi^{\mathrm{rsp}}_\phi(\eta,z)=\Psi^{\mathrm{rsp}}_{\phi,t}(\eta,z)
&
=
\sum_{k=0}^\infty  z^k \int_{I_t}  M^\eta_x (\phi \circ f_{t}^k)
\cdot  \rho_{t} \,\D x \, .
\end{align*}
The response function is related to the frozen susceptibility function (Proposition~\ref{CC2})
and will be used to prove Theorem~\ref{WTF}.
(See \cite{ABLP} for a  fractional response function  in the piecewise expanding setting.)

\smallskip

Although their value at $1$ is not expected to coincide with
 $M^{\eta,\Omega} \RR_\phi(t)$, 
we believe that $\Psi^\mathrm{fr}_\phi(\eta,z)$ and $\Psi^{\mathrm{rsp}}_\phi(\eta,z)$
share the qualitative properties of $\Psi^\Omega_\phi(\eta,z)$ (Remark~\ref{conjC}).
Finally, recalling \eqref{Rudef}, Proposition~\ref{CC2} and Lemma~\ref{limitok} imply (see \S\ref{dealing}):
\begin{customprop}{D}\label{propE}
As formal power series,
$$
\lim_{\eta \uparrow 1} \Psi^\mathrm{fr}_\phi(\eta,z)=
\lim_{\eta \uparrow 1} \Psi^{\mathrm{rsp}}_\phi(\eta,z) =\Psi_\phi(z) \, .
$$
\end{customprop}

\subsection{Whitney fractional  integrals and derivatives:
Abel's remark and the semifreddo fractional susceptibility function  $\Psi^{\Omega,\mathrm{sf}}_\phi(\eta,z)$}

 In 
\S\ref{bonusOmega} we introduce Whitney fractional integrals $I^{\eta,\Omega}$ and prove Lemma~\ref{halff+omega}, the analogue of Abel's remark
for $I^{1/2,\Omega}$ (and suitable sets $\Omega$ satisfying \eqref{tscond0}).
In \S\ref{five}, we introduce Whitney--Marchaud derivatives $M^{\eta,\Omega}$, and
use them to define the {\it semifreddo fractional susceptibility
function} $\Psi^{\Omega,\mathrm{sf}}_\phi(\eta,z)$,
a stepping-stone to the fractional susceptibility function from
its frozen version. 
Proposition~\ref{natomega} gives conditions ensuring
$\lim_{\eta \uparrow 1}M^{\eta, \Omega}g(x)=g'_\Omega(x)$, where
$g'_\Omega(x)$ is the $\Omega$-Whitney derivative of $g$ at $x\in \Omega$,
from Definition~\ref{relevant}.
\S\ref{five} also contains Proposition~\ref{natomega} on
$\lim_{\eta \uparrow \zeta}
		 \bigl [ \frac{\Gamma(1-\eta)}{\Gamma(\zeta-\eta)}	(M^{\eta,\Omega} g) (x)  \bigr ]$.

\section{Defining fractional susceptibility functions}
\label{sec2}

\subsection{Preliminaries. Hilbert transform. Gamma and Beta functions}
\label{prel}
We next record classical facts for further use.
First, the
definition \eqref{Hilb} of the Hilbert transform
can be explicited as 
$(\HH \phi)(y)=-\frac{1}{\pi}\lim_{\delta\downarrow 0}
 \int_{\delta}^\infty \frac{\phi(y+u)-\phi(y-u)}{u} \, \D u$.
If $\phi$ is $C^1$ and compactly supported then $\HH \phi$ coincides with the following distributional derivative
\begin{equation*}
(\HH \phi)(y)=\frac{d}{d y}
\frac{1}{\pi}  \int \phi(x) \log |{y-x}| \, \D x \, ,
\end{equation*}
and the Cauchy principal value corresponds
to integration by parts,
since 
\begin{align*}
\frac{d}{d y}\frac{1}{\pi}  \int \phi(x) \log |{y-x}| \, \D x
&=\frac{d}{d y}\frac{1}{\pi}  \int \phi(y-u) \log |u| \, \D u
=\frac{1}{\pi}  \int \phi'(x) \log |{y-x}| \, \D x  .
\end{align*}
Note that there exists $C<\infty$ such that for any compact interval $J$
\begin{equation*}
|\HH (\mathbf{1}_{J}\phi)(x)|\le  C |J| \sup |\phi'|   \, , \forall  x \in \mbox{int} (J )\, .
\end{equation*}
Euler's Gamma function is $\Gamma(\eta)=\int_0^\infty x^{\eta-1} e^{-x} \, \D x$ 
 (recall that it has simple poles
at $\eta = 0, -1, -2,...$). 
The Beta function is defined for $\Re x>0$ and $\Re y>0$ by
$$
B(x,y)=\int_0^1 u^{x-1} (1-u)^{y-1} \, \D u \, .
$$
It satisfies $\Gamma(x) \Gamma(y)=B(x,y) \Gamma(x+y)$.
Since $\Gamma(3/2)=\sqrt \pi/2$, 
$\Gamma(1)=\Gamma(2)=1$, and $\Gamma(1/2)=\sqrt \pi$, we have
$B(1/2,1/2)=\pi$ and $B(1/2,3/2)=\pi/2$. Recall also that $\sin(\pi/4)=\cos(\pi/4)=\sqrt{2}/2$.

\subsection{Susceptibility
functions $\Psi^\Omega_\phi(\eta,z)$, $\Psi^\mathrm{fr}_\phi(\eta,z)$, $\Psi^\mathrm{rsp}_\phi(\eta,z)$. Proposition~\ref{propE}}
\label{ladeff}
We first motivate heuristically our definition of the fractional susceptibility
function $\Psi^\Omega_\phi(\eta,z)$. The starting point is
the right-hand side of \eqref{axiom} in [ii] from
Conjecture~\ref{laconj}, i.e. the Marchaud derivative of $\RR_\phi(t)$. Our first task is to
rewrite 
$$
\RR_\phi(s)-\RR_\phi(t)=
\int \phi \rho_{s} \D m - \int \phi \rho_{t} \D m
$$ 
along the lines of \cite{ABLP}:
If $s$ belongs to a suitable subset of $\Omega$ of $CE$, then  
for every  $r>0$, and $q>1$ there exists $\kappa<1$
such that for any bounded function $\phi$ supported in $[-2,2]$
and any  $\psi\in H^r_q[-2,2]$ with $\int_{I_t}  \psi \, \D m=0$, there exists $C_{\phi,  \psi}$ such that
$$
|\int \phi \, \LL_s^k( \psi)  \, \D m |=
|\int  (\phi \circ f_s^k) \,  \psi\,  \D m |
\le C_{\phi,  \psi} \kappa^k\, , \quad \forall k \ge 1 \, .
$$
In particular, if $\phi$ is supported in $I_t$,
\begin{equation}\label{trick}
\int \phi \, (\id-\LL_s) ^{-1} ( \psi )\, \D m = \sum_{k=0}^\infty \int
\phi   \LL_s^k (\psi)\,  \D m =
\sum_{k=0}^\infty  \int  (\phi \circ f_s^k)\,  \psi\,  \D m  \, .
\end{equation}
If $t$ also belongs to $\Omega$,   the fixed point property
 $\LL_\tau \rho_\tau =\rho_\tau$ for $\tau=s,t$ implies
$$
{\int \phi\, (\id-\LL_s) (\rho_s-\rho_{t})  \D m} =
\int \phi \, (\LL_s - \LL_{t}) \rho_{t} \,  \D m \, .
$$
Since  $\int (\LL_s - \LL_{t}) \rho_{t} \, \D m=0$ (using that
$f_t([c_{2,t},c_{1,t}])=[c_{2,t},c_{1,t}]\subset I_s$) if $|t-s|$ is small enough, we would like to multiply
the factor of $\phi$ in  both sides by $(\id-\LL_s)^{-1}$ to recover $\RR_\phi(s)-\RR_\phi(t)$ and
then attempt to implement the ``recipe'' in \S\ref{defMar} for the Marchaud derivative.  Writing
$(\id-z\LL_s)^{-1}
=\sum_{k=0}^\infty z^k \LL_s^k$, and using \eqref{trick},
this motivates our definition for the fractional susceptibility
function:

\begin{definition}[$\Omega$-Whitney--Marchaud  fractional
susceptibility function] 
For  $t \in \TSR$ and $\epsilon>0$
as in \eqref{extendd}, let $\Omega\subset \TSR$   have $t$ as
a Lebesgue density point.
For $\Re \eta \in (0,1)$, the (Whitney--Marchaud)  fractional
susceptibility function
$\Psi_\phi^\Omega (\eta, z) =\Psi^\Omega_{\phi, t,\epsilon} (\eta, z)$ (of the quadratic family, along
$\Omega$ at $t$, for the observable $\phi \in C^1$)
is the  formal power series in $z$
\begin{align}
\label{1}\Psi^\Omega_\phi (\eta, z):=\frac{\eta}{2 \Gamma(1-\eta)}
\sum_{k=0}^\infty  z^k  \int
\int_{\real \cap (\Omega-t)} &\phi ( f_{t+\delta}^k (x)) \cdot \\
\nonumber  & \cdot 
\frac{(\LL_{t+ \delta}-\LL_{t} )  \rho_{t}(x) }{|\delta|^{1+\eta}} \sgn(\delta) \, \D \delta \,
\D x  \, .
\end{align}
(The choice of $\epsilon$ implies that 
$x\mapsto (\LL_{t+ \delta}-\LL_{t} )  \rho_{t}(x)$  is supported in $ I_{t,\epsilon}\subset I_t$.)
\end{definition}

The coefficient of 
$z^k$ in  the  power series  \eqref{1} is a sum of improper integrals,
for $\delta\in (-\infty,0)$ and $\delta\in (0, \infty)$.
For each fixed $k\ge 1$, every $\delta$ such that $t+\delta\in \Omega$, and every $\psi_\delta \in L^1$ (and $\phi$) supported in $I_{t,\epsilon}$, we have, since $I_{t,\epsilon}\subset I_t$,
\begin{align}\label{termwise}
z^k \int_{I_{t}} (\phi \circ f_{t+\delta}^k) (x) \cdot 
\psi_\delta(x) \, \D x =
\int_{I_{t}} \phi(x) z^k (\LL_{t+\delta}^k \psi_\delta) (x)  \, \D x \, .
\end{align}
The  presence of $(\id - z\LL_{t+\delta})^{-1}$ in \eqref{termwise} is the reason we restrict the integral to good parameters
 $t+\delta\in \Omega$ (see also Appendix~\ref{wormell}).

\smallskip
In the present work, we  mostly study  the  {\it frozen
	fractional susceptibility function:}

\begin{definition}[Frozen susceptibility function]\label{defr}
Let $t$ be a TSR parameter and choose $\epsilon>0$
as in \eqref{extendd}. For $\eta\in (0,1)$
the frozen susceptibility  function $\Psi^\mathrm{fr}_\phi(\eta,z)=\Psi^\mathrm{fr}_{\phi,t,\epsilon}(\eta,z)$
(of the quadratic family, at $t$ for the observable
$\phi\in C^1$) is  the formal power series\footnote{\label{footd}Recalling  \eqref{extendd}, the function
$x\mapsto M^\eta_s (\LL_s \rho_{t}(x))|_{s=t}$ is supported in $I_{t,\epsilon}\subset I_t$.}
\begin{align}\label{is M}
\Psi^\mathrm{fr}_\phi(\eta,z)
&= \sum_{k=0}^\infty  z^k \int_{I_{t}}  
 (\phi \circ f_{t}^k) (x) 
M^\eta_s (\LL_s \rho_{t}(x))|_{s=t} \,
\D x\, ,
\end{align}
where $M^\eta_t$ is the two-sided Marchaud fractional derivative of order $\eta$, in the parameter $t$,  in the sense of distributions of order one (Definition~\ref{ddistr}). 
In other words, for fixed $\eta$, we have, as a formal power series in $z$,
\begin{align}
\nonumber \Psi^\mathrm{fr}_\phi(\eta,z)= \frac{\eta}{2\Gamma(1-\eta)} \sum_{k=0}^\infty  z^k &\int_{I_{t}}
 (\phi \circ f_{t}^k) (x) \\
\nonumber
& \cdot
\lim_{\epsilon \to 0}
\int_{|t|>\epsilon} 
\frac{\bigl ((\LL_{t+ \delta} -\LL_{t})  \rho_{t}\bigr )(x)}{|\delta|^{1+\eta}} \sgn(\delta) \, \D \delta \,
\D x\, ,
\end{align}
where the integral over $\D t$ is viewed as a distribution of order one.
\end{definition}

Applying \eqref{termwise} to each term of \eqref{is M}, we find (formally)
\begin{align*}
\Psi^\mathrm{fr}_\phi(\eta,z)
&=  \int_{I_{t}} 
\phi \, (\id -z \LL_{t}) ^{-1}
(M^\eta_s (\LL_s \rho_{t}(x))|_{s=t}) \,
\D x\, ,
\end{align*}

In Section~\ref{rigorr}, we shall prove Theorem~\ref{WTF}
on the frozen susceptibility function for $\eta=1/2$ and Misiurewicz--Thurston
parameters $t$.

\smallskip

Formulas  for fractional response are not as neat
as for linear response, since the usual Leibniz  and chain rules are
replaced by infinite expansions in the case of fractional derivatives. (See Eq.~2.209 in Section 2.7.3
of \cite{Pod}
for the chain rule.
For the Leibniz formula, see \S 15 in \cite{sam}.)
However, we shall see in Proposition~\ref{CC2} that  a simplification occurs for the frozen susceptibility function.
This motivates the definition of a
{\it response fractional
susceptibility function:} 

\begin{definition}[Response susceptibility function]\label{defrsp}
For $\eta \in (0,1)$ and $\phi \in C^1$ is
compactly  supported,  the response susceptibility function is defined by the following formal
power series
\begin{align}\nonumber
\Psi^{\mathrm{rsp}}_\phi(\eta,z)
&
:=\sum_{k=0}^\infty  z^k \int_{I_{t}}  M^\eta_x (\phi \circ f_{t}^k)
\cdot X_{t} \rho_{t} \,\D m=
\sum_{k=0}^\infty  z^k \int_{I_{t}}  M^\eta_x (\phi \circ f_{t}^k)
\cdot  \rho_{t} \,\D m \, .
\end{align}
\end{definition}

If $\eta \in (0,1/2)$, then
\begin{equation}
\label{identity}
\Psi^{\mathrm{rsp}}_\phi(\eta,z)
= 
-\sum_{k=0}^\infty  z^k \int  (\phi \circ f_{t}^k)
\cdot
M^\eta_x  \bigl ( \rho_{t}  \bigr )
\, \D x  
\end{equation}
follows from integration by parts for the Marchaud 
derivative\footnote{Use that $\phi$ 
and $\rho_{t}$ are compactly supported while, on the one hand, we have
$M^\eta_x (\phi \circ f_{t}^k)\in L^{p}_{loc}$
for all $p\ge 1$,
while $\phi \circ f_{t}^k\in L^s$ 
for all $s\ge 1$, and, 
on the other hand, we have $M^\eta (\rho_{t}) \in L^r_{loc}$ for \cite{Sedro} any
 $1\le r <2(1+2\eta)^{-1}$,
while 
$\rho_{t}\in L^{\tilde r}$ for all 
$1\le \tilde r <2$.}
\cite[(6.27)]{sam}.
We will see in Lemma~\ref{identity1} that  
\eqref{identity} in fact holds for all $\eta \in (0,1)$, 
up to taking the
Marchaud derivative of $\rho_{t}$   in the sense of distributions.

In the limit as $\eta\to 1$
the following easy lemma shows that
the response susceptibility function converges to the
Ruelle susceptibility function:

\begin{lemma}[Ruelle susceptibility as a limit of  response susceptibilities]
\label{limitok} Fix $t\in \SS$ and a compactly supported
$\phi \in C^1$, and let $\Psi_\phi(z)$
be Ruelle's susceptibility function \eqref{Rudef}.
Then, as formal power series in $z$,
$$
 \lim_{\eta \uparrow 1} \Psi^{\mathrm{rsp}}_\phi(\eta, z)=\Psi_\phi(z)
\, .
$$
\end{lemma}

\noindent The proof of  Lemma~\ref{limitok} does not use that $f_{t+\tau}(x)=f_{t}(x)+\tau$.

\begin{proof}[Proof of Lemma~\ref{limitok}]
Apply
 $\lim _{\eta \to 1} M^\eta g=g'$  (e.g.  \cite{ABLP}) to $g=\phi\circ f_{t}^k\in C^1$.
\end{proof}

Finally,  using Lemma~\ref{identity1}, we give the easy proof
of the following remarkable result
in \S\ref{dealing}
(the identity \eqref{magic}  greatly simplifies the proof of our main result
on the frozen susceptibility function, Theorem~\ref{WTF},
for more general smooth unimodal maps it seems there is no way to
bypass the  study of $M^{1/2}_s( \LL_s \rho_{t})$):

\begin{proposition}[Relating the frozen and response susceptibility functions]\label{CC2}
For any\footnote{\label{foot0}The proof shows that the proposition holds more generally, for example for mixing TSR parameters.}
mixing $t\in \MT$ and  $\eta \in (0,1)$,  we have, as distributions of order one,
\begin{equation}\label{Prop2.5}
M^\eta _s(\LL_s\rho_{t}(x))|_{s=t}=-
M^\eta_x \rho_{t}(x)+\frac{g_{\eta}(x)}{\Gamma(1-\eta)}
 \, ,
\end{equation}
where $g_\eta\in H^r_q$ for some $r>0$ and $q>1$, 
with $\sup_{\eta>\epsilon_1} \|g_\eta\|_{ H^r_q }<\infty$ 
for any fixed $\epsilon_1>0$, and
$
\int_\real g_\eta(x) \D m =0 
$.

In addition, there exists $\kappa<1$ and for any 
compactly supported $\phi\in C^1$,  there exists  $\VV^{\mathrm{rsp}}_{\phi,\eta}(z)=\sum_{j\ge 0} v_j z^j$ holomorphic in  the disc of radius
$\kappa^{-1}$ such that
\begin{equation}\label{magic}\Psi^{\mathrm{fr}}_\phi(\eta, z)-\VV^{\mathrm{rsp}}_{\phi,\eta}(z)=
\Psi^{\mathrm{rsp}}_\phi(\eta, z)\mbox{ as formal
power series in $z$.}
\end{equation}
Finally, we have, as formal power series,
$
\lim_{\eta \uparrow 1} \Psi^{\mathrm{fr}}_\phi(\eta, z)=\Psi_\phi(z)
$.
\end{proposition}

Proposition~\ref{CC2} and Lemma~\ref{limitok} imply 
Proposition~\ref{propE}: both  the response 
and the frozen fractional susceptibility functions converge to the
Ruelle susceptibility function as $\eta \to 1$. (However $\Psi^\mathrm{rsp}_\phi$ and $\Psi^\mathrm{fr}_\phi$
do not satisfy  [ii] from Conjecture~\ref{laconj}.)

\section{Half integrals of square root spikes}
\label{Abel}

After  recalling the definitions of Riemann--Liouville fractional integrals, we revisit
Abel's computation of the one-sided half-integral of a square root spike
and extend it to the two-sided  half-integral. The corresponding statements,
 Lemma~\ref{halff+} and Lemma~\ref{halff}, will be used in Section~\ref{mainsec}
to compute Marchaud derivatives.

\subsection{Riesz potentials and Riemann--Liouville fractional integrals}

For any  $\phi \in L^1$ and for $\eta \in (0,1)$, 
 the Riesz potential fractional integral is defined for
$\Re \eta >0$, $\eta \ne 1, 3, 5, \ldots$ by (see \cite[(5.2)--(5.3), \S 12.1]{sam})
\begin{equation}\label{RLhalf}
I^\eta \phi(t)
=\frac{1}{2 \Gamma(\eta)\cos (\eta \pi/2)}
\int_{-\infty}^\infty
\frac {\phi(\tau)}{|t-\tau|^{1- \eta}} \, \D \tau= \frac{I^\eta_+\phi(t) + I^\eta_- \phi (t)}{2 \cos (\eta \pi/2)}\, ,
\end{equation}
where $I^\eta_\pm$ are the  left- and right-sided Riemann--Liouville fractional integrals \cite[(5.2)--(5.3)]{sam} (there is a typo in the 
second line of 
\cite[(5.4)]{sam})
\begin{align*}
I^\eta_+ \phi(t)
=\frac{1}{ \Gamma(\eta)}
\int_{-\infty}^t
\frac {\phi(\tau)}{(t-\tau)^{1- \eta}}  \, \D \tau 
=\frac{1}{ \Gamma(\eta)}\int_0^\infty \frac{\phi(t-y)}{y^{1-\eta}} \, \D y\, , 
\\
I^\eta_- \phi(t)
=\frac{1}{ \Gamma(\eta)}
\int_{t}^\infty
\frac {\phi(\tau)}{(\tau-t)^{1- \eta}}  d \tau =
\frac{1}{ \Gamma(\eta)}\int_0^\infty \frac{\phi(t+y)}{y^{1-\eta}} \, \D y\, .
\end{align*}

If $g_t(x)$ is a function of two variables $x$ and $t$,  we write $(I^\eta_{t} g_t)(x)$  to denote the fractional integral acting on the parameter $t$
and evaluated at $x$ 
 and $t$, and similarly for the one-sided integrals $I^\eta_{-,t}$ and $I^\eta_ {+,t}$.

Note for further use that, setting $Q\phi (t)=\phi (-t)$,
$T_a\phi(t)=\phi(t+a)$, we have 
\begin{equation}
\label{QQ+}
I^{\eta}_{\sigma} \circ Q= Q \circ I^{\eta}_{-\sigma}\, ,
\,\,\, \, 
 I^{\eta}_\sigma \circ T_a= T_a \circ I^{\eta}_\sigma\, .
\end{equation}

The case which will interest us most is $\eta=1/2$, that is,
``half Riesz potential  integrals''
or ``half Riemann--Liouville  integrals.''
In \S\ref{half+}, we recall the proof of  a key
observation of Abel regarding the ordinary one-sided half-Riemann--Liouville integral of square root
spikes, and we present its two-sided version, Lemma~\ref{halff}.
(Lemma~\ref{halff} will be a key ingredient to prove
our main result in Section~\ref{rigorr}.)

\subsection{Abel's remark: One-sided  half integration of square-root spikes}
\label{half+}

In this section, we recall a result of Abel on one-sided half integrals
 (Lemma~\ref{halff+}) and extend it to two-sided half integrals (Lemma~\ref{halff}). The corresponding results will be used to prove
Lemma~\ref{AbelMarchaud} below about the half Marchaud derivative of a spike.

The following fact was  probably first observed by Abel \cite{Ab0,Ab} (see also \cite{Rag}):

\begin{lemma}[Abel's remark]\label{halff+}
Fix $k\ge 1$ and $\sigma\in \{-1, +1\}$.
Consider the left and right  square-root spikes (in $x$) at $c_k+t$
\begin{equation}\label{lrspike}
\phi_{c_k,\sigma}(x,t)=(|x-c_k-t|)^{-1/2} \mathbf{1}_{\sigma x > \sigma (c_k+t)}\, ,
\quad x, t\in \real \, .
\end{equation}
Then the one-sided Riemann--Liouville half integrals
$I^{1/2}_{\pm} (\phi_{c_k,\mp})$   (with respect to $t$)  
are the following Heaviside jumps (in $x$) at
$c_k+t$:
$$
I^{1/2}_{-,t} (\phi_{c_k,+}) (x,t)= \sqrt \pi \cdot 
\mathbf{1}_{x > c_k+t}(x)\, ,\,\,
I^{1/2}_{+,t} (\phi_{c_k,- })(x,t)= \sqrt \pi \cdot
\mathbf{1}_{x < c_k+t}(x)\, .
$$
\end{lemma}

\begin{proof}[Proof of Lemma~\ref{halff+}] 
The half  integral $I^{1/2,t}_{-}$ of  $\phi_{c_k,+}(x,t)$ with respect to $t$ is 
\begin{align*} 
 (I^{1/2,t}_{-}\phi_{c_k,+})(x)&= 
\frac{1}{\Gamma(1/2)} 
\int_{t}^{+\infty}  \frac{\phi_{c_k,+}(x,\tau)}{ (\tau -t)^{1/2} }\, \D \tau\\
&= \frac{1}{\Gamma(1/2)}
 \int_{t}^{+\infty} \frac{(x-{c_k}-\tau)^{-1/2} \mathbf{1}_{ y >  {c_k}+\tau}(x) }
 {(\tau -t)^{1/2}} \, \D \tau\\ 
&=\begin{cases} 0  &\text{  if }    c_k +t  \geq x,   \\  
\frac{1}{\Gamma(1/2)} \int_{t}^{x-{c_k}}  \frac{1}{((\tau-t)(x-{c_k}-\tau))^{1/2}}  \, \D \tau 
&\text{  if }  {c_k}+t<x \, .   \end{cases}
\end{align*} 
If $c_k+t <x$, making the substitution $\tau = t+(x-c_k-t)u$, we get
\begin{align} \label{p14}
\int_{t}^{x-{c_k}} \frac{1}{((\tau-t)(x-{c_k}-\tau))^{1/2}}  \, \D \tau 
 =  \int^1_0  \frac{1}{(u(1-u))^{1/2}} \, \D u=B(1/2,1/2)\, .
\end{align} 
Recalling 
$B(1/2,1/2)=\pi$ and $\Gamma(1/2)=\sqrt \pi$, we find 
$$
(I^{1/2}_{-,t}\phi_{c_k,+})(x,t) = 
\begin{cases} 0  &\text{  if }    c_k +t \geq {x} \, ,  
 \\  \sqrt \pi&\text{  if } c_k +t < {x} \, .  \end{cases}.
$$

The other claim follows from \eqref{QQ+} since
$$
\phi_{c_k,-} (x,t)= 
\phi_{c_k,+} (x,2(x-c_k)-t)= Q \circ T_{2 (x-c_k)} (\phi_{c_k,+} ) (x,t)\, .
$$
Indeed,  we find
\begin{align*}
\nonumber I^{1/2}_{+,t} \phi_{c_k,-} (x,t)
&= I^{1/2}_{+,t} \circ Q \circ T_{2 (x-c_k)} (\phi_{c_k,+} ) (x,t)\\
\nonumber
&=  I^{1/2}_{-,t} \circ T_{2 (x-c_k)}(\phi_{c_k,+}) (x,-t)
= 
I^{1/2}_{-,t} \phi_{c_k,+} (x,-t+2(x-c_k)) \, .
\end{align*}
Finally, $x> c_k -t+2(x-c_k)$ if and only if $x<c_k+t$.
\end{proof}

Replacing the one-sided Riemann--Liouville fractional integral $I^\eta_\pm$
by the (two-sided) Riesz potential  $I^\eta$ from  \eqref{RLhalf}, Lemma~\ref{halff+} must be replaced by the following lemma, which includes
 an unbounded logarithm
corresponding to the ``other side.''

\begin{lemma}[Two-sided version of Abel's remark]\label{halff}
For any real number  $\EE> 1$,
any integer $k\ge 1$ and
any $x\in I$, the one-sided Riemann--Liouville half integrals
  of the $\EE$-truncated right and left square-root spikes
\begin{equation}\label{cutoff}
\phi_{c_k,+,\EE} (x,t) =
\frac{\mathbf{1}_{(c_k+t,c_k+t+\EE)}(x)}
{ (x-c_k-t)^{1/2} }\, ,  \, \, \phi_{c_k,-,\EE} (x,t) =
\frac{\mathbf{1}_{(c_k+t-\EE, c_k+t)}(x)}
{ (c_k+t-x)^{1/2} }\, ,
\end{equation}
satisfy, for $\sigma=\pm $ and any  $|x-c_k-t |< \EE/2$,
\begin{align*}
&I^{1/2}_{\sigma} (\phi_{c_k,\sigma,\EE}) (x,t)=
\frac{\sigma}{\sqrt{\pi}}\bigl (
-
 \log |x- c_k-t| +\log \EE+
G_{\EE}(\sigma(t-x+c_k))\bigr )\,  ,
\end{align*}
where $G_{\EE}(y)$ is analytic on $|y |< \EE/2$,
with $\lim_{\EE\to \infty}\sup_{ |y|<\EE/2}
| \partial_y G_\EE(y)|=0$, and
$$
\sup_{\EE> 1} \,\, \sup_{| y|< \EE/2}
\max\{|  G_{\EE}(y)|,
| \partial_y G_{\EE}(y)|, | \partial^2_y G_{\EE}(y)|\} < \infty\, .
$$
\end{lemma}

The elementary proof of the above crucial lemma (which will be used to
prove Lemmas~\ref{AbelMarchaud} and ~\ref{AbelMarchaud2}) is given in Appendix~\ref{Ahalff}.

Finally, the remark below will be used several times in the sequel:

\begin{remark}[Phase and parameter half-integrals of a spike]
\label{exchangext}
Since $x > c_k+t-u$ if and only if $t<x+u-c_k$, we have for any $1<\EE\le \infty$, recalling \eqref{cutoff},
\begin{align*}
I^{1/2}_{-,t} (\phi_{c_k,+,\EE}) (x,t)
&=I^{1/2}_{+,x} (\phi_{c_k,+,\EE}) (x,t)\, , \, \, 
 I^{1/2}_{+,t} (\phi_{c_k,-,\EE}) (x,t)=I^{1/2}_{-,x} (\phi_{c_k,-,\EE}) (x,t)\, ,
\end{align*}
and  for any $1<\EE < \infty$
\begin{align*}
 I^{1/2}_{+,t} (\phi_{c_k,+,\EE}) (x,t)&=I^{1/2}_{-,x} (\phi_{c_k,+,\EE}) (x,t)\, , \,\,
 I^{1/2}_{-,t} (\phi_{c_k,-,\EE}) (x,t)=I^{1/2}_{+,x} (\phi_{c_k,-,\EE}) (x,t)\, .
\end{align*}
\end{remark}

\section{Marchaud derivatives applied to spikes and square roots}
\label{mainsec}

After recalling the definition of Marchaud derivatives $M^\eta$ and extending them
as distributions in \S\ref{defMar}, we show 
in \S\ref{applM} how $M^{1/2}$ acts on the singular components
(spikes and square roots) of the invariant density $\rho_t$. 
The lemmas in this section will be crucial to prove Theorem~\ref{WTF} in
Section~\ref{rigorr}. 

\subsection{One-sided and two-sided Marchaud derivatives
$M^\eta_\pm$ and  $M^\eta$}\label{defMar}
Let $g:\real \to \complex$ be bounded and $\gamma$-H\"older.
We recall that the left-sided Marchaud fractional derivative (with lower limit 
$a=-\infty$)  \cite[pp. 110--111, Theorem 5.9, p. 225]{sam}, where it is denoted
by $\mathbf{D}^\eta_+$, see also \cite[\S 2.2.2.3]{hilf} is defined for 
$\eta \in (0,\gamma)$ and $x \in \real$, by 
\begin{align}
\nonumber(M^\eta_{+} g) (x)
=\frac{\eta}{ \Gamma(1-\eta)}
\int_{-\infty}^{x}
\frac {g(x)-g(y)}{(x-y)^{1+ \eta}}  \, \D y
=\frac{\eta}{ \Gamma(1-\eta)}
\int_{-\infty}^{0}
\frac {g(x)-g(x+\tau)}{|\tau|^{1+ \eta}}  \, \D \tau\\
\label{above} =\frac{\eta}{ \Gamma(1-\eta)}
\int_{0}^{\infty}
\frac {g(x)-g(x-\tau)}{\tau^{1+ \eta}}  \, \D \tau\, .
\end{align}
If $g$ is bounded on $\real$ and differentiable\footnote{If $g$ is bounded  and differentiable to the left at $x$,  the limit as $\eta \uparrow 1$ of $M^\eta_{+}(g)(x)$ is  equal to the left-sided derivative $g'_-(x)$, the notation is thus confusing.} at $x$,  the limit as $\eta \uparrow 1$ of $M^\eta_{+}(g)(x)$ is  equal to the ordinary derivative $g'(x)$
(see e.g. \cite[\S 3.2]{OM} or \cite{ABLP}).

The integral \eqref{above} is an improper integral. In the application of this paper, $g(t)$ will be
bounded as $t \to \pm\infty$, so\footnote{This is an advantage of  Marchaud derivatives over Riemann--Liouville
fractional derivatives.} the
only delicate limit is $\tau\to 0$. Concretely, we will work with the expression
(see \cite[(5.59--5.60)]{sam})
\begin{equation*}
(M^\eta_{+} g)(x)
=\lim_{\epsilon \uparrow 0} (M^\eta_{+,\epsilon} g)(x)
:= \lim_{\epsilon \uparrow 0} \frac{\eta}{ \Gamma(1-\eta)}
\int_{-\infty}^{\epsilon}
\frac {g(x)-g(x+\tau)}{|\tau|^{1+ \eta}}  \, \D \tau\, .
\end{equation*}

The right-sided Marchaud fractional derivative (with upper limit 
$b=+\infty$) is defined for $\eta \in (0,1)$ and $x \in \real$ by
\begin{align*}
M^\eta_{-} g(x) &=\frac{\eta}{ \Gamma(1-\eta)}
\int_0^{\infty}
\frac {g(x)-g(x+\tau)}{\tau^{1+ \eta}}  \, \D \tau\\
&=\lim_{\epsilon \downarrow 0 } (M^\eta_{-,\epsilon} g)(x)=
\lim_{\epsilon \downarrow 0 }\frac{\eta}{ \Gamma(1-\eta)}
\int_\epsilon^{\infty}
\frac {g(x)-g(x+\tau)}{\tau^{1+ \eta}}  \, \D \tau\, .
\end{align*}
If $g$ is bounded on $\real$ and differentiable at $x$ (differentiable to the right is enough),  then  $\lim_{\eta \uparrow 1}M^\eta_{-}(g)(x)=-g'(x)$
(see e.g.  \cite{ABLP}).

We define the two-sided Marchaud derivative by
\begin{equation*}
M^\eta g(x) =\frac{M^\eta_{+} g(x)-
	M^\eta_{-} g(x)}{2} \, .
\end{equation*}
Note that $M^\eta g(x)=\lim_{\epsilon \downarrow 0}M^\eta_\epsilon g(x)$
where
\begin{equation}\label{Meps}
M^\eta_\epsilon g(x)=\frac{\eta}{2 \Gamma(1-\eta)}
\int_{|\tau|> \epsilon}\frac {g(x+\tau)-g(x)}{|\tau|^{1+ \eta}} \sgn(\tau)  \, \D \tau
\, .
\end{equation}

Note for further use that, recalling $Q g (t)=g (-t)$,
$T_ag (t)=g(t+a)$, we have
\begin{equation}
\label{QQM+}
M^{\eta}_{\sigma} \circ Q= Q \circ M^{\eta}_{-\sigma}\, ,
\,\,\, \, 
 M^{\eta}_\sigma \circ T_a= T_a \circ M^{\eta}_\sigma\, \, , \, \sigma=\pm\, .
\end{equation}
Therefore,
\begin{equation}
\label{QQM}
M^{\eta} \circ Q= -Q \circ M^{\eta}\, ,
\,\,\, \, 
 M^{\eta} \circ T_a= T_a \circ M^{\eta}\, .
\end{equation}

\smallskip

We shall sometimes need to consider $M^\eta g$ (if $g$ is  not H\"older, for example) in the {\it sense of distributions (of order one):}

\begin{definition}[Marchaud derivative in the sense of distributions
of order one] \label{ddistr}
For  $\eta \in (0,1)$ and a measurable function $g$ 
such that the integral
 $G(y)=\int_{-\infty}^y g(u) \, \D u$ is well-defined
and almost everywhere finite, with\footnote{One could weaken this condition, up to exchanging the limit and the derivative in 
\eqref{Mardistr}. We shall not need this more general notion.} 
$$\lim_{\epsilon \to 0}  
M^\eta_\epsilon G(x)\in L^1_{loc}\, , 
$$
we define the two-sided Marchaud derivative
of $g$  in the sense of distributions
of order one by 
setting, for  any compactly supported
$C^1$ function $\psi$,
\begin{align}
&\int (M^\eta g)(x) \psi(x) \, \D x:=-\int \bigl  [\lim_{\epsilon \to 0}  
M^\eta_\epsilon G(x) \bigr ]  \, 
\psi'(x) \, \D x \, .
\label{Mardistr}
\end{align}
The one-sided Marchaud derivatives $M^\eta_-$ and $M^\eta_+$ in the sense of distributions are defined analogously (for $M^\eta_-$, it is convenient to set $G(y)=-\int^{\infty}_y g(u) \, \D u$.).
\end{definition}

Note that \eqref{QQM} and \eqref{QQM+} extend to the setting of
Definition~\ref{ddistr}.

\smallskip
If $g_t(x)$ is a function of two variables $x$ and $t$,  then $(M^\eta_{t} g_t)(x)$ or $(M^\eta_{s} g_s)(x)|_{s=t}$ denote the Marchaud derivative acting on the parameter $t$
and evaluated at $x$ and $t$, and similarly for the one-sided derivatives $M^\eta_{-,t}$ and $M^\eta_ {+,t}$.

\begin{remark}[Marchaud in the sense of distributions]
If $g \in C^1$
is compactly suppported then the definition \eqref{Mardistr}
is in fact an identity which can be deduced from Fubini, Lebesgue dominated convergence, and integration by parts
(for $C^1$ compactly supported $\psi$). Let us write the computation in the
one-sided case:
\begin{align*}
\int (M^\eta_{+} g)(x) \psi(x) \, \D x&=\int \biggl [ \lim_{\epsilon \uparrow 0} \frac{\eta}{ \Gamma(1-\eta)}
\int_{-\infty}^{\epsilon}
\frac {g(x)-g(x+\tau)}{|\tau|^{1+ \eta}}  \, \D \tau \biggr ] \,
\psi(x) \, \D x\\
&= 
\frac{\eta}{ \Gamma(1-\eta)}
\int_{-\infty}^{0} \int
\frac {g(x)-g(x+\tau)}{|\tau|^{1+ \eta}}   \,
\psi(x) \, \D x \,  \D \tau\\
&= -
\frac{\eta}{ \Gamma(1-\eta)}
\int_{-\infty}^{0} \int
\frac {G(x)-G(x+\tau)}{|\tau|^{1+ \eta}}    \,
\psi'(x) \, \D x \,  \D \tau\\
&=
-\int \biggl [ \lim_{\epsilon \uparrow 0} \frac{\eta}{ \Gamma(1-\eta)}
\int_{-\infty}^{\epsilon}
\frac {G(x)-G(x+\tau)}{|\tau|^{1+ \eta}}  \, \D \tau \biggr ] \,
\psi'(x) \, \D x \, .
\end{align*}
\end{remark}

\begin{remark}[Marchaud and Riemann--Liouville]\label{thesame}
If $g$ is $C^1$ and and $|g'(\tau)|=O(|\tau|^{\eta-1-\epsilon})$
for some $\epsilon>0$
as $\tau \to -\infty$
(\cite[pp. 109--110]{sam}) then  the left-sided Marchaud derivative of $g$ coincides with the left-sided Riemann--Liouville derivative
with lower limit $a=-\infty$ of $g$
$$
(M^\eta_{+} g) (t)=\frac{d}{dt} I^{1-\eta}_+(g)(t)=
\frac{1}{\Gamma(1-\eta)}\frac{d}{dt} \int_{-\infty}^t \frac{g(\tau)}{(t-\tau)^\eta} \, \D \tau\, .
$$
Similarly, the  right-sided Marchaud derivative of $g$ coincides with the right-sided Riemann--Liouville derivative
with upper limit $a=\infty$ of $g$
$$
(M^\eta_{-} g) (t)=-\frac{d}{dt} I^{1-\eta}_-(g)(t)=
\frac{-1}{\Gamma(1-\eta)}\frac{d}{dt} \int^{\infty}_t \frac{g(\tau)}{(\tau-t)^\eta} \, \D \tau\, .
$$

The remark above
will be used in the proof of Lemma~\ref{identity1}. (Note that Lemma~\ref{AbelMarchaud} is a generalisation of this remark, for $g$  a one-sided spike and $\eta=1/2$.)
\end{remark}

\subsection{The half derivative $M^{1/2}$ of  spikes, square roots, and $C^1$ functions}
\label{applM}

The key fact we shall use is the following lemma about 
 Marchaud derivatives
 (in the sense \eqref{Mardistr} 
of distributions) of  spikes and truncated spikes,   for  $\EE >1$, 
$$\phi_{x_0,\sigma }(x)=\frac{\mathbf{1}_{\sigma x> \sigma x_0} } {\sqrt{|x-x_0|}}
\, ,
\qquad
\phi_{x_0,\sigma,\EE}(x)= \mathbf{1}_{0<\sigma(x - x_0)<\EE} \cdot \phi_{x_0,\sigma}(x)\, .
$$

\begin{lemma}[Half Marchaud derivatives of a spike]\label{AbelMarchaud}  For $x_0\in \real$ and $\sigma=\pm$, the following holds:
 The one-sided half Marchaud derivatives
satisfy, as distributions  on continuous compactly supported functions, 
$$
M_+^{1/2}( \phi_{x_0,+})(x)= \sqrt{\pi}\cdot \delta_{x_0}\, ,
\quad M_-^{1/2}( \phi_{x_0,-})(x)= {\sqrt{\pi}}\cdot \delta_{x_0}\, .
$$
The two-sided half Marchaud derivative  satisfies,
as a distribution  on $C^1$ compactly supported functions,
$$M^{1/2}( \phi_{x_0,\sigma})(x)
=\frac{\sigma}{2\sqrt \pi} \cdot 
\biggl ( \pi \delta_{x_0}
 - \frac{1}{x-x_0} \biggr ) \, .
$$
Finally,  for any $\EE>1$,
 the two-sided half Marchaud derivative  satisfies,
as a distribution  on $C^1$
 functions supported in $[x_0-\EE/2,x_0+\EE/2]$,
$$M^{1/2}( \phi_{x_0,\sigma,\EE})(x)
=\frac{\sigma }{2\sqrt\pi }  \cdot\biggl ( 
\pi \delta_{x_0}
-
  \frac{1}{x-x_0} +\Phi_{\EE}(\sigma(x_0-x)) \biggr )\, ,
$$
where $\Phi_{\EE}(y)$ is analytic on $|y |< \EE/2$,
with $\lim_{\EE\to \infty}\sup_{ |y|<\EE/2}
| \Phi_\EE(y)|=0$, and
$$
\sup_{\EE> 1} \,\, \sup_{| y|< \EE/2}
\max\{
| \Phi_{\EE}(y)|, | \partial_y \Phi_{\EE}(y)|\} < \infty\, .
$$
\end{lemma}

In view of the expansion \eqref{Ruellerho} for the invariant density, we also need  Marchaud derivatives of  square roots
and truncated square roots, defined for $\EE>1$ by,
$$
\bar \phi_{x_0,+,\EE}(x)= \mathbf{1}_{x_0<x < x_0+\EE} \cdot
\sqrt{x-x_0 }\, ,\, \,\,
\bar \phi_{x_0,-,\EE}(x)= \mathbf{1}_{x_0-\EE<x < x_0} \cdot
\sqrt{x_0-x }\, .
$$

\begin{lemma}[Half Marchaud derivatives of a square root]\label{AbelMarchaud2} Let  $x_0\in \real$.
The one-sided Marchaud derivatives of square roots satisfy, for $\sigma=\pm$,
$$
M_{\sigma}^{1/2}\bigl ({\mathbf{1}_{\sigma x> \sigma x_0}}({\sqrt{|x-x_0|}})\bigr )= 
\frac{\sqrt \pi}{2} \mathbf{1}_{ \sigma x > \sigma x_0}(x) 
\, .
$$

For $\EE >1$, the two-sided Marchaud derivatives of truncated square roots
 satisfy 
$$
M^{1/2}(\bar \phi_{x_0,\sigma,\EE})(x)
=\frac{\sigma }{2\sqrt\pi } 
\bigl ( \pi \mathbf{1}_{ \sigma x > \sigma x_0}(x)
 -  \log|x-x_0| +\log \EE \bigr )+ \bar \Phi_{\EE}(\sigma(x_0-x) ) \, ,
$$
where $\bar \Phi_{\EE}(y)$ is analytic on $|y |< \EE/2$,
with $\lim_{\EE\to \infty}\sup_{ |y|<\EE/2}
| \bar \Phi_\EE(y)|=0$, and
$$
\sup_{\EE> 1} \,\, \sup_{| y|< \EE/2}
\max\{
| \bar \Phi_{\EE}(y)|, | \partial_y \bar \Phi_{\EE}(y)|\} < \infty\, .
$$ 
\end{lemma}

\begin{lemma}[Action of Marchaud derivatives on $C^1$ functions]\label{AbelMarchaud3} 
For any $\eta \in (0,1)$ and any $C^1$ function $g:\real \to \real$
with $\sup_\real |g'|<\infty$, the two-sided Marchaud derivative $M^{\eta}(g)$ 
is $(1-\eta)$-H\"older. 
\end{lemma}

\begin{proof}[Proof of Lemma~\ref{AbelMarchaud}]
To show the claim on $M_+^{1/2}( \phi_{x_0,+})(x)$,
we must show that, for any $C^1$ function
$\psi$, compactly supported on a bounded interval $J$, we have
\begin{equation}\label{eqq}
\int_J \psi(x)  M_+^{1/2} \biggl (\frac {\mathbf{1}_{x > x_0}}{\sqrt{x-x_0}}\biggr ) \, \D x
=\sqrt \pi \psi(x_0)\, .
\end{equation}	

We shall use two facts.
On the one hand,  the distributional derivative of the Heaviside $\mathbf{1}_{x > y}$ is the Dirac mass at
$y$, in particular, for any compactly supported $C^1$ function $\psi$, and any bounded (interval $[a,b]$ containing $y$, we have
\begin{equation}\label{Dirr}
\int_y^b  \psi'(x) \, \D t=\int_a^b  \mathbf{1}_{x > y}(x)\psi'(x) \, \D t= -\psi(y)+ \psi(b) \, .
\end{equation}
On the other hand,  in view of Remark~\ref{exchangext},
Lemma~\ref{halff+} gives 
\begin{equation}
	\label{starting}
	I^{1/2}_{+,x} (\phi_{x_0,+} )(x)=
	\frac{1}{\Gamma(1/2)}
	\int_{-\infty}^0 \frac{\phi_{x_0,+}(x+\tau)}{|\tau|^{1/2}} \, \D\tau
	=\sqrt \pi  \mathbf{1}_{ x > x_0}(x)\, .
\end{equation}
We now move on to prove \eqref{eqq}. We have,  recalling \eqref{Mardistr}	
(in other words, integrating by parts with respect to $x$ using Fubini,
before taking the limit
$\epsilon\to 0$),
\begin{align*}
	\int_J  \psi(x) M_{+,x}^{1/2}& \phi_{x_0,+}(x)  \, \D x\\
	&=-\frac{1}{2\sqrt \pi} \int_J\psi'(x)
\lim_{\epsilon \uparrow 0} \int^\epsilon_{-\infty} \frac{\tilde \phi_{x_0,+}(x)- \tilde \phi_{x_0,+}(x+\tau)}{|\tau|^{3/2}} \, \D \tau \, \D x\, ,
\end{align*}
where $\tilde \phi_{x_0,+}(x)=0$ if $x_0>x$ and, otherwise,
\begin{align}\label{otherwise}
	\tilde \phi_{x_0,+}(x)&= \int_{-\infty}^x  \phi_{x_0,+} (y)\, \D y 
= \int_{x_0}^x \frac{1}{\sqrt{y-x_0}} \, \D y=
	2 \sqrt {x -x_0} \, .
\end{align}
Next, for $x_0-x<\epsilon <0$, integrating by parts,
we find,
\begin{align}
\nonumber \int^\epsilon_{-\infty}& \frac{\tilde \phi_{x_0,+}(x)- \tilde \phi_{x_0,+}(x+\tau)}{2|\tau|^{3/2}} \, \D \tau \\
\nonumber&=  2  
\biggl [\int^{x_0-x}_{-\infty} \frac{\sqrt {x -x_0}}{2|\tau|^{3/2}}\, \D \tau
+
	\int^\epsilon_{x_0-x}\frac{\sqrt {x -x_0}- \sqrt {x+\tau -x_0}}{2|\tau|^{3/2}} \, \D \tau \biggr ]\\
\nonumber&=   2 
\biggl [\frac{\sqrt {x-x_0}}{\sqrt{|\tau|}}\bigg |_{\tau=-\infty}^{\tau=x_0-x}
	+ \int^\epsilon_{x_0-x} \frac{ 1}
	{2\sqrt {x+\tau -x_0}}\frac{1}{ |\tau|^{1/2}} \, \D \tau \\
\nonumber&\qquad\qquad\qquad\qquad\qquad\qquad\qquad\qquad+\frac{\sqrt{x-x_0}-\sqrt{x-x_0+\epsilon}}{|\epsilon|^{1/2}} -1
 \biggr ]\\
\nonumber&=   2   
\biggl [1
	+ \int^\epsilon_{-\infty} \frac{\mathbf{1}_{x+\tau>x_0} (x) }
	{2\sqrt {x+\tau -x_0}}\frac{1}{ |\tau|^{1/2}} \, \D \tau +\frac{\sqrt{x-x_0}-\sqrt{x-x_0+\epsilon}}{|\epsilon|^{1/2}}- 1 \biggr ]\\
\nonumber&=
\int^\epsilon_{-\infty} \frac{\mathbf{1}_{x+\tau>x_0} (x) }
	{\sqrt {x+\tau -x_0}}\frac{1}{ |\tau|^{1/2}} \, \D \tau
+2\frac{\sqrt{x-x_0}-\sqrt{x-x_0+\epsilon}}{|\epsilon|^{1/2}}\\
\label{forlater}& =\int^\epsilon_{-\infty} 
\frac{\phi_{x_0,+}(x+\tau)}{ |\tau|^{1/2}} \, \D \tau  +2
\frac{\sqrt{x-x_0}-\sqrt{x-x_0+\epsilon}}{|\epsilon|^{1/2}}\, .
\end{align}
Note that $\lim_{\epsilon\uparrow 0} \frac{\sqrt{x-x_0}-\sqrt{x-x_0+\epsilon}}{|\epsilon|^{1/2}} =0$
for any fixed $x> x_0$. 
Using first 
\eqref{starting},  and then \eqref{Dirr} (recalling that $\psi$ vanishes at the endpoints of $J$), 
we have,
\begin{align*}
	-\frac{1}{\sqrt\pi} &\int_J  \psi'(x)
	\lim_{\epsilon \uparrow 0} \int^\epsilon_{-\infty} \frac{\phi_{x_0,+}(x+\tau)}{ |\tau|^{1/2}}\, \D \tau  \, \D x=
	- \int_J \psi'(x) I^{1/2}_+(\phi_{x_0,+})(x)  \, \D x\\
&\qquad\qquad= -\frac{\pi}{\sqrt\pi}  \int_{J \cap[x_0,\infty)}
	 \psi'(x) \, \D x= \sqrt\pi  \cdot \psi(x_0) \, , 
\end{align*}
which concludes\footnote{In particular we have shown that $M^{1/2}_+(\phi_{x_0,+})= 
d I^{1/2}_+(\phi_{x_0,+})$, as  expected, see Remark~\ref{thesame}.}  the proof of \eqref{eqq} for $M_+^{1/2}(\phi_{x_0,+})$.

For the claim\footnote{In particular, $M^{1/2}_-(\phi_{x_0,-})= 
-d I^{1/2}_-(\phi_{x_0,-})$, as expected, see Remark~\ref{thesame}.} on $M_-^{1/2}(\phi_{x_0,-})$, we use  \eqref{QQM+} and
\begin{equation}
\label{QQM!}
\phi_{x_0,+}(x)=\phi_{x_0,-}(2x_0-x)\, .
\end{equation}

Next, we show the claim on the two-sided  Marchaud derivative
$M^{1/2}( \phi_{x_0,+})(x)$.  We will apply Lemma~\ref{halff}.
We first claim that $ M_{-}^{1/2} (\phi_{x_0,+}-\phi_{x_0,+,\EE})(x)$
is $C^1$ (in fact, $C^\infty$) on $x< x_0+\EE$.
Indeed
\begin{align*}
2\Gamma(1/2) \cdot M^{1/2}_{-} &(\phi_{x_0,+}-\phi_{x_0,+,\EE})(x)\\ &=
\int_0^{\infty}
\frac {\mathbf{1}_{x > x_0+\EE} \cdot \phi_{x_0,+}(x)-
\mathbf{1}_{x+\tau > x_0+\EE} \cdot \phi_{x_0,+}(x+\tau)}{\tau^{3/2}}  \, \D \tau\, .
\end{align*}
If $x<x_0+\EE$, we find
\begin{align}
\nonumber M^{1/2}_{-} (\phi_{x_0,+}&-\phi_{x_0,+,\EE})(x) =-\frac{1}{2\Gamma(1/2)}
\int_{x_0-x+\EE}^{\infty}
\frac { \phi_{x_0,+}(x+\tau)}{\tau^{3/2}}  \, \D \tau\\
\label{recalling} &=-\frac{1}{2\Gamma(1/2)}\int_{x_0-x+\EE}^{\infty}
\frac {1}{\sqrt{x+\tau-x_0} }\frac{1}{ \tau^{3/2}} \, \D \tau 
=: \widetilde G_{\EE}(x-x_0)\, .
\end{align}
Clearly, if $y<\EE/2$,
\begin{align}
\nonumber
| \widetilde G_{\EE}(y)| &=\biggl |\frac{1}{2\Gamma(1/2)}\int_{-y+\EE}^{\infty}
\frac {1}{(y+\tau)^{1/2} \cdot \tau^{3/2}} \, \D \tau\biggr | \\
\label{proptildeGa}
&\le \biggl |\frac{1}{2\sqrt \pi} \frac{1}{(\EE/2)^{1/2}}\int_{-y+\EE}^{\infty}
 \tau^{-3/2}\, \D \tau \biggr |
= \frac{1}{\sqrt \pi} \frac{1}{(\EE/2)^{1/2}}\frac{1}{\sqrt{\EE-y}}
\le \frac{2}{\sqrt \pi} \frac{1}{\EE}\, .
\end{align}
We next focus on $ M_{-}^{1/2} (\phi_{x_0,+,\EE})(x)$.
Just like in the proof of \eqref{eqq},
taking a $C^1$ function $\psi$ compactly supported
in an interval $J$, we integrate by parts:
\begin{align*}
	\int_J \psi(x) M_{-}^{1/2}& (\phi_{x_0,+,\EE})(x) \, \D x\\&=-\frac{1}{2\Gamma(1/2)} \int_J\psi'(x)\lim_{\epsilon \downarrow 0} 
\int_\epsilon^{\infty} \frac{\tilde \phi_{x_0,\EE}(x)- \tilde \phi_{x_0,\EE}(x+\tau)}{\tau^{3/2}} \, \D \tau  \, \D x\, ,
\end{align*}
where $\tilde \phi_{x_0,\EE}(x)=0$ if $x<x_0$  and, otherwise, 
\begin{align}
\nonumber	\tilde \phi_{x_0,\EE}(x)&= \int_{-\infty}^x  \phi_{x_0,+,\EE} (y)\, \D y = \int_{-\infty}^x \frac {\mathbf{1}_{x_0+\EE>y > x_0}}{\sqrt{y-x_0}}\,  \D y\\
\label{otherwise2} &= \int_{x_0}^{\min(x,x_0+\EE) }\frac{1}{\sqrt{y-x_0}} \, \D y=
\begin{cases}
2\sqrt \EE\mbox{ if } x > x_0+\EE\, ,\\
	2 \sqrt {x -x_0} &\mbox{ if } x<x_0+\EE\, .
\end{cases}
\end{align}
Next, integrating by parts again,
we find for $x\in (x_0,x_0+\EE)$, 
that, for any $0<\epsilon< x_0-x+\EE$,
\begin{align*}
&-\int_\epsilon^{\infty} \frac{\tilde \phi_{x_0,\EE}(x)- \tilde \phi_{x_0,\EE}(x+\tau)}{2\tau^{3/2}} \, \D \tau\\
& = - 2
	\int_\epsilon^{x_0-x+\EE}\frac{\sqrt {x -x_0}- \sqrt {x+\tau -x_0}}{2\tau^{3/2}} \, \D \tau
- 2 \int_{x_0-x+\EE}^{\infty}\frac{\sqrt {x -x_0}-\sqrt \EE}{2\tau^{3/2}} \, \D \tau \\
& =   2
	\biggl [\int_\epsilon^{x_0-x+\EE} \frac{ 1}
	{2\sqrt {x+\tau -x_0}}\frac{1}{ \tau^{1/2}} \, \D \tau
 + \frac{\sqrt{x-x_0}-\sqrt \EE}{\sqrt{x_0-x+\EE}}
-\frac{\sqrt{x-x_0}-\sqrt {x-x_0+\epsilon}}{\sqrt{\epsilon}}\\
&\qquad\qquad\qquad\qquad\qquad\qquad\qquad\qquad\qquad
-\frac{\sqrt{x-x_0}-\sqrt \EE}{\sqrt{x_0-x+\EE}} \biggr ]\\
&=\int_\epsilon^{x_0-x+\EE} \frac{\phi_{x_0,+}(x+\tau)}{ \tau^{1/2}} \, \D \tau
-2\frac{\sqrt{x-x_0}-\sqrt {x-x_0+\epsilon}}{\sqrt{\epsilon}}\, .
\end{align*}
As $\epsilon \downarrow 0$, the right hand-side above
 tends to 
$\int_0^{\infty} \frac{\phi_{x_0,+,\EE}(x+\tau)}{ \tau^{1/2}} \, \D \tau$ for any fixed $x<x_0$.

If $x<x_0$, we find for any $0<\epsilon <x_0-x$, 
\begin{align}
\nonumber \int_{\epsilon}^{\infty} &\frac{\tilde \phi_{x_0,\EE}(x)- \tilde \phi_{x_0,\EE}(x+\tau)}{2\tau^{3/2}} \, \D \tau\\
\nonumber &
=  2   
	\int_{x_0-x}^{x_0-x+\EE}\frac{ \sqrt {x+\tau -x_0}}{2\tau^{3/2}} \, \D \tau +
2 \int_{x_0-x+\EE}^\infty \frac{\sqrt \EE }{2 \tau^{3/2}}
\, \D \tau \\
\nonumber  & =2
	\biggl [\int_{x_0-x}^{x_0-x+\EE} \frac{ 1}
	{2\sqrt {x+\tau -x_0}}\frac{1}{ \tau^{1/2}} \, \D \tau 
 -  \frac{\sqrt \EE}{\sqrt{x_0-x+\EE}} +  \frac{\sqrt \EE}{\sqrt{x_0-x+\EE}}\biggr ] \\
\label{recalling3}&=
\int_{0}^{x_0-x+\EE} \frac{\phi_{x_0,+}(x+\tau)}{ \tau^{1/2}} \, \D \tau = \int_{0}^{\infty} \frac{\phi_{x_0,+,\EE}(x+\tau)}{ \tau^{1/2}} \, \D \tau \, .
\end{align}
So, recalling the definition of  $I^{1/2}_-$,
for any $x < x_0+\EE$, we have
\begin{align*}
-\frac{1}{2 \Gamma(1/2)} \lim_{\epsilon \downarrow 0}& \int_{\epsilon}^{\infty} \frac{\tilde \phi_{x_0,\EE}(x)- \tilde \phi_{x_0,\EE}(x+\tau)}{\tau^{3/2}} \, \D \tau= I^{1/2}_- (\phi_{x_0,+,\EE})(x) \, .
\end{align*}
Summarising, and recalling \eqref{recalling},
 we have shown that
 if
$J \cap [x_0+\EE/2, \infty)=\emptyset$, then
\begin{align}
\label{recalling3b}
\int_J \psi(x) M^{1/2}_+(\phi_{x_0,+})(x) \, \D x=
-\int_J \psi'(x) I^{1/2}_+(\phi_{x_0,+})(x)\, \D x=\sqrt \pi \psi(x_0)\, ,
\end{align}
and\footnote{In particular, $M^{1/2}_-(\phi_{x_0,+,\EE})= 
-d I^{1/2}_-(\phi_{x_0,+,\EE})$, as  expected.}
\begin{align*}
\int_J \psi(x) M^{1/2}_-(\phi_{x_0,+})(x) \, \D x&=
\int_J \psi(x) \widetilde G_{\EE}(x-x_0) \, \D x+
\int_J \psi(x) M^{1/2}_-(\phi_{x_0,+,\EE})(x) \, \D x\\
&=\int_J \psi(x) \widetilde G_{\EE}(x-x_0)\,  \D x
+\int_J \psi'(x) I^{1/2}_- (\phi_{x_0,+,\EE})(x)\,  \D x\, .
\end{align*}
Next, by Remark~\ref{exchangext},
Lemmas~\ref{halff+} and ~\ref{halff} give  
for $x_0-\EE/2<x < x_0+\EE/2$ that
\begin{align}
\nonumber	I^{1/2}_{-,x} (\phi_{x_0,+,\EE} )(x)&=
 I^{1/2}_{+,t}(\phi_{x_0+t,+,\EE} )(x)|_{t=0} 
\\
\nonumber&= 
\frac{1}{\sqrt{\pi}}\bigl ( 
-
 \log |x- x_0| +\log \EE+
G_{\EE}(x_0-x)\bigr ) \\
\label{recalling4}&= 
\frac{1}{\sqrt{\pi}}\bigl (
-
 \log |x- x_0| +\log \EE+
G_{\EE}(x_0-x)\bigr )
\,  ,
\end{align}
where $y\mapsto G_{\EE}(y)$ is analytic, with 
\begin{equation}
\label{propGa}
\lim_{\EE\to\infty} \sup_{|y|<\EE/2} |\partial_y G_{\EE}(y)|=0 \, ,
\,\, \,\, 
\sup_{\EE>1} \, \sup_{|y|<\EE/2} \max( |\partial_y G_{\EE}(y)|,
 |\partial^2_y G_{\EE}(y)|)<\infty \, ,
\end{equation}

Recalling that $\psi$ is $C^1$ and vanishes at the endpoints of $J$, 
we have shown that
\begin{align}
 \nonumber 2\int_J  \psi(x) M^{1/2} &(\phi_{x_0,+})(x) \, \D x=
\int_J \psi(x)  \bigl (M^{1/2}_+ (\phi_{x_0,+}(x) -
 M^{1/2}_- (\phi_{x_0,+}(x))\bigr ) \,  \D x\\
\nonumber	&= 
\sqrt \pi \psi(x_0)
-\int_J  \psi(x) \widetilde G_{\EE}(x-x_0)\, \D x
\\ 
\nonumber&\qquad \qquad\qquad
-
\int_J \frac{\psi'(x)}{\sqrt\pi}
	\bigl [ -\log |x-x_0| +\log \EE+ G_{\EE}(x_0-x)
\bigr ]\,  \D x
\\
\label{recalling2}&=
\sqrt \pi\psi(x_0)
+\int_J  \psi(x) \bigl  [-\widetilde G_{\EE}(x-x_0)
+\frac{\partial_x G_{\EE}(x_0-x) }{\sqrt\pi}
 \bigr ]\, \D x
\\
\nonumber &\qquad\qquad\qquad\qquad\qquad\qquad\qquad\qquad\qquad
 - \frac{1}{\sqrt\pi}\int_J  \frac{\psi(x)}{x-x_0}\,  \D x  \, , 
\end{align}
if $\EE$ is large enough (depending only on $x_0$ and $J$).
Since the left-hand side above is independent of $\EE$,  the function
$$\GG(x-x_0):=-\widetilde G_{\EE}(x-x_0)
+{\partial_x G_{\EE}(x_0-x) }/{\sqrt\pi}$$ does not depend on $\EE$.
By \eqref{proptildeGa} and the first claim of \eqref{propGa},  we get 
$\GG(y)=0$. This   establishes  the claim on
$M^{1/2}_-(\phi_{x_0,+})$ and  the two-sided half derivative $M^{1/2}(\phi_{x_0,+})$. 
The claim on
$M^{1/2}( \phi_{x_0,-})(x)$ then  follows from
 \eqref{QQM} and \eqref{QQM!}.

Finally, the claims on truncated spikes $\phi_{x_0,\sigma,\EE}$
follow from \eqref{recalling} and \eqref{recalling2} combined with the fact that if $\sigma (x - x_0)<\EE$ then
$
M^{1/2}_\sigma(\phi_{x_0,\sigma}-\phi_{x_0,\sigma,\EE})(x)= 0 
$.
\end{proof}

\begin{proof}[Proof of Lemma~\ref{AbelMarchaud2}]
In view of \eqref{starting}, to show the claim on  $M_{+}^{1/2} ({\mathbf{1}_{x> x_0}}({\sqrt{x-x_0}}))$,  it is enough to check that
for any continuous $\psi$ vanishing at the endpoints of $J$, 
\begin{align*}
	\frac{1}{2\sqrt\pi} \int_J  \psi(x)
	\lim_{\epsilon \uparrow 0} \int^\epsilon_{-\infty}& \frac{{\mathbf{1}_{x> x_0}}({\sqrt{x-x_0}})
-{\mathbf{1}_{x+\tau> x_0}}({\sqrt{x+\tau-x_0}})}{ |\tau|^{3/2}}\, \D \tau  \, \D x\\
&=\frac{1}{2} \int_J \psi(x) I^{1/2}_+(\phi_{x_0,+})(x)  \, \D x \, .
\end{align*}
The above  follows from \eqref{forlater} and \eqref{otherwise}.
The claim on  $M_{-}^{1/2} ({\mathbf{1}_{x< x_0}}({\sqrt{x_0-x}}))$
then follows immediately from \eqref{QQM+} and \eqref{QQM!}.

For the  two-sided half derivative of truncated square roots,  
we  note that 
$$
M^{1/2}_\sigma(\bar \phi_{x_0,\sigma}-\bar \phi_{x_0,\sigma,\EE})(x)= 0 
\mbox{ if }\sigma (x - x_0)<\EE
\, .
$$
The claim on $M^{1/2}(\bar \phi_{x_0,\sigma,\EE})$ then follows from\eqref{recalling3}, \eqref{recalling3b}, and \eqref{recalling4}.
\end{proof}

\begin{proof}[Proof of Lemma~\ref{AbelMarchaud3}]
Since $\eta \in (0,1)$, for any $h\in \real$, we have
\begin{align*}
|M^\eta g(x+h)&-M^\eta g(x)|\\
&\le \frac{\eta}{2\Gamma(1-\eta)}\lim_{\epsilon\downarrow 0} \int_{|\tau |>\epsilon} \frac{|g(x+h+\tau)-g(x+h)-g(x+\tau)+g(x)|}
{|\tau|^{1+\eta} } \, \D \tau \\
&\le  \frac{\eta }{\Gamma(1-\eta)}\sup|g'|
\biggl (\lim_{\epsilon\downarrow 0} \int_\epsilon^{|h |}
\frac{\tau}
{\tau^{1+\eta} } \, \D \tau
+ \int_{|h|}^\infty\frac{|h|} {\tau^{1+\eta} } \, \D \tau \biggr )\\
&= \frac{\eta}{\Gamma(1-\eta)} \sup|g'|
\biggl (
\frac{|h|^{1-\eta}}{1-\eta}+\frac{|h|^{1-\eta}}{\eta}
\biggr ) \, .
\end{align*}
\end{proof}

\section{Rigorous results on fractional susceptibility functions}
\label{rigorr}

Before stating
Theorem~\ref{WTF}  in \S\ref{state} and proving it in \S\ref{WTFWTF},
we recall an expansion  for the invariant density $\rho_t$ due to Ruelle
in \S\ref{dealing}, and prove some of its consequences.

\subsection{Ruelle's formula for $\rho_t$. 
Fractional integration by parts. Exponential bounds. Proof of Proposition~\ref{CC2}}
\label{dealing}

Let $f(x)=f_{t}(x)=t-x^2$ for $t\in\MT$, let
$c_k=c_{k,t}$ and recall the sequence $s_k=s_{k,t}$
from \eqref{sigmak}.
The starting point for the proof below of our main Theorem~\ref{WTF}  is the expansion given by Ruelle  \cite[Theorem 9,
Remark 16A]{Ru} 
(in the slightly more general analytic   Misiurewicz setting)
for  the invariant density $\rho_{t}$ of $f_{t}$,  supported
in $[c_2,c_1]$:
\begin{align}\label{Ruellerho}
\rho_{t}(x)&=\psi_0(x) + \sum_{k = 1}^{\infty}
C^{(0)}_k 
\frac{\mathbf{1}_{ w_{0}<s_{k-1} (x-c_k)<0}}
{\sqrt{|x- c_k|}}\\
\nonumber &\qquad\qquad\qquad+ \sum_{k = 1}^{\infty}
C_k^{(1)} \cdot 
\mathbf{1}_{ w_{1}<s_{k-1} (x- c_k)<0} \cdot \sqrt{|x- c_k|}
\, ,
\end{align}
where\footnote{The cutoff is slightly different in  Ruelle  \cite[Theorem 9,
Remark 16A]{Ru}, who observes that ``other choices can be useful.''} $\psi_0$ is a $C^1$ function, 
$w_{1}<0$, $w_{0}<0$, 
and where (for some $U_{t}\ne 0$)
\begin{equation*}
C^{(0)}_k=\frac{\rho_{t}(0)}
{|D f^{k-1}_{t}(c_1)|^{1/2}}\, , 
\quad |C_k^{(1)} |\le \frac{U_{t}}
{|D f^{k-1}_{t}(c_1)|^{3/2}}\, ,
\quad \forall k\ge 1\, .
\end{equation*}
Since $t \in \MT$,  we have  $c_{k+P}=c_k$ for
$k\ge L$. Note also  that, if $D f^P_t (c_{L})>0$, then the spikes and square roots along the postcritical orbit are all one-sided. If $D f^P_t(c_{L})<0$ then
the spikes and square roots along the periodic part of the postcrititical orbit are all two-sided.

\begin{remark}\label{BS12exp}	For more general TSR parameters, one could
use  \cite[Prop 2.7]{BS2}   instead  of \eqref{Ruellerho}.
(To obtain an  expansion involving  spikes and square roots in the TSR setting,
one could upgrade the results of \cite{BS2}, showing that if
$f_t$ is smooth enough then the smooth component of $\rho_t$ belongs to 
$W^r_1$ for large enough $r>2$.)
\end{remark}

In the remainder of this section, we show three consequences of  \eqref{Ruellerho}.

First, we show 
that the  integration by parts formula \eqref{identity} holds for all $\eta\in (0,1)$
(this will be used to prove Proposition~\ref{CC2} which implies Proposition~\ref{propE}):

\begin{lemma}[Fractional integration by parts in the response susceptibility]\label{identity1}
Let $t\in \MT$.
 For any $\eta \in (0,1)$ and any compactly supported $\phi \in C^1$, we have, as formal power series,
\begin{equation*}
\sum_{k=0}^\infty  z^k \int M^\eta_x   (\phi \circ f_{t}^k)
\cdot
 \rho_{t}  
\, \D x
= 
-\sum_{k=0}^\infty  z^k \int  (\phi \circ f_{t}^k)
\cdot
M^\eta_x  \bigl ( \rho_{t}  \bigr )
\, \D x  \,  .
\end{equation*}
(By definition, the left-hand side above is just $\Psi^{\mathrm{rsp}}_\phi(\eta,z)$.)
\end{lemma}

\begin{proof}
Fix $\eta \in (0,1)$.
It suffices to show that, for any compactly supported
$\psi \in C^1$, we have 
\begin{equation*}
 \int M^\eta_x ( \psi)(x)
\cdot
 \rho_{t}(x)  
\, \D x
= 
-\int  \psi(x)
\cdot
M^\eta_x   ( \rho_{t}   ) (x)
\, \D x  \,  ,
\end{equation*}
where $M^\eta_x  ( \rho_{t}   )$ is understood in the sense of distributions (of order one).

Since $\psi$ is $C^1$ and compactly supported, we have
\begin{equation}\label{neat}
2 M^\eta_x ( \psi)(x)
=\partial_x ( (I^{1-\eta}_++I^{1-\eta}_-)\psi)(x)=  (I^{1-\eta}_++I^{1-\eta}_-) \psi'(x)\, .
\end{equation}
(Use Remark~\ref{thesame} for the first equality and the definition of $I^{1-\eta}_\pm$ for the second.)
Next, using the expansion \eqref{Ruellerho} for $g(x):=\rho_{t}(x)$,  
we find that 
 $G(y):=\int_{-\infty}^y g(u) \, \D u$ is  the sum of a $C^1$ function with a (finite) sum of one- or two-sided  truncated square roots
along the postcritical orbit (see 
\eqref{otherwise2}). Thus (recalling \eqref{Meps})
$$\lim_{\epsilon \to 0}  
M^\eta_\epsilon G(x)=M^\eta G(x)\in L^1_{loc} \, .
$$
The above claim is clear if $\eta <1/2$. For $\eta=1/2$, it follows from
Lemma~\ref{AbelMarchaud2}. Finally, for $\eta \in (1/2,1)$, we may decompose
$M^\eta_\pm=M^{\eta-1/2}_\pm \circ M^{1/2}_\pm$, using the semigroup property
(\cite[Property 2.4]{Kilbas},
for\footnote{The reference to Lemma 2.4  there should be replaced by Lemma 2.5.} $m=1$ and $\alpha=\eta-1/2$,
noting that  $G\in L^1$ and $I^{3/2-\eta}_\pm (G)\in AC$ since $3/2-\eta>1/2$).

Now, on the one hand, by definition, we have
\begin{align}
\nonumber 
&2 \int  \psi(x)  (M^\eta g)(x)\, \D x
=-2 \int  \psi'(x) \cdot M^\eta G(x) 
 \, \D x \\
\nonumber
&\qquad\qquad\quad= \int  \psi'(x)\bigl [  \frac{\eta}{ \Gamma(1-\eta)}
\int
\frac {G(x+\tau)-G(x)}{|\tau|^{1+ \eta}} 
\sgn(\tau) \, \D \tau \bigr ] 
 \, \D x\\
\label{Mardistr'}&\qquad\qquad\quad= \int  \psi'(x)\bigl [ \frac{\eta}{ \Gamma(1-\eta)}
\int
\frac {\int_{x}^{x+\tau} g(u) \, \D u}{|\tau|^{1+ \eta}} 
\sgn(\tau) \, \D \tau \bigr ] 
 \, \D x \, ,
\end{align}
where \eqref{Mardistr'} can be rewritten, integrating by parts in $\tau$, as
\begin{align*}
\int  \psi'(x)  \frac{1}{ \Gamma(1-\eta)}
\int
\frac { g(x+\tau) }{|\tau|^{\eta}} 
 \, \D \tau  
 \, \D x\, .
\end{align*}
On the other hand,  \eqref{neat} followed by fractional integration by parts \cite[(5.16)]{sam} gives
\begin{align*}
2 \int M^\eta_x ( \psi)(x)
\cdot
 g  (x)
\, \D x
&=  \int  (I^{1-\eta}_++I^{1-\eta}_-) (\psi') (x)  \cdot
 g (x)
\, \D x \\
&= \int\psi'  (x)  \cdot
  (I^{1-\eta}_++I^{1-\eta}_-) ( g)(x) 
\, \D x\\
&= \int\psi'  (x)  \cdot
\frac{1}{\Gamma(1-\eta)} \int 
 \frac{   g(x+\tau) }{|\tau|^\eta} \D \tau 
\, \D x
\, .
\end{align*}
\end{proof}

Next, we use\footnote{It would probably be possible to
apply \cite[Thm 2.II.b)]{Yo2} instead.} 
the expansion \eqref{Ruellerho} to get  the following exponential bounds, useful to prove Theorem~\ref{WTF}:
\begin{lemma}
[Action of the transfer operator on Sobolev spaces $H^r_q$ for $r>0$]\label{A3}
Let $t\in \MT$ be a mixing parameter. Let $r>0$, $q>1$.
 There exist $C<\infty$ and $\kappa<1$ such that,
for any\footnote{If $\supp(\psi)\subset I_t$, the first term   is
$\int \varphi( x) \LL^j_t( \psi (x))\, \D x$,
thus the name of the  lemma.} $\psi \in H^r_q[-2,2]$ 
and any bounded 
$\varphi$ supported in $[-2,2]$
$$
|\int \varphi( f^j_t(x)) \,  \psi (x)\, \D x -
\int \varphi(x) \, \D\mu_{t} \cdot \int_{I_t} \psi \, \D x |
\le C \|\varphi\|_{L^\infty}  \|{\psi} \|_{H^r_q} \kappa^j  \, ,\forall j\ge 0 \, .
$$
\end{lemma}

Lemma~\ref{A3}  applies to the
Heaviside function $\psi=\mathbf{1}_{x>y}$ .

\begin{proof}[Proof of Lemma~\ref{A3}]
Since $q>1$ and we are in a one-dimensional setting, the Sobolev embeddings imply that, for any $\tilde r >2$ (we may choose $\tilde r<2+r$),
there exists $\tilde C$ such that for any compactly supported $g\in H^{\tilde r}_q$
$$
\|g\|_{C^1} \le \tilde C  \|g\|_{H^{\tilde r}_q} \, .
$$
Since $r>0$, using mollification, we can approach $\mathbf{1}_{[-2,2]} \psi$ by $C^1$ functions
 $\psi_\epsilon$ with
$$
\|\psi_\epsilon\|_{C^1} \le \tilde C  \| \psi_\epsilon\|_{H^{\tilde r}_q} 
\le  \tilde C_0\frac{\|\psi\|_{H^r_q}}{\epsilon^{2}}\, ,
\quad 
\|  \psi-\psi_\epsilon\|_{L^q[-2,2]}
\le \tilde C_1 \epsilon^{r} \|\mathbf{1}_{[-2,2]} \psi\|_{H^r_q}\, , 
\, \forall \epsilon >0 \,  .
$$
Note that $(\mathbf{1}_{[c_2,c_1]}\psi_\epsilon)/\rho_{t}\in BV$,
with BV norm bounded by $C_0\|\psi_\epsilon\|_{C^1[c_2, c_1]}$, because
$\mathbf{1}_{[c_2,c_1]}/\rho_{t}\in BV$. (To check this, use that\footnote{See e.g. \cite[Theorem 2c)]{Yo}, or, in the MT case \cite{Og}.} $\inf_{[c_2, c_1]} \rho_{t}>0$  and that $\rho_{t}$
is the sum of a  $C^1$ function together with
finitely many square roots spikes and square roots, by \eqref{Ruellerho},
and consider separately each maximal interval bounded by postcritical points.)

Since it is easy to find $C_0< \infty$ and $\kappa <1$ (independent of $\varphi$, $\psi$) such  that
$$
\int_{\real\setminus I_t}| (\varphi \circ f_{t}^j)\,  \psi |\, \D m
\le C_0 \kappa^j \|\varphi\|_{L^\infty} \|\psi \|_{L^q}\, ,
\, \, \forall j\ge 1\, ,
$$
and since we can write $\int^{a_t}_{c_1} (\varphi \circ f_{t}^j)\,  \psi \,\D m=\int_{ -a_t}^{c_2} (\varphi \circ f_{t}^{j-1})\,  (\psi\circ f^{-1})
|(f^{-1})' |\,\D m$,
and 
\begin{align*}
\int_{ -a_t}^{c_2} (\varphi \circ f_{t}^j)\,  \psi \,\D m
&= \int_{-a_t}^{c_2} \varphi (f_{t}^{j-[j/2]}(x))
\frac{\psi(f^{-[j/2]}(x)}{(f^{[j/2]})'(f^{-[j/2]}(x))}\,  \D x\\
&\qquad+\sum_{\ell=[j/2]}^{j-1}\int_{c_2}^{c_3} \varphi(f^\ell(x)) \frac{\psi(f^{\ell-j}(x))}{|(f^{j-\ell})'(f^{\ell-j}(x))|} \D x\, ,
\end{align*}
($f^{-k}$ above denotes
$(f^k|_{\cap_{j=1}^k f^{-j}[-a_t,c_2]})^{-1}$),
which gives the limiting contribution $\int \varphi(x) \, \D\mu_{t} \cdot \int_{I_t\setminus[c_2,c_1]} \psi \, \D x$, the lemma follows from three facts.
First, 
$$
\int_{[c_{2},c_{1}]} (\varphi \circ f_{t}^j) \psi_\epsilon \, \D x=
\int_{[c_{2},c_{1}]} (\varphi \circ f_{t}^j) \frac{\psi_\epsilon}{\rho_{t}} \, \D \mu_{t}
\, .
$$
Second, there exist $\theta <1$, $C_1 <\infty$ (independent of
$\varphi$, $\psi_\epsilon$, see \cite[Theorem~1.1]{KeNo}, by  the
principle of uniform boundedness,
$C_1$ does not depend on $\psi_\epsilon$) such that
$$
|\int_{c_2}^{c_1} (\varphi \circ f_{t}^j) \frac{\psi_\epsilon }{\rho_{t}}\, \D \mu_{t}-
\int \varphi \, \D \mu_{t} \int_{c_2}^{c_1} \psi_\epsilon \, \D m
|\le C_1  \|\varphi\|_{L^1} \biggl \|\frac{\mathbf{1}_{[c_2,c_1]}\psi_\epsilon}{\rho_{t}}\biggr \|_{BV} \theta^j\, , \,\,
\forall j \ge 1 \, .
$$
Third, 
$$
\max \bigl \{ |\int \varphi \, \D \mu_{t} \int_{I_t} (\psi_\epsilon-\psi) \, \D m|,
|\int_{I_t} (\varphi \circ f_{t}^j) (\psi -\psi_\epsilon )\, \D m|
\bigr \}\le
\sup|\varphi| \| \psi-\psi_\epsilon \|_{L^q[-2,2]} \, .
$$
To conclude, for each $j$  choose $\epsilon=\theta^{j/(r+2)}$, so that
$
\frac{\theta^j}{\epsilon^2} =
\epsilon^r = \theta^{jr/(r+2)}=: \kappa^j
$.
\end{proof}

We can now provide the proof of  
Proposition~\ref{CC2}:

\begin{proof}[Proof of Proposition~\ref{CC2}]
Setting $R_{t}(y)=\int_{-10}^y \rho_{t}(x)\, \D x$, it is easy
to see (using e.g \eqref{Ruellerho}, or, in the TSR case, \cite{BS2})
that $\mathbf{1}_{[-1,1]} \cdot  M^\eta_x R_{t}(x)$ belongs to $L^q$ for any
$\eta \in (0,1)$ and any $1\le q<2$. So $\sup_{\eta \in (0,1)}
\int_{-1}^1 |M^\eta_x R_{t}(x)| \, \D x <\infty$. In particular,
$M^\eta_x (\rho_{t})(x)$ is well defined in the sense of distributions
of order one\footnote{This was already established in Lemma~\ref{identity1}.} uniformly in $\eta \in (0,1)$.

Next, since $f_{t+\tau}(x)=f_{t}(x)+\tau$ for the
quadratic family, we find, recalling \eqref{extendd},
\begin{align*}
(\LL_{t+\tau} \rho_{t})(x)
&=(\LL_{t} \rho_{t})(x-\tau)=\rho_{t}(x-\tau)  \, , \forall x 
\, ,
\, \forall |\tau |<\epsilon_0\, , \\
(\LL_{t+\tau} \rho_{t})(x)&=(\LL_{t} \rho_{t})(x\mp \epsilon_0)=\rho_{t}(x\mp \epsilon_0)\, , \forall x \, ,
\mbox{ if } \pm \tau > \epsilon_0\, .
\end{align*}
This implies that for any $\eta \in (0,1)$ and any $x$, we have
\begin{align*}
\Gamma(1-\eta) \bigl( M^\eta _t&(\LL_s\rho_{t}(x))|_{s=t}+M^\eta_x \rho_{t}(x)\bigr) =\\
&= \frac{\eta}{2}\int_{|\tau|>\epsilon_0} 
\frac{\rho_{t}(x-\sgn(\tau)\epsilon_0)-\rho_{t}(x+\tau) }{|\tau|^{1+\eta}} \sgn(\tau)\,  \D \tau \\
&
=\frac{1}{2}\frac{\rho_{t}(x-\epsilon_0)
-\rho_{t}(x+\epsilon_0)}{\epsilon_0^{\eta}} 
- \frac{\eta}{2}\int_{|\tau|>\epsilon_0} 
\frac{\rho_{t}(x+\tau) }{|\tau|^{1+\eta}} \sgn(\tau)\,  \D \tau\, .
\end{align*}
The above defines a function $g_\eta\in H^r_q$ for some $r>0$ and $q>1$
for any $\eta >0$, uniformly in $\eta>\epsilon_1$,
for any fixed $\epsilon_1>0$, and such that \eqref{Prop2.5} holds. Clearly $\int_{\real} g_\eta(x) \D x=0$.
Since $M^\eta_x \rho_{t}(x)$ is  a  distribution of order one, 
$M^\eta _s(\LL_s\rho_{t}(x))|_{s=t}$
(which is compactly supported) is also a distribution of order one.

We next establish the relation between the frozen and the response susceptibility functions: On the one hand, recalling \eqref{is M} and using \eqref{Prop2.5}, we  have
\begin{align*}
\Psi^\mathrm{fr}_\phi(\eta,z)
&= \sum_{k=0}^\infty  z^k \int_{I_{t}}  
 \phi ( f_{t}^k (x) ) \biggl (\frac{g_{\eta}(x)}{\Gamma(1-\eta)}
 -M^\eta_x \rho_{t}(x)
\biggr ) \,
\D x\, .
\end{align*} 
Lemma~\ref{A3} holds for the term involving $g_\eta$.  On the other hand,  Lemma~\ref{identity1} gives
\begin{equation*}
\Psi^{\mathrm{rsp}}_\phi(\eta,z)
=  -\sum_{k=0}^\infty  z^k \int_{I_t}  \phi(f^k_t(x)) 
M^{\eta}_x    \rho_{t}  (x)
\, \D x \, .
\end{equation*}
The last claim of Proposition~\ref{CC2} follows from  Lemma~\ref{limitok}.
\end{proof}

\subsection{Theorem~\ref{WTF} on  
$\Psi^\mathrm{fr}_\phi(1/2,z)$ at  MT parameters}\label{state}
For a CE parameter  $t$, recall $\lambda_c >1$ from \eqref{00},   
 $\UU_{ 1/2}(z)$, $\UU_{ 1/2}^+(z)$   from  \eqref{defUU}, \eqref{defUU+},
and $\Sigma_\phi(z)$, $\Sigma^\HH_\phi(z)$, and
$\Sigma^{\tilde \psi}_{t}(z)$
from \eqref{defSig}, \eqref{defSigH},
and \eqref{defSigH2}. Generalising \eqref{sigmak}, we put, $s_{k,1}=s_{k,1,t}=s_{k,t}$ and
\begin{equation}\label{sigmakell}
s_{k,\ell}:=s_{k,\ell,t}=\sgn( Df^k_t(c_\ell)) 
\, , \quad k \ge 1 \, , \, \ell \ge 1 \, .
\end{equation}
The following elementary lemma is proved at the end of \S\ref{WTFWTF} 
(see \cite[Remark 1.2]{BMS} for the case of piecewise
expanding maps):

\begin{lemma}\label{prelim}
Let $t \in\MT$ with $f_t^P(c_L)=c_L$. 
Then $\UU_{ 1/2,t}(z)$
is rational, with poles at the $P$th roots of $\sgn(Df^P(c_L))|Df^P(c_L)|^{-1/2}$,
and $\UU_{ 1/2,t}(1)=u_t\JJ_{1/2}(t)$, and
$\UU^+_{ 1/2,t}(z)$ is rational, with poles at the $P$th roots of $|Df^P(c_L)|^{-1/2}$, and $\UU^+_{ 1/2,t}(1)=u_t\JJ_{1/2}^+(t)$.

For $\phi \in C^1$, the function  $\Sigma_{\phi,t}(z)$ is
  rational,  with possible simple poles at the $P$th roots
of unity, while  $\Sigma^\HH_{\phi,t}(z)$ is
  rational,  with possible simple poles at the $P$th roots
of $\sgn ( Df^P(c_{L}))$. 
For  any $r>0$, $q>1$ and any  sequence
$\tilde \psi(\ell)\in H^r_q[-2,2]$ such that $\tilde \psi(\ell)=\tilde \psi(\ell+p)$
for $\ell \ge L$, the function 
$z \mapsto  \Sigma_t^{\tilde \psi}(z)\in H^r_q[-2,2]$ is
  rational,  with possible simple poles at the $P$th roots
of $\sgn ( Df^P(c_{L}))$.

Set
$
\PP_t(z)=(z-1) \cdot  \Sigma_{\phi,t}(z)$ and $\PP_t^+(z)=(z-1)\cdot 
\Sigma^\HH_{\phi,t}(z)$.
Then we have
$\PP_t(1)=   \frac{ 1 }{P}\sum_{\ell=L}^{L+P-1}  \phi(c_\ell)$.

If\footnote{If  $\sgn(Df^P(c_L))=-1$, then, clearly,  $\PP^+_t(1)=\PP_t^{\tilde \psi}(1)=0$.}  $\sgn(Df^P(c_L))=+1$, then we have that $\PP^+_t(1)=  
 \frac{1 }{P}\cdot  \sum_{\ell=L}^{L+P-1}   s_\ell\,  \HH(\mathbf{1}_{I_t}\phi)(c_{\ell})$, and,
setting $\PP_t^{\tilde \psi}(z)=(z-1)\cdot 
\Sigma^{\tilde \psi}_{t}(z)$,  that
$\PP_t^{\tilde \psi}(1)=\frac{1}{P}\sum_{\ell=L}^{L+P-1}   s_\ell \tilde \psi_\ell$.
\end{lemma}

Our
main theorem is  proved in \S\ref{WTFWTF} (it is reminiscent\footnote{With respect to \cite{BMS} the term $\WW_{\phi, 1/2}(z)$ and the presence of
the Hilbert transform  are new.} of \cite{BS0, BMS, Ru}):

\begin{customthm}{C}[Frozen susceptibility function $t\in \MT$]\label{WTF}
Let  $t \in \MT$ be  mixing with $f_t^P(c_L)=c_L$.
There exist $\kappa<1$ and a  sequence
$\tilde \psi_t(\ell)\in H^r_q[-2,2]$ with $\tilde \psi_t(\ell)=\tilde \psi(\ell+p)$
for $\ell \ge L$,  and $\int_{I_t} \tilde \psi_\ell\,  \D m=0$ for all $\ell$,
 such that the following holds:
For any compactly supported $\phi$ in $C^1$,
\begin{align*}
 \Psi^\mathrm{fr}_{\phi,t}(1/2, z)&= 
\UU_{ 1/2,t}(z)\Sigma_{\phi,t}(z)
+ \WW_{\phi, 1/2,t}(z)+
\VV_{\phi, 1/2,t}(z) \, ,
\end{align*}
where  $\VV_{\phi, 1/2,t}(z)$ is holomorphic in the annulus
$\{\lambda_c^{-1/2}<|z|<\kappa^{-1}\}$,
while, 
\begin{equation*}
  \WW_{\phi, 1/2,t}(z)= 
\UU^{+}_{1/2,t}(z)
\Sigma^\HH_{\phi,t}(z)
+    \sum_{\ell=0}^\infty \int
(\phi \circ f^\ell_t) \cdot   \Sigma^{\tilde \psi_t}_{t} (z) \, \D m 
 \, .
\end{equation*}
If $\sgn ( Df^P(c_L))=-1$, then $\Psi^\mathrm{fr}_\phi(1/2, z)$
has a simple pole at  $z=1$, with residue 
\begin{equation}\label{lastcorclaim2}
  u_t \cdot \frac {\JJ_{1/2}(t)} P  \sum_{k=L}^{L+P-1}  \phi(c_k)\, .
\end{equation}
If $\sgn ( Df^P(c_{L}))=1$, then
there\footnote{
The notation $\int \phi \,  \tilde\psi^* \, \D m$  represents
the action of $\tilde\psi^* \in (L^\infty[-2,2])^*$ on $\phi\in L^\infty[-2,2]$.
The formula defining $\tilde \psi^*_t$ is given in \eqref{defphitilde0}--\eqref{defphitilde}, it does not depend on $\phi$.}  exists $\tilde \psi^*_t\in (L^\infty[-2,2])^*$ with $\int_{I_t} 
\psi_t^* \, \D m=0$, such that
the residue of the simple pole at
$z=1$ of $\Psi^\mathrm{fr}_\phi(1/2, z)/u_t$ 
is equal to
\begin{align}\label{valW}\frac{\JJ_{1/2}(t)}{P}\cdot \sum_{k=L}^{L+P-1}  \phi(c_k)
+
\frac{\JJ^+_{1/2}(t)} P \cdot  \biggl (
\sum_{\ell=L}^{L+P-1} s_\ell \cdot \HH (\mathbf{1}_{I_t}\phi)(c_{\ell}) 
+  \int
\phi \cdot  \tilde \psi^*_t  \, \D m \biggr )\, .
\end{align}
\end{customthm}

\smallskip

The vanishing of \eqref{valW}  is a codimension-one condition on $\phi$.
If $\JJ_{1,2}(t)\ne 0$ the vanishing of \eqref{lastcorclaim2} is a codimension-one condition on $\phi$. Thus, in view of Lemma~\ref{prelim},
 Theorem~\ref{WTF}  establishes the analogue of  Conjecture~\ref{laconj}[iii] and Conjecture~\ref{laconj+} 
 for the frozen susceptibility function at MT parameters.

\begin{remark}The proof of Theorem~\ref{WTF}  shows that the statements also
hold for $\Psi^\mathrm{rsp}_\phi(1/2, z)$, up to replacing the function
 $\VV_{\phi, 1/2}(z)$ by $\VV_{\phi, 1/2}(z)-\VV_{\phi, 1/2}^{\mathrm{rsp}}(z)$
(using Proposition~\ref{CC2}).
\end{remark}

Besides Ruelle's expansion \eqref{Ruellerho}, the proof of Theorem~\ref{WTF} in \S\ref{WTFWTF}  will be based on Proposition~\ref{CC2}, Lemmas~\ref{AbelMarchaud}--\ref{AbelMarchaud3} and Lemma~\ref{A3} above, and   Lemma~\ref{A2} below.

\begin{lemma}
[Action of the transfer operator on poles] \label{A2}
For $\ell\ge 2$ with $c_{\ell-1}=c_{\ell-1,t}\ne 0$, set
\begin{equation*}
\tilde \chi_{\ell}(x)=\tilde \chi_{\ell,t}(x)=
\frac{s_{1,\ell-1} \cdot \mathbf{1}_{x\ge c_1}} {x-c_{\ell}}
+\frac{\mathbf{1}_{x<c_1}}{c_{\ell-1} \sqrt{c_1-x} } 
+\frac {\mathbf{1}_{x<c_1}} {c_{\ell-1}} \frac{\int_{c_{\ell}}^x \frac{1}{2\sqrt{c_1-u}}\D u}
{x-c_{\ell}}\, .
\end{equation*}
 Then, for any $k\ge 1$ such that $c_{k}\ne 0$, we have, setting $\chi_k=\chi_{k,t}=(x-c_{k,t})^{-1}$,
\begin{equation*}
 \LL_{t} \chi_k=
  s_{1,k} \, \chi_{k+1}
+\tilde \chi_{k+1} 
 \, .
\end{equation*} 
If $t \in \MT$,  there exist $r>0$ and $q>1$
such that $\tilde \chi_{\ell} \in H^r_q[-2,2]$ 
 for each $\ell\ge 2$.
\end{lemma}

\begin{proof} [Proof of Lemma~\ref{A2}]
For any $x< t=c_1$
and $k\ge 1$, using $c_k^2=c_1-c_{k+1}$ twice in the second equality of the
third line, we find, inspired by the beginning of the proof of \cite[Theorem 2]{Tsu2},
\begin{align}
\nonumber \LL_t \chi_k (x)&=
\frac{1}{2\sqrt{c_1-x}}\biggl(\frac 1 {\sqrt{c_1-x}-c_k}+
\frac 1 {-\sqrt{c_1-x}-c_k}\biggr )\\
\nonumber & =\frac{1}{c_k \sqrt{c_1-x} }\, 
\frac{- c_k^2}{(c_k^2 -c_1+x)}\\
\nonumber &= \frac{ \sqrt{c_1-x}}{c_k  }
\frac{- c_k^2}{(c_1-x)(c_k^2 -c_1+x)}=
-\frac{ \sqrt{c_1-x}}{c_k  }\, \frac{c_1-c_{k+1}}{(c_1-x)(x-c_{k+1})}
\\
\nonumber&=-\frac{ \sqrt{c_1-x}}{c_k  }\biggl(
\frac 1 {x-c_{k+1}}- \frac 1 {x-c_1}
\biggr)=-\frac{ \sqrt{c_1-x}}{c_k  } \, (\chi_{k+1}(x)-\chi_1(x))\, .
\end{align}
Now, $-\sqrt{c_1-c_{k+1}}= c_k$ if $c_k<0$ 
that is, $s_{1,k}=1$,
while  $-\sqrt{c_1-c_{k+1}}= -c_k$ if $c_k>0$  
that is, $s_{1,k}=-1$. Thus,
using a Taylor series at $c_{k+1}$, we find 
$$
-\sqrt{c_1-x}=s_{1,k}\cdot c_k +\int_{c_{k+1}}^x \frac{1}{2\sqrt{c_1-u}}\D u\, , \,\,\,\forall x<c_1\, .
$$
Therefore
\begin{equation*}
\LL_t \chi_k (x)=s_{1,k}\cdot \mathbf{1}_{x<c_1} \cdot \chi_{k+1}(x)
+\frac{\mathbf{1}_{x<c_1}}{c_k \sqrt{c_1-x} } 
+\frac {\mathbf{1}_{x<c_1}} {c_k} \frac{\int_{c_{k+1}}^x \frac{1}{2\sqrt{c_1-u}}\D u}
{x-c_{k+1}} \, .
\end{equation*}
In other words, setting
\begin{equation*}
\tilde \chi_{k+1}(x)=
\frac{s_{1,k} \cdot \mathbf{1}_{x\ge c_1}} {x-c_{k+1}}
+\frac{\mathbf{1}_{x<c_1}}{c_k \sqrt{c_1-x} } 
+\frac {\mathbf{1}_{x<c_1}} {c_k} \frac{\int_{c_{k+1}}^x \frac{1}{2\sqrt{c_1-u}}\D u}
{x-c_{k+1}} \, ,
\end{equation*}
we have proved
\begin{equation*}
\LL_t \chi_k (x)=s_{1,k} \cdot \chi_{k+1}(x)
+\tilde \chi_{k+1}(x) \, .
\end{equation*}
It is easy to find $r>0$ and $q>1$
such that all $\tilde \chi_{\ell}  \in H^r_q[-2,2]$ 
 if $t \in \MT$.
\end{proof}

\begin{remark}\label{notC}
We introduce notation useful for the proof
of Theorem~\ref{WTF}.
Let $\YY_t$ be the $L+P-1$-dimensional vector space generated
by the functions 
\begin{equation}
\label{defchik}
\chi_k=\chi_{k,t}=(x-c_{k,t})^{-1}\,  , \,\, k = 1,\ldots L+P-1\, .
\end{equation}
We write $\chi(\vec Y)=\sum_{k=1}^{L+P-1} Y_k \cdot \chi_k$
for $\vec Y=(Y_k)\in \complex^{L+P-1}$.
Then in view of Lemma~\ref{A2}  it is natural to introduce
the finite $L+P-1\times L+P-1$ matrix $\mathbb{S}=\mathbb{S}_t$
acting on $\YY_t$, 
with coefficients 
\begin{equation*}
\mathbb{S}_{k, j}=s_{  1,j} \delta_{k, j+1}
+s_{ 1,L+P-1} \delta_{k,L}\delta_{j,L+P-1}\, .
\end{equation*}
The eigenvalue zero of  $\mathbb{S}_t$ has algebraic multiplicity
$L-1$ but geometric multiplicity equal to one. Since $\sgn ( Df^P_t (c_{L}))=\prod_{k=L}^{L+P-1} s_{1,k}=
s_{P,L}=\sgn ( Df^P_t (c_{L}))$,
the nonzero eigenvalues of $\mathbb{S}_t$ consist
in the $P$th roots of $\sgn ( Df^P_t (c_{L}))$,  they are simple. 
\end{remark}

\subsection{Proof of the  main result (Theorem~\ref{WTF})}
\label{WTFWTF}

\begin{proof}[Proof of Theorem~\ref{WTF}]
We already observed that
$x\mapsto M^{1/2}_s (\LL_s \rho_{t}(x))|_{s=t}$  is supported in $I_{t,\epsilon}\subset I_t$  (recall Footnote \ref{footd}),
while the support of $\mu_t$ is contained in $I_t$.
Thus, for each compactly supported
$C^1$ function $\tilde \phi$
  such that $\tilde \phi(x)=\int \phi \, \D \mu_t=:\phi_*$
for all $x$ in $I_t$, and each sequence of   $C^1$
functions $\upsilon_k$
with 
$\upsilon_k(x)\equiv 1$ if $|x|\le k$
and  $\upsilon_k(x)\equiv 0$ if $|x|\ge 2 k$,
Proposition~\ref{CC2} gives
\begin{align*}
\int  \tilde \phi (x) &M^{1/2}_t (\LL_t \rho_{t}(x)) \D x
= \phi_*
\cdot 
\int_{I_t}
 M^{1/2}_t (\LL_t \rho_{t}(x)) \D x\\
&=
 \phi_*
\cdot \lim_{k \to \infty}
\int_\real \upsilon_k(x)
 M^{1/2}_t (\LL_t \rho_{t}(x)) \D x\\
&=\phi_*
\cdot  \lim_{k \to \infty}\bigl (
\int_\real \upsilon_k(x) M^{1/2}_x \rho_{t}(x) \D x +
\int_\real \upsilon_k(x) \frac{g_{1/2}(x)}{\Gamma(1/2)} \D x  \bigr )
=0\, .
\end{align*} 
(To show $\lim_{k \to \infty}
\int_\real \upsilon_k(x) M^{1/2}_x \rho_{t}(x) \D x=0$, recall Definition~\ref{ddistr}, note that $\upsilon'_k(x)=0$ if $|x|>2k$,
and
$G(y):=\int_{-\infty}^y \rho_{t}(x)  \, \D x=0$  
if $y<c_{2}$ while $G(y)=1$ if $y>c_1$, and use that the Marchaud derivative
of any constant function vanishes.)
Therefore, 
$$
\int  \phi(f_t^j (x)) M^{1/2}_t (\LL_t \rho_{t}(x)) \D x=
\int_{I_t}  (\phi -\tilde \phi) (f_t^j (x)) M^{1/2}_t (\LL_t \rho_{t}(x)) \D x
\, ,\,\, \forall j \ge 0 \, .
$$

From now on, replacing $\phi$ by $\phi -\tilde \phi$
if necessary, we may thus assume that $\phi$ is compactly
supported, $C^1$ and has zero average
with respect to $\D \mu_t$. (This will allow us to 
  exploit exponential decay 
of correlations from Lemma~\ref{A3}.)

Our starting point is then that $\Psi^{\mathrm{fr}}_\phi(1/2,z)=\Psi^{\mathrm{rsp}}_\phi(1/2,z)+\VV^{\mathrm{rsp}}_{\phi,1/2}(z)$, with $\VV^{\mathrm{rsp}}_{\phi,1/2}(z)$, holomorphic in the disc of radius $\kappa^{-1}>1$,
from Proposition~\ref{CC2}.
By  Lemma~\ref{identity1},
\begin{equation*}
\Psi^{\mathrm{rsp}}_\phi(1/2,z)
=  -\sum_{j=0}^\infty  z^j \int_{I_t}  \phi(f^j_t(x)) 
M^{1/2}_x   ( \rho_{t})  (x)
\, \D x \, .
\end{equation*}
Therefore, using the expansion  \eqref{Ruellerho} for  $\rho_{t}(x)$, and recalling $u_t=-\frac{\sqrt \pi}{2}  \rho_t(0)$, 
Lemmas~\ref{AbelMarchaud}--\ref{AbelMarchaud3} 
imply that
there exist $r>0$, $q>1$, and a  function $\tilde g\in H^r_q[-2,2]$ 
such that $	\Psi^\mathrm{rsp}_\phi(1/2, z)$ can be written as
(using the\footnote{The  improper integral is well-defined and finite, since $\phi$ is $C^1$ and 
$[-a_t, c_1]$ contains a neighbourhood of each $c_{\ell}$.}
Hilbert transform \eqref{Hilb})
\begin{align*}
\sum_{j=0}^\infty z^j    \bigl [\int \phi(f^j_t(x))  \, \tilde g(x) \D x  +\sum_{k\ge 1}
 \frac{ u_t \cdot s_{k-1}}{\sqrt{|Df^{k-1}(c_1)|}} \bigl (\phi(c_{k+j}) +\HH(\mathbf{1}_{I_t}\cdot ( \phi\circ f^j_t) )(c_k )\bigr ) \bigr ]\, .
\end{align*}
(Indeed,  the Heaviside function and the logarithm from
Lemma~\ref{AbelMarchaud2}, the $1/2$-H\"older contribution from Lemma~\ref{AbelMarchaud3}, and --- using the MT assumption --- the functions $\mathbf{1}_{\supp(\phi\circ f^j_t) \setminus I_t}(x-c_k)^{-1}$
have uniformly bounded $H^r_q[-2,2]$ norms, for  $r>0$, $q>1$.) 

Next, recalling Proposition~\ref{CC2}, set
$$
\VV^{reg}_{\phi,1/2}(z)=\VV^\mathrm{rsp}_{\phi,1/2}(z)
+
\sum_{j=0}^\infty z^j    \int  \phi(f^j_t(x))\,    \tilde g(x) \D x  \, .
$$
Since  $\int \phi \, \D \mu_{t}=0$, Lemma~\ref{A3} gives $\kappa<1$, independent of $\phi$, such that the function $\VV^{reg}_{\phi,1/2}(z)$
is holomorphic in the disc of radius $\kappa^{-1}>1$.

\smallskip

The rest of the proof is devoted  to the study
of the singular term of the susceptibility function, that is the formal power series
  \begin{align}
\nonumber \Psi^{sing}_\phi(1/2, z)&:= 
\sum_{j=0}^\infty z^j  \sum_{k\ge 1}
 \frac{ u_t \cdot s_{k-1}}{\sqrt{|Df^{k-1}(c_1)|}}\bigl  (\phi(c_{k+j}) +
\HH(\mathbf{1}_{I_t}\cdot(\phi\circ f^j_t))(c_k)\bigr )\, .
\end{align}

We first concentrate on the contribution of $\phi(c_{k+j})$, which can be rewritten as
\begin{align}
\label{rem} 
\Psi^{sing,0}_\phi(1/2, z):=u_t   \sum_{\ell=1}^\infty \phi(c_{\ell}) 
 \sum_{j=0}^{\ell-1} z^j \frac{ s_{\ell-j-1} }{\sqrt{|Df^{\ell-j-1}(c_1)|}} 
 \, .
 \end{align}
 Following the arguments of \cite[App. B, Remark 1.2]{BMS} (see also \cite[\S5]{Ba} and the proof of \cite[Prop 4.6]{BS0}),
 and recalling that our choices imply $X(c_\ell)\equiv 1\equiv v(c_\ell)$ for all $\ell$,  we introduce for every $\ell \ge 1$ the formal Laurent series (recalling \eqref{sigmakell})
\begin{equation*}
 \alpha_{1/2}(c_\ell,z)
 =-\sum_{k=1}^\infty z^{-k} \frac{s_{k,\ell} }{ \sqrt{|Df^k(c_\ell)|}}\, .
 \end{equation*} 
Our assumptions imply that $\alpha_{1/2}(c_\ell, \cdot)$ is rational and  that it is holomorphic 
 in $|z|> 1/\sqrt \lambda_c$.
 Recalling the definition \eqref{defUU} of $\UU_{1/2}$,  
the coefficient of $\phi(c_\ell)$  in \eqref{rem} is 
 \begin{align}
\nonumber& u_t z^{\ell-1}  \sum_{j=0}^{\ell-1} z^{-(\ell-1-j)}
  \frac{ s_{\ell-1-j}}{\sqrt{|Df^{\ell-1-j}(c_1)|}}
 \\
 \label{simple}&\qquad = z^{\ell-1} 
 \biggl ( \UU_{ 1/2}(z)
  -u_t \frac{s_{\ell-1}}{z^{\ell-1} \sqrt{Df^{\ell-1} (c_1)}} \sum_{k=1}^\infty \frac{s_{k,\ell} }{z^k \sqrt{|Df^k(c_\ell)|}} \biggr ) \,  .
  \end{align}
Thus,
we find 
$$
 \Psi^{sing,0}_\phi(1/2, z)=
 \UU_{ 1/2}(z)\sum_{\ell=1}^\infty \phi(c_{\ell})z^{\ell-1}
  -  u_t \sum_{\ell=1}^\infty \frac{\phi(c_{\ell})s_{\ell-1}\alpha_{1/2}(c_\ell,z)}{\sqrt{|Df^{\ell-1}(c_1)|}}
  \, . 
$$
Next,  our MT assumption implies that the function
$$
 \VV^{sing,0}_{\phi, 1/2}(z):=- u_t \sum_{\ell=1}^\infty \frac{\phi(c_{\ell})s_{\ell-1}\alpha_{1/2}(c_\ell,z)}{\sqrt{|Df^{\ell-1}(c_1)|}} 
$$
is  rational, and that it is holomorphic 
 in the domain  $\{|z|> 1/\sqrt \lambda_c\}$.

It remains to consider  the contribution of $\HH (\mathbf{1}_{I_t}\cdot(\phi \circ f^j_t))(c_k)$, that is,
\begin{align}
\nonumber
\Psi^{sing,1}_\phi(1/2, z):=-\frac{u_t}{\pi} 
\sum_{j=0}^\infty z^j  \sum_{k\ge 1}
 \frac{ s_{k-1}}{\sqrt{|Df^{k-1}(c_1)|}} \int_{I_t}
(\phi \circ f^j_{t}) \, {\chi_k} \, \D m
 \, ,
\end{align}
with $\chi_k$ from \eqref{defchik}.  

Using 
 the notation introduced in Remark~\eqref{notC}, 
and introducing  $\MM_t: \YY_t \to H^r_q$  by
\begin{equation*}
\MM_t(\vec Y)= Y_{L+P-1} \tilde \chi_L + \sum_{k=1}^{L+P-2}  Y_k \tilde \chi_{k+1}\, ,
\end{equation*} 
 Lemma~\ref{A2} allows us to   write, 
setting
$\vec Y_k = (\delta_{j,k})_{j=1, \ldots L+P-1}\in \{0,1\}^{L+P-1}$,
\begin{align}
\nonumber
\sum_{j=0}^\infty z^j  \LL^j_{t} (\chi_k)
&=\sum_{j=0}^\infty z^j  \mathbb S^j_t (\vec Y_k)
+\sum_{\ell =0}^\infty z^\ell \LL_t^\ell
\MM_t \sum_{n=0}^\infty z^n \mathbb S_t^n(\vec Y_k)\\
\label{AB}&=:  \sum_{j=0}^\infty z^j A_j(\vec Y_k) +\sum_{j =0}^\infty z^j
B_j(\vec Y_k)
\, .
\end{align}
Since $\int_{I_t}
(\phi \circ f^j_{t}) \, \chi_k \, \D m=\int_{I_t}
\phi \,  \LL^j_t( \chi_k )\, \D m$,
using $A_j$ and $B_j$ from \eqref{AB},
we write $\Psi^{sing,1}_\phi(1/2,z)$ as
\begin{equation}
\label{decomp0}-\frac {u_t} \pi \sum_{j=0}^\infty z^j \sum_{k\ge 1}
 \frac{ s_{k-1}}{\sqrt{|Df^{k-1}(c_1)|}} 
\int_{I_t} \phi(x)\biggl [ A_j ( \vec Y_k )
+ B_j  (\vec Y_k)
\biggr ] \D m \, .
\end{equation}
We start with the terms for the $A_j$. Applying Lemma~\ref{A2} ($j$ times), we find,
\begin{align*}
 -\frac 1 \pi \int_{I_t} \phi(x) A_j (\vec Y_k)\, 
\D m  &=  -\frac{1}{\pi}
 s_{j,k}
\int_{I_t} \phi(x) \chi_{k+j}(x)\,  \D x  \, .
\end{align*}
Therefore, since $s_{k-1}\cdot s_{j,k}=s_{k+j}$,
proceeding as for \eqref{simple}, but with the signs removed, 
the contribution of the $A_j$ terms in \eqref{decomp0} give 
\begin{align*}
-\frac{u_t}{\pi}\sum_{j=0}^\infty z^j \sum_{k\ge 1}
 \frac{1 }{\sqrt{|Df^{k-1}(c_1)|}} &
\int_{I_t} \phi(x) \cdot s_{k+j} \chi_{k+j}(x)\D x 
\\\nonumber &
=\UU^+_{ 1/2}(z) \Sigma^\HH_\phi(z)+\VV^{sing,1}_{\phi,1/2}(z)\, ,
\end{align*}
with $\VV^{sing,1}_{\phi,1/2}(z)$ rational and holomorphic outside of the disc of
radius $\sqrt \lambda_c^{-1}$.

Next, we  analyse the contribution of the $B_j$  in \eqref{decomp0}.
Using $s_{k-1}\cdot s_{n,k}=s_{k+n}$, proceeding as for \eqref{simple} with the
signs removed,  using $\int \phi \, \D \mu_t=0$, 
and setting $\vec \Upsilon_t:=-(u_t/\pi)\cdot \sum_{k\ge 1}
 \frac{s_{k-1}}{\sqrt{|Df^{k-1}(c_1)|}} \vec Y_k$, we find\footnote{We identify $\chi_k=\chi(\vec Y_k)$ with $\vec Y_k$
in \eqref{pre2}.}
\begin{align}
-\frac{u_t}{\pi}\nonumber \sum_{j=0}^\infty & \sum_{k= 1}^\infty
 \frac{ s_{k-1}\, z^j}{\sqrt{|Df^{k-1}(c_1)|}} 
 \int_{I_t} \phi \cdot  B_j   (\vec Y_k)
\,\D m  
\\ \nonumber &
 =   
 \sum_{\ell =0}^\infty z^\ell \int_{I_t} \phi  \cdot
\LL_t^\ell
\MM_t \sum_{n=0}^\infty z^n \mathbb S_t^n (\vec \Upsilon_t) \, \D m
\\ \nonumber&
=   \sum_{\ell =0}^\infty z^\ell\int_{I_t} (\phi \circ f^\ell_t)\,  \MM_t \sum_{n=0}^\infty z^n \mathbb S_t^n  (\vec \Upsilon_t) \,\D m\, 
\\
\label{pre2}&
=-\frac{u_t}{\pi}\cdot 
\sum_{\ell=0}^\infty z^\ell \int_{I_t}
(\phi \circ f_t^\ell)  \cdot 
\sum_{n=0}^\infty z^n  \sum_{k=1}^\infty 
\frac{ s_{k-1} s_{n, k}}{\sqrt{|Df^{k-1}(c_1)|}} \MM_t(\chi_{k+n}) \, \D m
\\
\nonumber&=-\frac{1}{\pi}\cdot \sum_{\ell=0}^\infty z^\ell \int_{I_t}
(\phi \circ f_t^\ell)  \, \bigl ( \UU^+_{ 1/2}(z) \Sigma_t^{\tilde\psi_t}(z)+\VV^{sing,2}_{1/2}(z)\bigr ) \, \D m \, ,
\end{align}
with 
\begin{equation}\label{defphitilde0}
\tilde \psi_t(\ell)=\tilde \chi_{\ell+1}-\rho_t \int_{I_t} \tilde \chi_{\ell+1}\D m\, ,
\end{equation}
 and where $z\mapsto \VV^{sing,2}_{1/2}(z)\in H^r_q[-2,2]$ is rational, and it is
holomorphic outside of the disc of
radius $\sqrt \lambda_c^{-1}$. 
 Hence,
using again $\int \phi \D \mu_t=0$ and Lemma~\ref{A3},
$$
\VV_{\phi,1/2}(z):=\VV^{reg}_{\phi,1/2}(z)+\VV^{sing,0}_{\phi,1/2}(z)+\VV^{sing,1}_{\phi,1/2}(z)
-\frac{1}{\pi}\cdot \sum_{\ell=0}^\infty z^\ell \int_{I_t}
(\phi \circ f_t^\ell) \,  \VV^{sing,2}_{1/2}(z)\,\D m
$$
 is holomorphic in the annulus
$\{\lambda_c^{-1/2}<|z|<\kappa^{-1}\}$.

Finally, the formulas \eqref{lastcorclaim2}
and \eqref{valW}  for  the residues  
follow from Lemma~\ref{prelim}. In particular,
 since $\int \phi \D \mu_t=0$, using Lemma~\ref{A3}, we may take 
\begin{equation}
\label{defphitilde}
\int \phi \, \tilde \psi^*_t \, \D m
:= -\frac 1 \pi \sum_{j=0}^\infty z^j \int_{I_t}
(\phi \circ f_t^j)  \sum_{\ell=L}^{L+P-1} s_\ell \, \tilde \psi_t(\ell)  \,   \D m
\, .
\end{equation}
\end{proof}

\begin{proof}[Proof of Lemma~\ref{prelim}]
If $t\in \MT$, then $\UU_{1/2}(z)$  is the rational function
\begin{align*}
&\UU_{1/2}(z)=\frac{u_t }{z^{L-1}}\biggl (\sum_{\ell=0}^{L-1} \frac{s_\ell z^{L-1-\ell}}{ \sqrt{|Df_t^\ell(c_1)|}}
+\sum_{\ell=L}^{L+P-1} \frac{s_\ell z^{L-1-\ell}}{ \sqrt{|Df_t^\ell(c_1)|}}
\sum_{k=0}^\infty 
\frac {s_{kp,\ell}}{z^{kp}\sqrt {|Df_t^{kp}(c_L)| }}  \biggr )\\
&\qquad=\frac{u_t }{z^{L-1}}\biggl (\sum_{\ell=0}^{L-1} \frac{s_\ell z^{L-1-\ell}}{ \sqrt{|Df_t^\ell(c_1)|}}
+\sum_{\ell=L}^{L+P-1} \frac{s_\ell z^{L+P-1-\ell}}{ \sqrt{|Df_t^\ell(c_1)|}}
\frac {1}{z^P-\frac{s_{P,L}}{\sqrt {|Df_t^{P}(c_L)| }}}  \biggr ) \, .
\end{align*}
Similarly, $\UU^+_{1/2}(z)$  is the rational function
\begin{align*}
\UU^+_{1/2}(z)=\frac{u_t}{z^{L-1}}\biggl (\sum_{\ell=0}^{L-1} \frac{z^{L-1-\ell}}{ \sqrt{|Df_t^\ell(c_1)|}}
+\sum_{\ell=L}^{L+P-1} \frac{ z^{L+P-1-\ell}}{ \sqrt{|Df_t^\ell(c_1)|}}
\frac {1}{z^P-\frac{1}{\sqrt {|Df_t^{P}(c_L)| }}}  \biggr ) \, .
\end{align*}

 We show that $\Sigma_\phi(z)$ and $\Sigma^\HH_\phi(z)$ are
 rational, with possible poles
at the $P$th roots of unity for $\Sigma_\phi(z)$, and at  the $P$th roots of $\sgn (Df^P(c_L)$ for $\Sigma^\HH_\phi(z)$. Indeed,
$$
\Sigma_\phi(z)= \sum_{\ell=1}^\infty \phi(c_{\ell})z^{\ell-1}=
\sum_{\ell=1}^{L-1} \phi(c_{\ell})z^{\ell-1}
+ \frac{z^{L-1}}{1-z^P} \sum_{\ell=0}^{P-1} \phi(c_{L+\ell})z^{\ell}\, .
$$  
The  residue of $\Sigma_\phi(z)$ at $1$  is thus
$\frac{1}{P}\sum_{\ell=L}^{L+P-1}  \phi(c_\ell)$.
If $\sgn(Df^P(c_L))=+1$ then
$$
\Sigma^\HH_\phi(z)
=\sum_{\ell=1}^{L-1} s_\ell \, \HH(\mathbf{1}_{I_t}\phi)(c_{\ell})z^{\ell-1}
+ \frac{z^{L-1}}{1-z^P} \sum_{\ell=0}^{P-1} s_{L+\ell}\,  \HH(\mathbf{1}_{I_t}\phi)(c_{L+\ell})z^{\ell}\, , 
$$
in which case the residue at $1$ is $\frac{1}{P}\sum_{\ell=L}^{L+P-1}   s_\ell \HH(\mathbf{1}_{I_t}\phi)(c_{\ell})$.
If $\sgn(Df^P(c_L))=-1$ 
$$
\Sigma^\HH_\phi(z)
=\sum_{\ell=1}^{L-1} s_\ell \, \HH(\mathbf{1}_{I_t}\phi)(c_{\ell})z^{\ell-1}
+ \frac{z^{L-1}}{1+z^P} \sum_{\ell=0}^{P-1} s_{L+\ell}\,  \HH(\mathbf{1}_{I_t}\phi)(c_{L+\ell})z^{\ell}\, .
$$
The same argument gives that $\Sigma_t^{\tilde \psi}(z)$ is rational, with possible poles
at the $P$th roots of  of $\sgn (Df^P(c_L))$, and, when $\sgn (Df^P(c_L))=1$,
its residue at $z=1$ is equal to
$\frac{1}{P}\sum_{\ell=L}^{L+P-1}   s_\ell \,  \tilde \psi_\ell$.
\end{proof}

 
\section{One-half transversality: Numerics and Conjecture~\ref{conj2}}
\label{numerics}
Sums $\JJ_\eta(t)$ of the form \eqref{JMT}  with $\eta=1/2$ 
play an important role in our study of the fractional  response in the quadratic family at a Misiurewicz--Thurston (MT) parameter $t$. In particular the ``one half transversality condition'' condition $\JJ_{1/2}(t)\neq 0$ is essential in  Theorem~\ref{WTF}.

The sums \eqref{JMT} already appeared in the literature:
For $\eta=1$, we recover the Tsujii transversality condition $\JJ_1(t)\neq  0$ (see Tsujii \cite{Tsu}), which is  satisfied for {\it every} MT  parameter in the quadratic family (and in a far larger class of parameters, see Footnote~\ref{foot1}).
Figure~\ref{Fig1a} illustrates the graph of $\JJ_1(t)$
over hundreds of MT parameters. (We explain in 
\S\ref{algo'}, how these
MT parameters were obtained.)

\begin{figure}
	\includegraphics[scale=0.7]{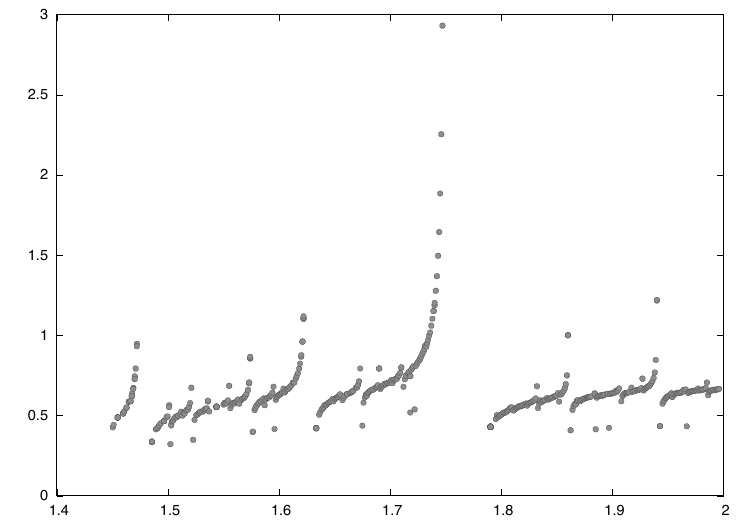}
	\caption{ $\JJ_{1}(t)$ for Misiurewicz--Thurston (MT) parameters $t$.}\label{Fig1a}
\end{figure}

For $\eta=1/2$,  the  set of $1/2$-summable parameters is important in the study of unimodal maps. Nowicki and van Strien \cite{nvs} proved (in particular) that quadratic maps that satisfy the $1/2$-summability condition  have an absolutely continuous invariant probability measure.   It turns out  that,  in the complement of the hyperbolic parameters,  {\it almost every parameter} satisfies the $1/2$-summability condition (see Lyubich \cite{lyub}, and also Martens and Nowicki \cite[\S 4]{mn}). 

However, the condition  $\JJ_{1/2}(t)\ne 0$ does not seem to have appeared in the literature.  The reader may wonder when this condition holds. We do not have a definitive answer for this. As observed after the statement of Theorem~\ref{WTF},
it is easy to see that $\JJ_{1/2}(2)=0$. This  first came as a surprise to us, but
it is in fact natural, as we explain next.

We already noticed that the piecewise expanding and piecewise analytic map $F_t$
conjugated to $f_t$ via the change of variable $\Lambda=\Lambda_t$ given by its invariant densities (see \eqref{magicOg})
$(D_t F_t)^k(\Lambda(c_{1,t})_\pm)=s_k \sqrt{ |(Df_t^k)(c_{1,t})|}$ 
we see that $\JJ_{1/2}(t)\ne 0$ is just the ordinary  transversality assumption \cite{BS0}
of $F_t$ for the vector field $v\equiv 1$ on $[F_t(\Lambda(c_{1,t})), \Lambda(c_{1,t})]$.
Noting that $F_2$ is just the full tent map with slopes $\pm 2$,
the fact that  $\JJ_{1/2}(t)= 0$ for the full quadratic map 
mirrors the fact that the family of tent maps $\tilde F_t$
is tangential\footnote{The topological entropy is  the logarithm
of the slope and thus constant, so there are no bifurcations. It is  illuminating  to construct explicitly the corresponding topological
conjugacy.}
so that $F_2$ is tangential  for the vector field $v\equiv 1$.

Note also that the measure of maximal entropy of $f_2$ coincides with
the absolutely continuous measure. See \cite[\S 8, \S 9]{Mis81} for 
classical necessary conditions
for this property to hold.
More recently, based on \cite[Theorem 2]{Do}, Dobbs and Mihalache observed
\cite[Fact 5.2]{DM} 
that 
the measure of maximal entropy of  an S-unimodal map $f$ with positive entropy is absolutely continuous if and only
if $f$ is\footnote{A unimodal map $f$ is called pre-Chebyshev if  $f$ is
exactly $m$ times renormalisable, for some $m \ge  0$,  each renormalisation
being of period two, and, in addition, if   $J$ is the restrictive interval for the $m$th renormalisation,
$f^{2^m}|_J : J \to J$ is smoothly conjugate on $J$ to $x \mapsto 1 - 2|x|$ on $(-1, 1)$.} pre-Chebyshev.
The (\cite[Proposition 5.1]{DM}  only pre-Chebyshev quadratic map  is
$f_2 : x \to 2-x^2$.

The parameter $t=2$ corresponds to the simplest combinatorics
$0\mapsto c_1 \mapsto c_2 \mapsto c_2$. One can check that $\JJ_{1/2}(t)\ne 0$ holds for the
parameter $t$ corresponding to the next simplest Misiurewicz--Thurston combinatorics (beware that it is not mixing)
$
0\mapsto c_1 \mapsto c_2 \mapsto c_3 \mapsto c_3
$.
Indeed,  this parameter is $t=1.54368\ldots$, and we have $f_{t}(c_{3})=c_{3}$
with $\lambda_1=f'_{t}(c_1)=-3.0874\ldots$,
$\lambda_2=Df_{t}(c_2)=-Df_{t}(c_3)=1.6786\ldots$,
so that a geometric series gives
$$\JJ^{1/2}(t)=
1-\frac{1}{\sqrt {|\lambda_1|}} -
\frac{1}{\sqrt{ |\lambda_1 \lambda_2|}}
\frac{1}{1+1/\sqrt{\lambda_2}}
=0.182959\ldots
\, .
$$

\subsection{Numerics}\label{algo'}
We have performed numerical experiments to investigate $\JJ_{1/2}(t)$
for hundreds of Misiurewicz--Thurston parameters: We calculate  858  MT parameters $t$, with high accuracy,   and we compute the corresponding sums $\JJ_{1/2}(t)$.

The algorithm consists into finding  approximate values  for Misiurewicz--Thurston parameters  in  the real line,   and then use the Milnor--Thurston transformation  to obtain such parameters  with higher precision.  Indeed, given a real Misiurewicz--Thurston parameter $t=c_1$  such that 
\begin{equation*}
f_{c_1}^{k+j}(0)=f_{c_1}^{k}(0)
\end{equation*}
for some $k\ge 1$ and $j\ge 1$, choose a point $x=(x_1,x_2, \dots, x_{k+j-1})\in \mathbb{R}^{k+j-1}$ such that  $x_i\cdot f^i_{c_1}(0) > 0$ for every $1\le  i< j+k$.   Next, define
$$T(x_1,x_2, \dots, x_{k+j-1})=(y_1,y_2, \dots, y_{k+j-1})$$
where $f_{x_1}(y_i)=x_{i+1}$ for $i <  k+j-2$, while $f_{x_1}(y_{k+j-1})=x_{k}$, and $y_i \cdot  f^i_{c_1}(0) > 0$ for every $0< i< j+k$. Then Milnor and Thurston \cite[Proof of Lemma 13.4]{MT} proved that 
$T^\ell(x)$ converges to $(c_1,f_{c_1}(c_1), \dots,f_{c_1}^{k+j-1}(c_1))$ exponentially fast. 

\begin{figure}
	\includegraphics[scale=0.28]{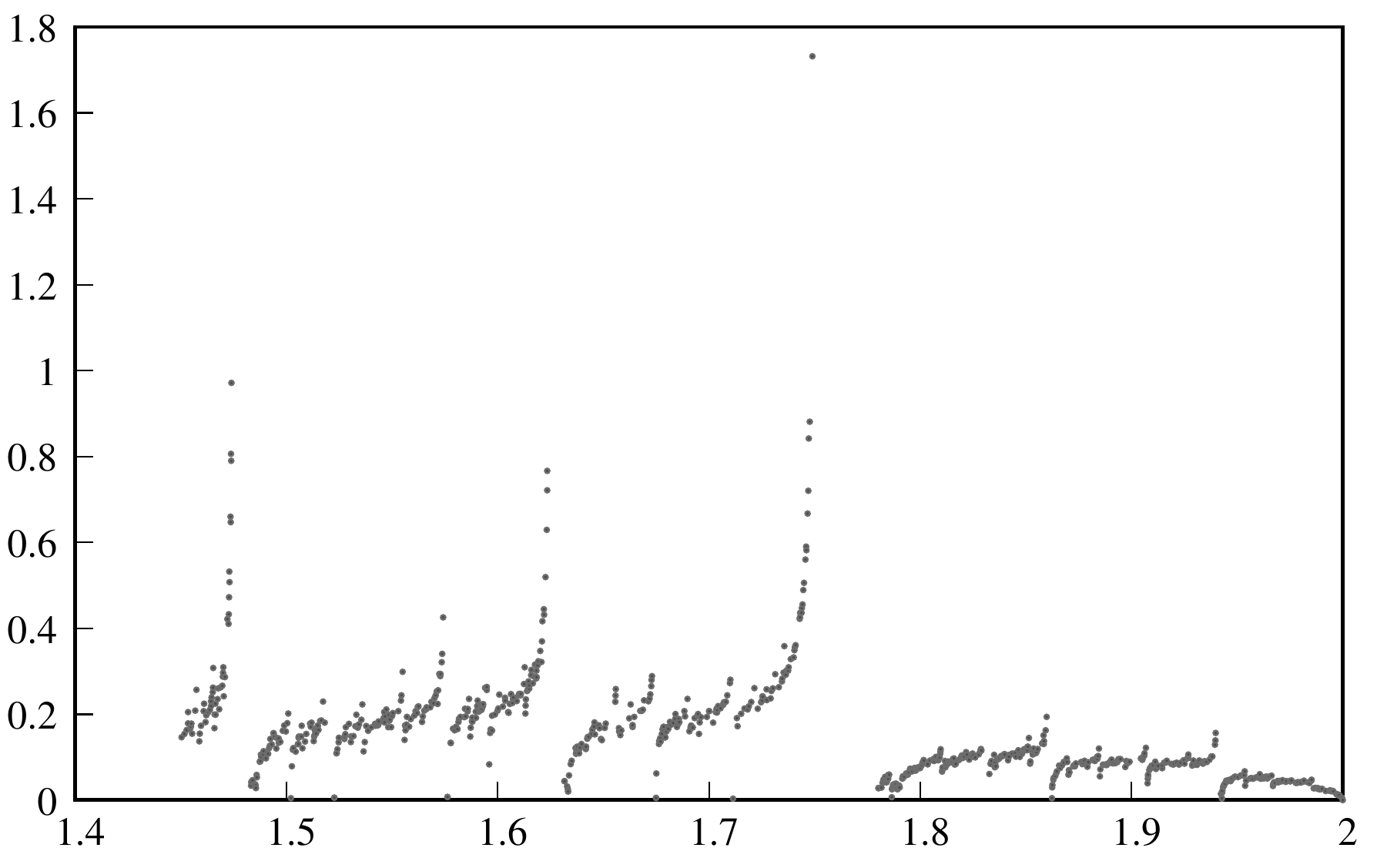}
	\caption{ $\JJ_{1/2}(t)$ for MT parameters $t$.}\label{Fig1}
\end{figure}

\smallskip
We explain next why some of the Misiurewicz--Thurston parameters found by this algorithm are not renormalizable, and hence mixing:
 The critical point of a Misiurewicz--Thurston  map $f$ is not periodic, and  there is $L > 0$ such that $f^L(0)$ is periodic. 
Taking $L$  minimal with this property, let   $P\geq 1$ be such that  $f^{P}(f^L(0))=f^L(0)$.  Suppose that
 $P$ is a prime number, $P\ne 2$, the multiplier $Df^P(f^L(0))$ is positive, and $f$ is renormalizable.
Then    the period of the first (and only) renormalization of $f$  is $P$,
 $$
F=f^P\colon [|f^L(0)|,-|f^L(0)|]\rightarrow [-|f^L(0)|,|f^L(0)|]$$
 is the first (and only) renormalization of  $f$, and, additionally, $F(f^L(0))=f^L(0)$, while $F(0)$ is not a fixed point of $F$,  and $F^2(0)=f^L(0)$. In particular $f^{2P}(f^L(0))=f^L(0)$. This implies $2P\geq  L > P$. Our numerical experiment give Misiurewicz--Thurston maps $f$
for which $P$ is a prime number, $P\ne 2$, the multiplier $Df^P(f^L(0))$ is positive, but  $2P\geq  L > P$ does not hold, so   that $f$ is not renormalizable.  Moreover the numerical experiment  gives $J_{1/2}(f)\neq 0$.

\smallskip

 The resulting graph for $\JJ_{1/2}(t)$ can be seen in Figure~\ref{Fig1}.
 (To be compared with Figure~\ref{Fig1a} 
 for the  graph of the Tsujii transversality condition $\JJ_{1}(t)$.) The value of $\JJ_{1/2}(t)$  seems to be always strictly positive except at $t=2$, where it vanishes. However, $\JJ_{1/2}(t)$  appears to be   close to zero (see Figure~\ref{Fig2} for a close-up) at a few values of $t$. 
The ``almost vanishing of $\JJ_{1/2}(t)$'' phenomenon seems  to occur when the real Misiurewicz--Thurston parameter $t$ is  such that $f_{t}$ is renormalisable with deepest (i.e., last) renormalisation  has\footnote{In particular, this deepest renormalisation is topologically conjugated to $f_2$.} topological entropy $\log 2$ (that is, there is a periodic point $-x_* \in \mathbb{R}$ with period  $n \ge 2$ such that 
the intervals $f_{t}^k[-x_*,x_*]$, for $k=0, \dots, n-1$,  are pairwise disjoint, except possibly at their boundaries, with $f_{t}^n[-x_*,x_*]\subset [-x_*,x_*]$, and 
 the unimodal map $g\colon [-1,1]\rightarrow [-1,1]$ defined 
	by $g(x)=x_*^{-1} f^n_{t}(x_*x)$ satisfies $g(-1)=g(1)=-1$ and $g(0)=x_*$).
Moreover it seems that this deepest  renormalisation is close to a quadratic polynomial on the interval of renormalisation. (This last property  happens when the so called ``complex bounds" are large enough. This  occurs for instance  in the first renormalisation  of parameters very close to $t=2$, see e.g. Douady and Hubbard \cite[Proof of Thm~5]{DH}.)
\begin{figure}
	\includegraphics[scale=0.32]{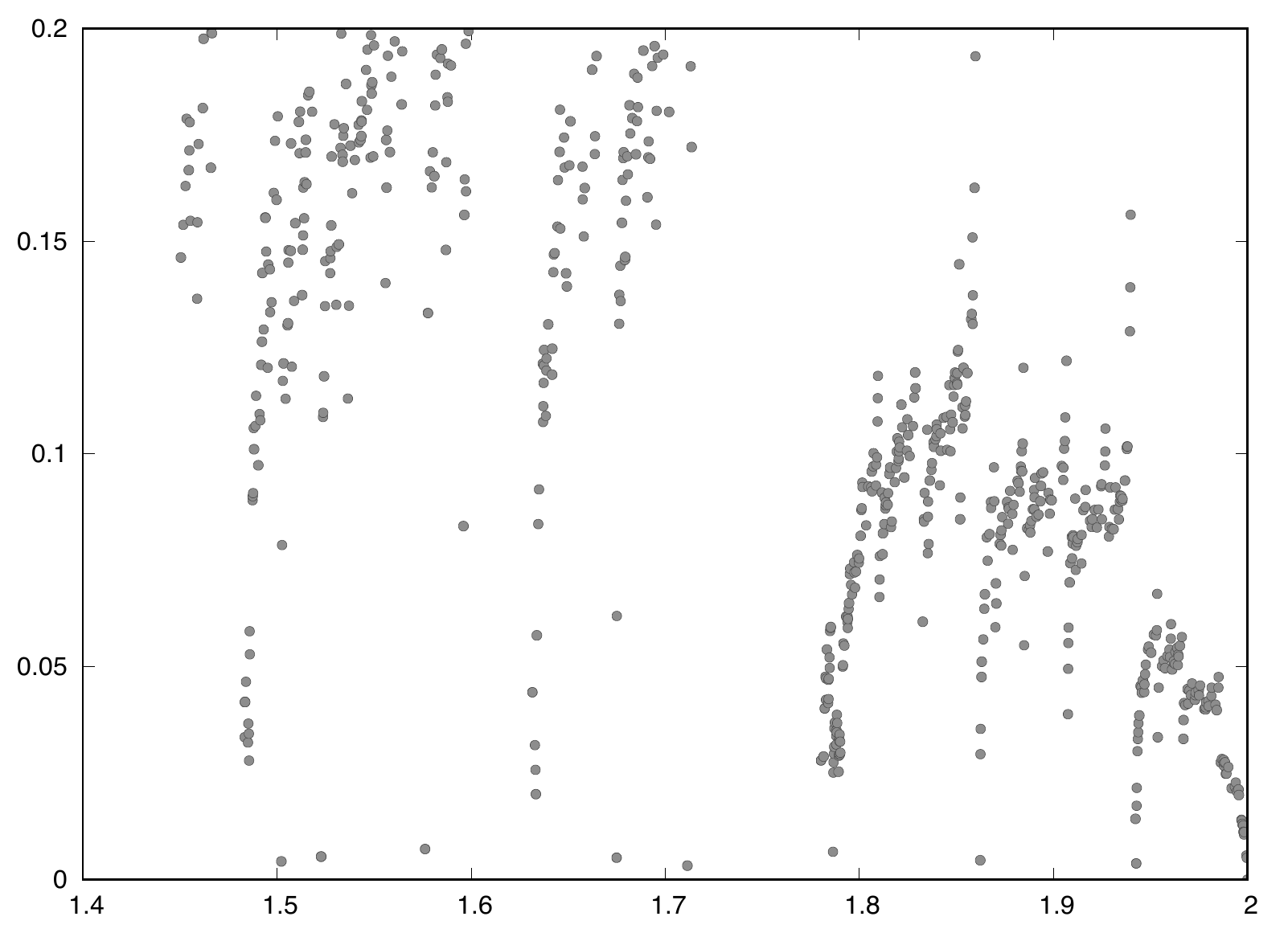}
	\caption{Close-up of $\JJ_{1/2}(t)$ for MT parameters $t$.}\label{Fig2}
\end{figure}

Note that if the deepest renormalisation
is  conjugated to the Ulam--von Neumann map, then \cite{BBS} does not give any
lower bound for the regularity of the SRB measure.

\subsection{Conjecture~\ref{conj2} on one-half-transversality for the quadratic family}
\label{conjB}
Our numerical experiments support and motivate the following conjecture.

 \begin{customconj}{B}\label{conj2}  For the quadratic family $f_t$, we have:
\begin{itemize}
 \item[i.]
For every real MT parameter $t\neq 2$, we have  $\JJ_{1/2}(t)> 0$.

\item[ii.] In fact, $\inf\,  \{\JJ_{1/2}(t) \mid t \text{ a real  MT parameter, } t\ne 2\}=0$.

\item[iii.] 
More generally,  if  $t\neq 2$ satisfies  $1/2$-summability, then  $\JJ_{1/2}(t)> 0$.

\item[iv.]  The parameter $\eta=1/2$ is a critical exponent in the following sense: If $\eta > 1/2$ then   $\JJ_{\eta}(t)> 0$ for {\it every} real MT parameter $t$.  If $\eta \in (0,1/2)$ there are infinitely many real parameters $t$ such that $\JJ_{\eta}(t) <  0$.
\end{itemize}
\end{customconj}

 G. Levin suggested that we perform experiments also for  parameters such that $f^P_{t}(c)=c$ for some
$P\ge 1$. For such $t$,  we set
\begin{equation*}
\JJ^{\mathrm{per}}_\eta(t)=\sum_{k=0}^{P-1}   \frac{\sgn ( (D f^k_{t}(c_1)))}{|D f^k_{t}(c_1)|^\eta}\, .
\end{equation*}
In view of the resulting data, which  is presented in Figure~\ref{Fig4}, we expect
that claim [i]  of Conjecture~\ref{conj2}  
also holds for all real periodic parameters.

\begin{figure}
	\includegraphics[scale=0.6]{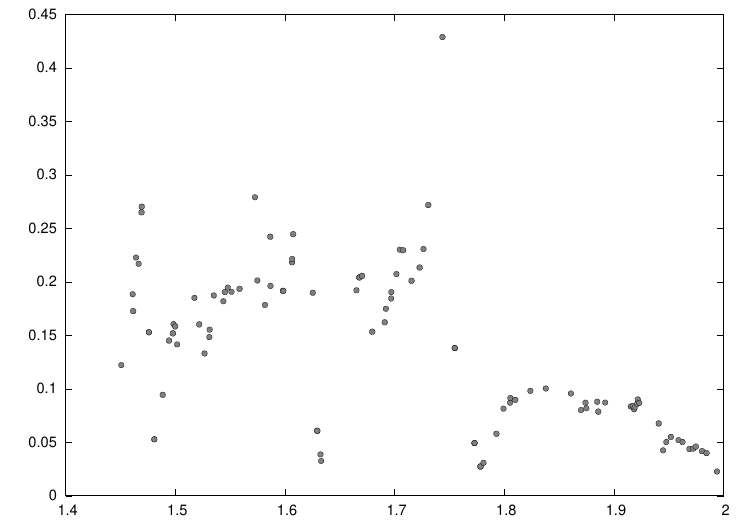}
	\caption{ $\JJ^{\mathrm{per}}_{1/2}(t)$ for periodic parameters $t$.}\label{Fig4}
\end{figure}

\smallskip

Finally, note that if $t>2$ then  $c_1>2$ and,  for all $k\ge 1$,  we have $f^{k+1}_{t}(c_1)>f^k_{t}(c_1)>2$, so that $|D f^{k+1}_{t}(c_1)|>|D f^k_{t}(c_1)|>4$,
while $\sgn( (D f^{k+1}_{t}(c_1))=-\sgn ((D f^k_{t}(c_1)))$. Thus,  the sum \eqref{JMT} converges absolutely for any $\eta>0$ and
$$
 \JJ_\eta(t)>1-\frac{1}{4^\eta}>0 \, , \, \forall t >2 \, , \, \forall \eta >0 \, .
$$
	
\section{Whitney	fractional integrals $I^{\eta,\Omega}$ and  derivatives  $M^{\eta,\Omega}$}
\label{Whitneyfrac}

\subsection{Abel's remark for Whitney fractional integrals $I^{1/2,\Omega }$
(Lemma~\ref{halff+omega})}
\label{bonusOmega} For  
$\Omega\subset (1,2)$ satisfying \eqref{tscond0} and  $t\in \Omega$,
it is natural to consider the one-sided {\it $\Omega$-(Whitney--)Riemann--Liouville fractional integrals of $\phi\in L^1$}
on $\Omega$ defined by
\begin{align*}
	(I^{\eta,\Omega}_{+} \phi) (t)
	&=\frac{1}{ \Gamma(\eta)}
	\int_{\Omega \cap (-\infty,t]}
	\frac {\phi(\tau)}{(t-\tau)^{1- \eta}}  \, \D \tau\, , \\
	(I^{\eta,\Omega}_{-} \phi) (t)
	&=\frac{1}{ \Gamma(\eta)}
	\int_{\Omega \cap [t,\infty)}
	\frac {\phi(\tau)}{(\tau-t)^{1- \eta}}  \, \D \tau \,  ,
\end{align*}
and the two-sided corresponding object defined by
\begin{equation*}
I^{\eta,\Omega} \phi(t)
=\frac{1}{2 \Gamma(\eta)\cos (\eta \pi/2)}
\int_{ \Omega}
\frac {\phi(\tau)}{|t-\tau|^{1- \eta}}  \, \D \tau\, .
\end{equation*}

Recalling the spikes $\phi_{c_k,\pm}$ from
\eqref{lrspike}, we give an
 analogue of Lemma~\ref{halff+}.
	
\begin{customlemma}{E}[Abel's remark on $\Omega$]\label{halff+omega}
Let  $\Omega\subset (1,2)$ be a compact positive
Lebesgue measure set satisfying Tsujii's property \eqref{tscond0} for all
$\beta <2$.
For any $k\ge 1$,
the one-sided $\Omega$-Riemann--Liouville half integrals
 of the square root spikes   
satisfy
\begin{align*}
I^{1/2,\Omega}_{-,t} (\phi_{c_k,+}) (x,t)&= A^\Omega_{c_k,+}(x,t) 
+\sqrt \pi \cdot 
\mathbf{1}_{x > c_k+t}(x)\, ,\,\,
\\
I^{1/2,\Omega}_{+,t} (\phi_{c_k,- })(x,t)&= A^\Omega_{c_k,-}(x,t) +
\sqrt \pi \cdot \mathbf{1}_{x < c_k+t}(x)\, ,
\end{align*}
where 
$A^\Omega_{c_k,\pm}(x,t)\le 0$ 
are defined by
$A^\Omega_{c_k,\sigma}(x,t)=0$  if     $\sigma x \le \sigma (c_k+t)$, and
$$
A^\Omega_{c_k,\sigma}(x,t) = -
\frac{1}{\Gamma(1/2)}
\int_0^1 \frac{\mathbf{1}_{\Omega^c}(t+\sigma(x-(c_k+t))u)}{u^{1/2} (1-u)^{1/2}} \, \D u \mbox{  if  }   \sigma x > \sigma (c_k+t)  \, .
$$
In addition,  $x\mapsto A^\Omega_{c_k,\sigma}(x, t)$ is $\eta$ H\"older,
for $\sigma=\pm$ and for all $\eta <1/2$.
\end{customlemma}

We refrain from stating  the $\Omega$ version of the two-sided Lemma~\ref{halff}. (Since \eqref{195.01} cannot be used,
the proof  must be more ``hands on.'' See also Remark~\ref{wildd}.)

\begin{proof}[Proof of Lemma~\ref{halff+omega}]
We handle $\phi_{c_k,+}(x,t)$. The case of $\phi_{c_k,-}(x,t)$  is symmetric.
The Whitney half  integral $I^{1/2,\Omega}_{-,t}$ of the  spike $\phi_{c_k,+}(x,t)$ with respect to $t$ is 
\begin{align*} 
(I^{1/2,\Omega}_{ -,t}\phi_{c_k,+})(x,t)&= 
\frac{1}{\Gamma(1/2)} 
\int_{\tau \in [t,\infty]\cap \Omega}  \frac{\phi_{c_k,+}(x,\tau)}{ (\tau-t)^{1/2} }\, \D \tau\\
&= \frac{1}{\Gamma(1/2)}
\int_{\tau \in [t,\infty]\cap \Omega} \frac{ \mathbf{1}_{ y >  {c_k}+\tau}(x) }
{(x-{c_k}-\tau)^{1/2}(\tau -t)^{1/2}} \, \D \tau\\ 
&=\begin{cases} 0  &\text{  if }    c_k +t  \geq x,   \\  
\frac{1}{\Gamma(1/2)} \int_t^{x-{c_k}} \Big( \frac{\mathbf{1}_\Omega(\tau)}{(\tau-t)(x-{c_k}-\tau)} \Big)^{1/2} \, \D \tau 
&\text{  if }  {c_k}+t<x \, .   \end{cases}
\end{align*} 
If $c_k+t <x$, making the substitution $\tau = t+(x-c_k-t)u$, we get
\begin{align*} 
(I^{1/2,\Omega}_{-,t}\phi_{c_k,+})(x,t)&= \frac{1}{\Gamma(1/2)}
\int_0^1 \frac{\mathbf{1}_\Omega(t+(x-c_k-t)u)}{u^{1/2} (1-u)^{1/2}} \, \D u \, .
\end{align*}
Recalling \eqref{p14}, the function $A^{\Omega}_{c_k,-}(x,t):=(I^{1/2,t}_{\Omega, -}\phi_{c_k,+})(x,t)-\sqrt \pi
\mathbf{1}_{x>c_k+t}$ vanishes for $c_k+t \ge x$, while
for $x> c_k+t$ we have,  
\begin{align*} 
-\Gamma(1/2) \cdot A^{\Omega}_{c_k,-}(x,t)&=
\int_0^1 \frac{\mathbf{1}_{\Omega^c} (t+(x-c_k-t)u)}{ u^{1/2} (1-u)^{1/2}} \, \D u 
\, .
\end{align*}
Next, we show  that 
$A^{\Omega}_{c_k,-}(x,t)$ is $\eta$ H\"older for all $\eta <1/2$ at $c_k+t$.
Fixing $q\in (2,1/\eta)$  and $\tilde q<2$
such that $1/q+1/\tilde q=1$, first note that  $\|   u^{-1/2} (1-u)^{-1/2} \|_{L^{\tilde q}([0,1])}<\infty$.
Then, setting
 $\epsilon=x-c_k-t>0$,  we have
$$
-\Gamma(1/2) \cdot A^{\Omega}_{c_k,-}(x,t)\le 
\| \mathbf{1}_{\Omega^c} (t+ \epsilon u)\|_{L_q([0,1])}
\cdot  \|   u^{-1/2} (1-u)^{-1/2}\|_{L^{\tilde q}([0,1])} \, ,
$$
by the Holder inequality. 
Next, using \eqref{tscond0}, we find for  any $\beta >2$,
\begin{align*} 
\| \mathbf{1}_{\Omega^c} (t+ \epsilon u)\|_{L_q([0,1])}
=\biggl (\int_0^\epsilon \frac{\mathbf{1}_{\Omega^c} (t+ v)}{\epsilon} \, \D v\biggr)^{1/q}
\le C_\beta \cdot \epsilon^{(\beta-1)/q}\, .
\end{align*}
By taking $\beta <2$ close enough to $2$ we may
ensure $(\beta-1)/q\ge \eta$.  Recalling that $x=\epsilon+c_k+t$, this proves that $A^{\Omega}_{c_k,-}(x,t)$ is $\eta$ H\"older  at $c_k+t$.
		
Finally, we prove that for any $\eta <1/2$ there exists $C_\eta <\infty$
such that
$$
|A^{\Omega}_{c_k,-}(x_2,t)- A^{\Omega}_{c_k,-}(x_1,t)|
\le C_\eta |x_2-x_1|^\eta \, , \quad \forall x_i > c_k+t\, , \, \, i=1, 2\, .
$$
For this, assuming without loss of generality that $x_1<x_2$, we have
\begin{align} 
\nonumber &\Gamma(1/2) |A^{\Omega}_{c_k,-}(x_2,t)- A^{\Omega}_{c_k,-}(x_1,t)|\\
\nonumber  &\quad=
\int_{t}^{x_2-c_k}
\frac{ \mathbf{1}_{\Omega}(\tau) }
{\sqrt{x_2-c_k-\tau} \sqrt{\tau -t}} \, \D \tau-
\int_{t}^{x_1-c_k}
\frac{ \mathbf{1}_{\Omega}(\tau) }
{\sqrt{x_1-c_k-\tau}\sqrt{\tau -t}} \, \D \tau\\
\label{rhs1} &\quad= \int_{x_1-c_k}^{x_2-c_k}
\frac{ \mathbf{1}_{\Omega}(\tau) }
{\sqrt{x_2-{c_k}-\tau}\sqrt{\tau -t}} \, \D \tau \\
\label{rhs2} &\qquad\qquad +
\int_{t}^{x_1-c_k}
\biggl ( \frac{ \mathbf{1}_{\Omega}(\tau) }
{\sqrt{x_2-{c_k}-\tau}\sqrt{\tau -t}}
-   \frac{ \mathbf{1}_{\Omega}(\tau) }
{\sqrt{x_1-{c_k}-\tau}\sqrt{\tau -t}} \biggr )  \, \D \tau \, .
\end{align}
Using the change of variable $\tau=(x_1-c_k) +(x_2-x_1) u$,
with $\D \tau = (x_2-x_1) \D u$, for the 
integral in \eqref{rhs1}, and the H\"older inequality 
for $1/\tilde q+1/q=1$ with $1<\tilde q<2$ and $2<q<1/\eta$, we find
\begin{align*} 
\int_{x_1-c_k}^{x_2-c_k}&
\frac{ \mathbf{1}_{\Omega}(\tau) }
{\sqrt{x_2-{c_k}-\tau}\sqrt{\tau -t}} \, \D \tau\\
&\le  \Biggl (  \int_{x_1-c_k}^{x_2-c_k}  (\mathbf{1}_{\Omega}(\tau) )^q \, \D \tau \Biggr ) ^{1/q}
\Biggl (   \int_{x_1-c_k}^{x_2-c_k}
\biggl (       \frac{ 1 }
{\sqrt{x_2-{c_k}-\tau}\sqrt{\tau -t}}      \biggr ) ^{\tilde q} 
\, \D u \Biggr ) ^{1/\tilde q}\\
&\le |x_1-x_2|^{1/q}
\Biggl (  \int_0^1 \biggl (\frac{1}
{ \sqrt {u+ (x_1-c_k-t)/(x_2-x_1)} \sqrt{1-u}}\biggr ) ^{\tilde q} \, \D u
\Biggr ) ^{1/\tilde q}\\
&\le 
|x_1-x_2|^{1/q}
\Biggl (  \int_0^1 \biggl (\frac{1}
{ \sqrt {u} \sqrt{1-u}} \, \D u\biggr ) ^{\tilde q} \Biggr ) ^{1/\tilde q}\le C_\eta |x_1-x_2|^{\eta} \, .
\end{align*}
It remains to estimate \eqref{rhs2}. We rewrite the integral as 
\begin{align*}
		\int_{t}^{x_1-c_k}
		\frac{ \mathbf{1}_\Omega(\tau) }  {\sqrt{\tau -t}}
		\frac{1}{\sqrt{x_1-c_k-t}}
		\biggl (   \frac{ \sqrt{x_1-c_k -\tau} }  {\sqrt{x_2-c_k-\tau}} -1  \biggr )  \, \D \tau\, .
\end{align*} 
Now, for any $1<\tilde q<2$, we have
$$     
		\frac{1}{\sqrt{\tau -t}\sqrt{x_1-c_k-t}}\in L^{\tilde q}([t,x_1-c_k])\, , 
		\mbox{ uniformly in } x_1 > c_k + t\, .
$$
So, by the H\"older inequality, setting $w=x_1-c_k$ and $\delta=x_2-x_1$, it suffices to take $2<q<1/\eta$, and $\tilde q<2$ with $1/\tilde q+1/q=1$, and estimate
\begin{align*}
		\int_t^{w}
		\biggl (  \sqrt{1- \frac{\delta}{w+\delta-\tau}} -1   \biggr )^q&  \, \D \tau
		\le
		\int_t^{w}
		\biggl ( {\frac{\delta}{w+\delta-\tau}} \biggr )^{q/2}
		\, \D \tau\\
		&  \le  \delta^{q/2}  \int_t^{w}
		\biggl ( {\frac{1}{w+\delta-\tau}} \biggr )^{q/2}   \, \D \tau\\
		&= \delta^{q/2}  \cdot \frac{1}{1-q/2}
		(w+\delta-\tau)^{1-q/2} \bigg |^w_{\tau=t }\, .
		\end{align*} 
		Finally, if $2<q<1/\eta$, we have
		$
		\bigl ( \delta^{q/2}\delta^{1-q/2} \bigr )^{1/q}\le  \delta^{1/q} \le |x_1-x_2|^\eta
		$.
\end{proof}

\subsection{The semifreddo function.  Proposition \ref{natomega}}\label{five}
	
Let $g$ be a $\gamma$-H\"older function defined on a closed
subset $\Omega\subset \real$ of positive Lebesgue measure.
Then, for any $\eta <\gamma$, by analogy with the notion of the derivative in the
sense of Whitney, we define\footnote{This definition is meaningful if $x$ is a point of $\Omega$ with nonzero Lebesgue
density, see also Proposition~\ref{natomega}. In our setting,
we may use the stronger
condition \eqref{tscond0}.} the left-sided {\it Whitney--Marchaud derivative of $g$}
on $\Omega$  to be
\begin{align*}
	(M^{\eta,\Omega}_{+} g) (x)
	&=\frac{\eta}{ \Gamma(1-\eta)}
	\int_{\Omega \cap (-\infty,x]}
	\frac {g(x)-g(y)}{(x-y)^{1+ \eta}}  \, \D y\\
	\nonumber &=\frac{\eta}{ \Gamma(1-\eta)}
	\int_{\Omega -x \cap (-\infty,0]}
	\frac {g(x)-g(x+\tau)}{|\tau|^{1+ \eta}}  \, \D \tau\, .
\end{align*}
We then define 
\begin{align*}
	(M^{\eta,\Omega}_{-} g) (x)
	&=\frac{\eta}{ \Gamma(1-\eta)}
	\int_{\Omega -x \cap [0,\infty)}
	\frac {g(x)-g(x+\tau)}{\tau^{1+ \eta}}  \, \D \tau\, .
\end{align*}
and 
\begin{align*}
	(M^{\eta,\Omega} g) (x)&=
\frac{\eta}{2\Gamma(1-\eta)}
	\int_{\Omega -x }
	\frac {g(x+\tau)-g(x)}{|\tau|^{1+ \eta}} \sgn(\tau) \, \D \tau\, .
\end{align*}

\begin{remark}[Boundary of $\Omega$]\label{wildd} Beware that integration by parts with respect
to the variable $t$ is problematic for $M^{\eta,\Omega}$
since $\partial \Omega$ is wild in our application.  (In particular, the analogue of   Proposition \ref{CC2} 
is not obvious.)
\end{remark}

\begin{remark}In view of Conjecture~\ref{laconj}, it is desirable to  prove  versions of 
Lemma~\ref{AbelMarchaud} (as well as Lemmas~\ref{AbelMarchaud2}
and ~\ref{AbelMarchaud3})
for $M^{1/2,\Omega}$, if
$\Omega$ is compact and satisfies Tsujii's condition \eqref{tscond0} for suitable $\beta$.
Although	  it seems possible to bypass the (problematic) integration by parts in $\tau$ in the proof of Lemma~\ref{AbelMarchaud} by using instead
an infinite Taylor series for $g(\tau)$, we refrain from including this analysis  here.
\end{remark}

We define yet another susceptibility function:  
\begin{definition}[Semifreddo fractional susceptibility function]
 The
 {\it semifredddo 
fractional susceptibility
function} at $t\in \TSR$ and along $\Omega$ is  the following formal power series
\begin{align*}
&\Psi^{\Omega, \mathrm{sf}}_\phi(\eta,z)
:=   
 \sum_{k=0}^\infty  z^k \int
 (\phi \circ f_{t}^k) (x) 
(M^{\eta,\Omega}_s\LL_s\rho_{t}(x)|_{s=t}) \,
\D x\, .
\end{align*}
\end{definition}
This function lies ``between'' 
$\Psi^\mathrm{fr}_\phi(\eta,z)$ and
$\Psi^\Omega_\phi(\eta,z)$ since $\Psi^{\Omega, \mathrm{sf}}_\phi(\eta,z)$ is
\begin{align*}
 \frac{\eta }{2\Gamma(1-\eta)}\sum_{k=0}^\infty  z^k \int  
 (\phi \circ f_{t}^k) (x) 
\int_{\real \cap (\Omega-t)} 
\frac{\bigl ((\LL_{t+ \delta} -\LL_{t})  \rho_{t}\bigr )(x)}{|\delta|^{1+\eta}} \sgn(\delta) \, \D \delta \,
\D x\, .
\end{align*}

The results in this section together with Theorem~\ref{WTF} motivate the statement
 on $\Psi^{\Omega,\mathrm{sf}}(\eta, z)$ in Remark~\ref{conjC}.
(In addition, we expect that Proposition~\ref{natomega} should allow to prove that
$
\lim_{\eta \uparrow 1} \Psi^{\Omega,\mathrm {fr}}_\phi(\eta, z)=\Psi_\phi(z) 
$, as formal power series.)

\smallskip

The following notion of differentiability seems to be relevant in our context:

\begin{definition}[$\Omega$-Whitney differentiability]\label{relevant}
Let $g$ be a  function defined on a closed
subset $\Omega\subset \real$.
We say that $g$ is $\Omega$-Whitney differentiable at $t\in \Omega$
if there exists  $g'_\Omega(t)\in \complex$ such that 
$$
	\lim_{\delta \to 0, t+\delta \in \Omega}
	\frac{g(t+\delta)-g(t)}{\delta}=g'_\Omega(t) \, .
$$
For $\zeta\in (0,1)$, we say that $g$ is $\Omega$-Whitney $\zeta$-differentiable at $t\in \Omega$
if there exists $g^{\zeta}_\Omega(t)\in\complex$, 
such that 
\begin{equation*}
	\lim_{\delta \to 0, t+\delta \in \Omega}
	\frac{g(t+\delta)-g(t)}{\sgn(\delta) {|\delta|^\zeta}} =g^{\zeta}_\Omega(t) \, .
\end{equation*}
\end{definition}

The following proposition shows that  $M^{\eta,\Omega}$
is naturally related to  $\Omega$-Whitney differentiability for large enough sets $\Omega$:
	
\begin{customprop}{F}[$M^{\eta,\Omega}$ and $\Omega$-Whitney differentiability]\label{natomega}
Let $\Omega\subset \real$ satisfy Tsujii's condition \eqref{tscond0} for some $\beta >1$.
Then for any bounded function $g: \real \to \complex$ which is $\Omega$-Whitney differentiable at 
$t\in \Omega$, we have
$
\lim_{\eta \uparrow 1}
(M^{\eta,\Omega} g) (t)=g'_\Omega(t)	
$.

Moreover,  for any $\zeta\in (0,1)$
and any bounded  $g: \real \to \real$ which is $\Omega$-Whitney $\zeta$-differentiable at 
$t\in \Omega$, we have,
$$
		\lim_{\eta \uparrow \zeta}
		\biggl ( \frac{\Gamma(1-\eta)}{\zeta\cdot  \Gamma(\zeta-\eta)}	(M^{\eta,\Omega} g) (t) \biggr ) =g^{\zeta}_\Omega(t)
		\, .
$$
\end{customprop}
	
\begin{proof}[Proof of Proposition~ \ref{natomega}]
Let us assume to fix ideas that $t=0$ and $g'_\Omega(0)\ge 0$.
		
We first prove that $\lim_{\eta \uparrow 1} M_+^{\eta,\Omega} g (0) = g_\Omega'(0)$
by showing that, for any $\epsilon >0$, we have
$	(1-\epsilon)  (g'(0) - \epsilon )\leq \liminf_{\eta \uparrow 1} M_+^{\eta,\Omega} g (0)	\leq \limsup_{\eta \uparrow 1} M_+^{\eta,\Omega} g (0) \leq g_\Omega'(0)+	\epsilon$.
		
First, since  $g'_\Omega(0) = \lim_{h \to 0, h \in \Omega} \frac{g(h) - g(0)}{h}$, we can find  $\delta > 0$ such that
\[
		- \epsilon \leq \frac{g(h) - g(0)}{h} - g'_\Omega(0) \leq \epsilon
		\, , \, \, \, \, \forall h \in \Omega\, , \,\, |h| \leq \delta \, .
\]
Then, since $0<\eta <1$, we can  write
\begin{align*}
		M_+^{\eta,\Omega} g (0) 
		&=
		\frac{\eta}{\Gamma (1-\eta)} \biggl( \int_{\Omega \cap [0,\delta]} \frac{g (0)
			- g(-t)}{t} \frac{1}{t^\eta} \, \D t +
		\int_{\Omega \cap [\delta,\infty)} \frac{g (0) - g(-t)}{t^{1+\eta}} \,
		\D t \biggr) \\ 
		&\leq \frac{\eta}{\Gamma (1-\eta)}
		\biggl( (g'_\Omega(0) + \epsilon) \int_0^\delta \frac{1}{t^\eta} \, \D t +
		\int_\delta^\infty \frac{2\sup|g|}{t^{1+\eta}} \,\D t \biggr) \\ & \leq
		\frac{\eta}{(1-\eta) \Gamma(1-\eta)}
		(g'_\Omega(0) + \epsilon)  \delta^{1-\eta}+  \frac{2}{\Gamma(1-\eta)} \sup|g|\, 
		\delta^{-\eta} \, .
\end{align*}
Using that 
\begin{equation}\label{change2}
		\begin{split}
		(1-\eta) \Gamma(1-\eta)=\Gamma(2-\eta) \mbox{ with }
		\Gamma (1) = 1\, ,  &\\\qquad \qquad \mbox{ while }
		\lim_{\eta \uparrow 1} \Gamma(1-\eta) =
		\infty 
		\mbox{ and } \lim_{\eta \uparrow 1} \delta^{1-\eta} = 1
		\, ,&
		\end{split}
\end{equation}
we find
	$\limsup_{\eta \uparrow 1} M_+^{\eta,\Omega} g (0) \leq g'_\Omega(0) + \epsilon$.

The proof that $\liminf_{\eta \uparrow 1} M_+^\eta g (0) \geq( g'_\Omega(0) - \epsilon )(1-\epsilon)$ is a bit trickier and will use \eqref{tscond0} for $\beta >1$.
By the above computation for the limsup, taking $\eta$ close enough to $1$
for fixed $\delta$ 
it suffices to show that
\begin{equation}\label{enuf}
		\liminf_{\eta \uparrow 1}
		\frac{\eta}{\Gamma (1-\eta)} \int_{\Omega \cap [0,\delta]} \frac{g (0)
			- g(-t)}{t^{1+\eta}}  \, \D t \ge g'_\Omega(0) - \epsilon \, .
\end{equation}
Since   
\begin{align*}
		\frac{\eta}{\Gamma (1-\eta)}  \int_{\Omega \cap [0,\delta]} \frac{g (0)
			- g(-t)}{t} \frac{1}{t^\eta} \, \D t 
		&\geq \frac{\eta}{\Gamma (1-\eta)} \int_{\Omega \cap [0,\delta]} 
		(g'_\Omega(0) - \epsilon) \frac{1}{t^\eta} \, \D t \, ,
\end{align*}
using again $(1-\eta) \Gamma(1-\eta)=\Gamma(2-\eta)$, the bound \eqref{enuf} 
reduces to
$$  
		\liminf_{\eta \uparrow 1}  \int_{\Omega \cap [0,\delta]}  \frac{1-\eta}{t^\eta} \, \D t \ge  1\, .
$$
We already know	that $\lim_{\eta \uparrow 1}  \int_{ [0,\delta]}  \frac{1-\eta}{t^\eta} \, \D t=\lim_{\eta \uparrow 1}\delta^{1-\eta} =1$, so it suffices to show
\begin{equation*}
		\lim_{\eta \uparrow 1}  \int_{\Omega^c \cap [0,\delta]}  \frac{1-\eta}{t^\eta} \, \D t =0 \, .
\end{equation*}
This will follow from the fact that
\begin{equation*}
		\lim_{\eta \uparrow 1}  \int_{\Omega^c \cap [0,\delta]}  \frac{1}{t^\eta} \, \D t <\infty  \, .
\end{equation*}
The above bound, i.e. uniform  Lebesgue integrability
of $F_\eta(\tau)=\mathbf{1}_{\Omega^c \cap [0,\delta]}(\tau)\cdot  \tau^{-\eta}$
as $\eta \uparrow  1$, will
follow from \eqref{tscond0} for $\beta >1$ at $t=0$.
Indeed,  for any fixed $1<\beta <2$, there exists $C_\beta<\infty$ such that
for all $t>\delta^{-\eta}$ we have
\begin{equation*}
		m\{\tau \in \Omega^c \cap [0,\delta]\mid \tau^{-\eta} >t \}
		=m\{\tau \in \Omega^c \cap [0,\delta]\mid \tau <t^{-1/\eta} \}
		\le C_\beta t^{-\beta/\eta}\, .
\end{equation*}
Observe next that 
$$\limsup_{\eta \uparrow 1} \int_{\delta^{-\eta}}^\infty C_\beta t^{-\beta/\eta} \, \D t <\infty
\mbox{ if } \beta >1\, ,
\mbox{ and }
\limsup_{\eta \uparrow 1} \delta \int_0^{\delta^{-\eta}}  \, \D t= 1
\, .
$$
To conclude, just apply the characterisation of 
Lebesgue integrability of a  nonnegative measurable function $F$ given
by $\int_0^\infty m(\tau\mid F(\tau) >t) \, \D t <\infty$, see e.g. \cite[\S1.5, p.~14, (2)]{LL}, and use $t \mapsto g(-t)$ to get 
$\lim_{\eta \uparrow 1} M_-^{\eta, \Omega} g (0) = - g'_\Omega(0)$.

The second claim of Proposition~\ref{natomega} follows from the
same 
replacing \eqref{change2}  by
$(\zeta-\eta) \Gamma(\zeta-\eta)=\Gamma(1+\zeta-\eta)$, with 
$\Gamma (1) = 1$, while 
$
\lim_{\eta \uparrow \zeta} \Gamma(\zeta-\eta) =\infty$
 and  $\lim_{\eta \uparrow \zeta} \delta^{\zeta-\eta} = 1$.
\end{proof}

\begin{appendix}

\section{Proof of Lemma~\ref{halff}. (Abel's remark, two-sided)}\label{Ahalff}

Recall that if $|x|<1$ then
\begin{equation}\label{Taylor}
\sqrt{1+x}=1+\sum_{n=1}^\infty b_n x^n \mbox{ with } b_n=\frac{(1/2)(-1/2)\cdots (1/2-n+1) }{n!} \, .
\end{equation}
(In particular $b_1=1/2$ and $b_2=-1/8$.)
We shall also use the fact \cite[195.01]{McMillan} that
for any real numbers $a>\tau$ and $b>\tau$
\begin{equation}\label{195.01}
\partial_\tau  \log \bigl ( |\sqrt{a-\tau} -\sqrt{b-\tau}|\bigr )
=\frac 1 2 \frac{1}{\sqrt{a-\tau}\sqrt{b-\tau}} \, .
\end{equation}

\begin{proof}[Proof of Lemma~\ref{halff}]
Set $x_0=x-c_k$. To prove the claim on $I^{1/2}_{+,t} (\phi_{c_k,+,\EE})(x,t)$, it is enough to show that, for any $c_k+t-\EE<x<c_k+t+\EE$, 
\begin{equation}\label{abovecl}
\sqrt{\pi}\cdot I^{1/2}_{+,t }(\phi_{c_k,+,\EE})(x,t)= 
-
 \log |x_0-t| + \log \EE+G_\EE(t-x_0)\,  ,
\end{equation}
where, for $y \in (-\EE/2,\EE/2)$,
$$
 G_\EE(y)=-2\log H_\EE(y)\, ,
\quad H_\EE(y)=\frac{\sqrt{1+\frac{y}{\EE}}-1}{y/{\EE}}>0\, .
$$
Indeed,  using \eqref{Taylor}, we have (the power series below are absolutely convergent)
\begin{align*}
& H_\EE(y)=1/2+
\sum_{j=1}^\infty b_{j+1} \left(\frac{y}{\EE}\right)^{j}
\, , \,\, H_\EE(0)=1/2\, ,\\
&\partial_y G_\EE(y)=-\frac{2}{H_\EE(y)}\cdot
\sum_{j=1}^\infty j \cdot b_{j+1} \frac{y^{j-1}}{\EE^{j}} 
\, ,
\\
&\partial^2_y G_\EE(y)=\frac{2}{(H_\EE(y))^2}\cdot
\sum_{j=1}^\infty j \cdot  b_{j+1} \frac{y^{j-1}}{\EE^{j}} 
-\frac{2}{H_\EE(y)}\cdot
\sum_{j=2}^\infty j (j-1)  b_{j+1} \frac{y^{j-2}}{\EE^{j}} 
\, .
\end{align*}
In particular, $\lim_{\EE\to \infty}\sup_{ y\in (-\EE/2,\EE/2)}
| \partial_y G_\EE(y)|=0$, and
$$
\sup_{\EE} \,\, \sup_{ y\in (-\EE/2,\EE/2)}
\max\{|  G_\EE(y)|,
| \partial_y G_\EE(y)|, 
| \partial^2_y G_\EE(y)|\} < \infty\, .
$$

\smallskip
We proceed to show \eqref{abovecl}.
From the definition \eqref{RLhalf} of $I^{1/2}_{+,t}$, we get
\begin{align*}
\sqrt{\pi} I^{1/2}_{+,t} \phi_{c_k,+,\EE} (x,t)
&=\int_{-\infty}^t\frac{\phi_{c_k,+,\EE}(x,\tau)}{\sqrt{t-\tau}} 
\, \D \tau
\\
\nonumber &=
\int_{x-c_k-\EE}^{\min(t,x-c_k)} \frac{1}{\sqrt{x-c_k-\tau} }\frac{1}{\sqrt{t-\tau}}\, \D \tau \, .
\end{align*}
Recalling that $x_0=x-c_k$, if $x<  c_k+t$, then  we find
\begin{equation*}
\sqrt{\pi} I^{1/2}_{+,t} \phi_{c_k,+,\EE} (x,t)=\int_{x_0-\EE}^{x_0} \frac{1}{\sqrt{x_0-\tau}} \frac{1}{\sqrt{t-\tau}}\, \D \tau \, .
\end{equation*}
(This term did not appear in Lemma~\ref{halff+}.)
Using \eqref{195.01}, we find
\begin{align*}
\int_{x_0-\EE}^{x_0}& \frac{1}{\sqrt{x_0-\tau}\sqrt{t-\tau}}\, \D \tau
=2\log( |\sqrt{x_0-\tau}-\sqrt{t-\tau}|) \Bigr |^{x_0}_{x_0-\EE}\\
&= \log (t-x_0) -2 \log (\sqrt{t-x_0+\EE}-\sqrt\EE) \, .
\end{align*}
If in addition $c_k+t-\EE<x$, then, by \eqref{Taylor} we find
\begin{align}
 \nonumber-2 &\log (\sqrt{t-x_0+\EE}-\sqrt\EE)
=
- 2 \log\bigl (\sqrt \EE \biggl (\sqrt{1+\frac{t-x_0}{\EE}}-1 \biggr ) \bigr )
\\
\nonumber &=
 - 2 \log\biggl ( \frac{t-x_0}{\sqrt \EE}
\sum_{n=1}^\infty b_n \left(\frac{t-x_0}{\EE}\right)^{n-1}\biggr )\\
 \label{similar}&=
+\log \EE - 2 \log (t-x_0)-2\log\biggl (1/2+
\sum_{j=1}^\infty b_{j+1} \left(\frac{t-x_0}{\EE}\right)^{j}\biggr )\, ,
\end{align}
and we have shown that if $c_k+t-\EE<x<c_k+t$ then
\begin{align}\label{one}
  \sqrt{\pi} I^{1/2}_{+, t} &\phi_{c_k,+,\EE} (x,t)\\
\nonumber &=-\log(t-x_0)
+\log \EE-2\log\biggl (\frac{1}{2}+
\sum_{j=1}^\infty b_{j+1} \left(\frac{t-x_0}{\EE}\right)^{j}\biggr )
\, .
\end{align}

We now consider the case $c_k+t<x<c_k+t+\EE$. Then
\begin{equation*}
\sqrt{\pi} I^{1/2}_{+,t} \phi_{c_k,+,\EE} (x,t)=
\int_{x_0-\EE}^{t} \frac{1}{\sqrt{x_0-\tau}} \frac{1}{\sqrt{t-\tau}}\, \D \tau \, .
\end{equation*}
Using again 
\eqref{195.01},
and we find 
\begin{align*}
\int_{x_0-\EE}^{t} \frac{1}{\sqrt{x_0-\tau}} \frac{1}{\sqrt{t-\tau}}\, \D \tau&=
\log (x_0-t) -2 \log (\sqrt \EE-\sqrt{t-x_0+\EE}) \, .
\end{align*}
Similarly as for \eqref{similar}, we find, using \eqref{Taylor},
\begin{align*}
-2 &\log (\sqrt\EE-\sqrt{t-x_0+\EE})=
- 2 \log\bigl (\sqrt \EE \biggl (1-\sqrt{1+\frac{t-x_0}{\EE}} \biggr ) \bigr )
\\
&\quad=
- 2 \log\bigl (\frac{x_0-t}{\sqrt \EE } \sum_{n=1}^\infty b_n \biggl (\frac{t-x_0}{\EE} \biggr )^{n-1}  \bigr )
\\
&\quad=+\log \EE - 2 \log (x_0-t)- 2 \log\biggl (\frac{1}{2}-\sum_{j=1}^\infty  b_{j+1} \left (\frac{t-x_0}{\EE}\right )^j\biggr )\, .
\end{align*}
We have thus shown that if $c_k+t-\EE<x<c_k+t+\EE$ then
\begin{align*}
  \sqrt{\pi}\cdot & I^{1/2}_{+,t} \phi_{c_k,+,\EE} (x,t)=
-\log|x_0-t|
+\log \EE
- 2 \log\biggl (  \frac{1}{2}+\sum_{j=1}^\infty  b_{j+1} \left (\frac{t-x_0}{\EE}\right )^j\biggr )\, .
\end{align*}
With \eqref{one}, the above identity shows the claim on the
right-handed spike ($\sigma=+$).

For the left-handed spike, we have, recalling $x_0=x-c_k$
and \eqref{QQ+},
$$
\phi_{c_k,-,\EE} (x,t)= 
\phi_{c_k,+,\EE} (x,2x_0-t)= Q \circ T_{2 x_0} (\phi_{c_k,+,\EE} ) (x,t)\, .
$$
Thus,  we find
\begin{align*}
 I^{1/2}_{-,t} \phi_{c_k,-,\EE} (x,t)
&= I^{1/2}_{-,t} \circ Q \circ T_{2 x_0} (\phi_{c_k,+,\EE} ) (x,t)\\
&=  I^{1/2}_{+,t} \circ T_{2 x_0}(\phi_{c_k,+,\EE}) (x,-t)
= 
I^{1/2}_{+,t} \phi_{c_k,+,\EE} (x,-t+2x_0) \, .
\end{align*}
Finally, note that $-(x-c_k-t)=x-c_k-(2x_0-t)$.
\end{proof}

\section{Vanishing of $X_t$ at the image of endpoints}
\label{vanX}
It is sometimes convenient to assume that $X_t$ vanishes at the
endpoints $\pm 1$. This can be achieved in several ways, as we explain now.
For $t\in (1,2)$,  setting 
$$\tilde f_t(y)=\frac{f_t(|a_t|y)}{|a_t|}=\frac{t}{|a_t|}-|a_t| y^2=
a_t y^2-\frac{t}{a_t}\, ,
$$
gives a family of maps $\tilde f_t$ preserving  $[-1,1]$, with  $c_{0,t}=0$, and such that
$$\tilde f_t(-1)=-1=\tilde f_t(1)\, , \,\,  \forall\,  t\in (-1,2)\, ,
$$
so that $\partial_t \tilde f_t$ vanishes at the
endpoints $-1$ and $1$. This is a transversal family of quadratic maps in the sense
of  Tsujii \cite{Tsu}, or \cite{AM2, ALM}. The formula for $\partial \tilde f_t$
being unwieldy,  it is convenient
to work with the family $\tilde h_t:[0,1]\to [0,1]$ defined by
$\tilde h_t(x)=t x(1-x)$, $t \in (1,4]$.
The critical point of each $\tilde h_t$ is $1/2$, and $X_t(x)=\partial_t \tilde h_t 
\circ \tilde h_t^{-1}=tx$ (there is a typo in \cite[eq. (2)]{BBS} where it is
stated incorrectly that $X^{\tilde h_t}_t(x)\equiv t$). Then
$\tilde h_t(0)=\tilde h_t(1)=0$ for all $t$, so that $\partial_t \tilde h_t$  vanishes at the
endpoints $0$ and $1$. A variant of $\tilde h_t$ 
is $h_t(x)=t(1-x^2)-1$ for $t\in (1,2]$ on $[-1, 1]$ (there, $c=0$ and $h_t(-1)=h_t(1)=-1$).
However, the formulas for $f_t$ are easier to manipulate than those of $\tilde f_t$, $h_t$, or $\tilde h_t$, compensating for the non vanishing 
of the vector field at the endpoints of a common invariant
interval. In addition\footnote{See \cite[Lemma 2.1]{Tsu} for
an analogous remark.}, for any fixed $t_0 \in (1,2)$ 
and  all $t$ close enough to $t_0$, we may extend $f_t$ on $[-2,  2]$
to a $C^4$ map, also called $f_t$,  with negative Schwarzian derivative,  such that 
$Df_t$ is positive on $[-2, t-t^2]$ and
negative on $[t, 2]$, with $f_t(-2)=f_t(2)=-2$ and  $f_t -  f_{t_0}=
O(|t-t_0|)$. (The extended family $f_t$ is not needed in the present paper, but
we expect it should be useful to prove
equality [ii] in Conjecture~\ref{laconj} in future works.)
Finally, using that $c_{2,t}=t-t^2>-|a_t|$, one can easily show that
for any $t_0\in (1,2)$ there are $\epsilon >0$ and an
interval $I'_{t_0}\subset (-2,2)$   such that $f^k_t (I'_{t_0})\subset I'_{t_0}$
for all $t \in (t_0-\epsilon, t_0+\epsilon)$. In other words, $f_t$ for $t \in (t_0-\epsilon, t_0+\epsilon)$ is  a transversal family
of unimodal maps on $I'_{t_0}$ in the sense of Tsujii \cite{Tsu}
since, recalling \eqref{eq.transversal}, if $\JJ({t_1})$ is absolutely convergent then $\JJ({t_1})\ne 0$.

\section{Averaging}\label{wormell}
For regular parameters $t$ (also called ``hyperbolic''), the physical measure $\mu_t^{sink}= P^{-1}
\sum_{j=1}^P \delta_{x_{t,j}}$ is atomic, supported
on an attracting periodic orbit $f^{P} (x_{1,t})=x_{1,t}$
with $P=P(t)$, and can be obtained as
$$
\lim_{k \to \infty} \sum_{j=0}^{P-1}\int \LL_t^{k+j}(\psi) \phi  \, \D m=
\lim_{k \to \infty} \sum_{j=0}^{P-1}\int \psi (\phi \circ f_t^{k+j}) \, \D m
= \int \psi  \, \D m  \cdot \frac{1}{P}
\sum_{j=1}^P \phi({x_{j,t}})
$$
The convergence is however not uniform in any interval
of regular parameters so one cannot a priori
sum over $k$ even if $\int_{I_t}\psi  \, \D m =0$. 
	
Since almost every parameter is either regular or
stochastic \cite{lyub} it is  natural to consider, for 
a $C^1$ observable, say, a Collet--Eckmann parameter $t$, and  $\epsilon >0$, the double
Lebesgue integral
$$
\AA_\epsilon(t):=\int_{[-\epsilon, \epsilon]} \int \phi(x) \D\mu_{t+\delta}(x)\,  \, \D \delta\, ,
$$
where $\mu_{t+\delta}=\mu_{t+\delta}^{sink}$ for regular parameters, and
$\mu_{t+\delta}=\rho_{t+\delta} dm$ is the SRB measure for stochastic parameters.
Then it is not hard to see, using Lebesgue differentiation,
that  the (ordinary) $t$-derivative $\AA_\epsilon'(t)$ exists for almost every $\epsilon >0$
and coincides with $\int \phi(x) \D\mu_{t+\epsilon}(x)$
(see	also \cite[\S3]{WG1} and  \cite[(16)]{WG2}).
This does not resolve the paradox described in the introduction, since the derivative
depends on $\epsilon$ and does not coincide with $\Psi_\phi(1)$ in general.
(Note that  the ``weakening of the linear response problem'' in the introduction
of \cite{BBS} --- existence and value of the limit as $\epsilon \to 0$
of the derivative $\AA'_\epsilon$ ---  does not
explain the paradox either.)

\section{Complements to the proof of Theorem~\ref{WTF}}
 
We record here interesting facts which  are
not needed for our proofs.

\begin{remark}[Spectrum on a pole-extended Banach space]
 For $t \in \MT$,
let $\Lambda=\Lambda_t:[-a_t,a_t]\to [0,1]$ be the absolutely continuous bijection
defined by $\Lambda_t(x)=\int_{-a_t}^x \rho_t(u) \D u$.
Then (see \cite{Og, Sch0}) the map $F_t:[0,1]\to [0,1]$ defined by $F_t  (\Lambda_t(x))=\Lambda_t (f_t(x))$ is Markov (for the partition $J_\ell$ defined
by the
endpoints $\Lambda_t(c_k)$, $k=0, \ldots L+P-1$), and $F_t$ is $C^1$
on the interior of each interval of monotonicity $J_\ell$,
with $\inf| F_t'|>1$. At the endpoints, we have\footnote{This is an  exercise left to the reader.}
\begin{equation}
\label{magicOg}
(D_t F_t)^k(\Lambda_t(c_{1,t})_\pm)=s_k \sqrt{ |(Df_t^k)(c_{1,t})|}
\end{equation}
(taking right or left-sided limits in the left-hand side according to the dynamical
orbit).
 In fact, $G_t:=1/F_t'$ extends to a  $C^1$ 
map on the closure of each $J_\ell$, with $\sup G''_t<\infty$.
On the Banach space $\BB_{\Lambda_t}$ of bounded functions $\phi$ on $[0,1]$
such that  each $\phi|_{\intt J_\ell}$ is $C^1$ and 
admits a $C^1$ extension to the closure of $J_\ell$, the  transfer operator $\LL^\Lambda_t \phi(y)=
\sum_{F_t(z)=y} \phi(z)/|F'_t(z)|$ thus has spectral radius equal
to one, with a simple eigenvalue at $1$, for the eigenvector
$\rho^\Lambda_t(y):=\rho_t(\Lambda_t^{-1}(y))$, and the rest of the spectrum is contained in a
disc of radius $\kappa$ strictly smaller than $1$.
Then $\BB_t=\{ \phi \circ \Lambda_t \, , \, \phi \in \BB_{\Lambda_t}\}$ is  a Banach space for the norm induced by $\BB_{t,\Lambda}$
and the operator $\LL_t$ on $\BB_t$ inherits the spectral properties of $\LL^\Lambda_t$ on $\BB_{\Lambda_t}$.
Any element of $\BB_t$ belongs to $L^1(dm)$, with $\int_{I_t} |\varphi|\,  \D m \le \|\varphi\|_{\BB_t}$.
Recall the notations $\YY_t$, $\chi_k$, $\chi(\vec Y)$, and $\MM_t$ from
Remark~\ref{notC}. Then  we claim that we may extend  $\LL_t:\BB_t\to \BB_t$  to a bounded operator $\mathbb L_t$
on the Banach space $\mathbb{B}_t:=\BB_t\oplus \YY_t$, whose
 nonzero spectrum 
 is the union of the $P$th roots of $\sgn ( Df^P_t (c_{L}))$ with the nonzero spectrum
of $\LL_t$ on $\BB_t$. Morever the following holds if $\sgn (Df^p(c_L))=+1$:
First,  
setting
$\MM^0_t(\vec Y)=\MM_t(\vec Y)-\rho_t \int_{I_t} \MM_t(\vec Y) \D m$, and letting
$\vec S_t$ be the fixed point of $\mathbb S_t$,
$$\psi^*_t:=(\id-\LL_t)^{-1}  (  \MM^0_t(\vec S_t))\in \BB_t\, , $$
and
the (rank-two) spectral 
projector $ \Pi_1$  for the eigenvalue $1$ of $\mathbb L_t$
satisfies
\begin{equation*}
   \Pi_1(\varphi, \chi(\vec Y))=\int \varphi \D m  \cdot \rho_t
+\frac{ \langle \vec S^*_t, \vec Y\rangle }{\langle \vec S^*_t, \vec S_t\rangle} \cdot (\psi^*_t+\chi(\vec S_t))\, .
\end{equation*}
Second, if $\int_{I_t}  \MM_t(\vec S_t) \, \D m=0$
  then  $\psi^*_t+\chi(\vec S_t) \in \mathbb{B}_t$
is a fixed point\footnote{In this case,  the eigenvalue $1$
has geometric multiplicity two, i.e. there is no Jordan block.} of $\mathbb L_t$, while
if $\int_{I_t}  \MM_t(\vec S_t) \, \D m\ne 0$ then 
there exists a non zero $\nu^N_t\in\mathbb{B}_t^*$ 
such that 
the (rank-one) nilpotent operator for
the eigenvalue $1$ of $\mathbb L_t$ satisfies 
$N_1^2=0$ and $\Pi_{\BB_t}\circ  N_1=\nu^N_t \cdot \rho_t$.

We justify the claims above: First,  there exists $\kappa <1$
such that, on $\BB_t$, 
$$
\LL_t^j(\varphi)=\int \varphi \, \D y \cdot \rho_t + \QQ^j_t(\varphi)\, , \,\, j \ge 1\, ,
$$
where for any $\epsilon>0$ there exists
$C$ such that $\|\QQ_t^j\|_{\BB_t}\le C (\kappa+\epsilon)^j$ for all $j\ge 1$.
Note that $\MM_t(\vec Y)=\Pi_{\BB_t} \bigl ( \mathbb L_t(\chi(\vec Y))\bigr )$.
Identifying $\chi(\vec Y)$ and $\vec Y$, we have  
\begin{equation*}
\mathbb L_t( \varphi, \chi(\vec Y))=\left (
\begin{array}{ccc}
1&0&  \rho_t\cdot \int \MM_t  \, \D m  \\
0&\QQ_t&\MM^0_t
\\
0&0&\mathbb S 
\end{array}
\right )
\left (
\begin{array}{c}
\rho_t \cdot \int \varphi \, \D m  \\
\varphi-\rho_t \cdot \int \varphi \,\D m \\\
\vec Y
\end{array}
\right )\, .
\end{equation*}
In particular 
if  $1/z$ does not belong to the spectrum
of $\LL_t$ or $\mathbb S_t$, then 
\begin{equation*}
\bigl (\id -z\mathbb L_t\bigr )^{-1}=\left (
\begin{array}{ccc}
\frac 1{1-z}&0&  -\frac{ \rho_t \int \MM_t(\id-z\mathbb S )^{-1} \,\D m}{1-z} \\
0&(\id-z\QQ_t)^{-1}&-(\id-z\QQ_t)^{-1}\MM^0_t(\id-z\mathbb S )^{-1}
\\
0&0&(\id-z\mathbb S )^{-1}
\end{array}
\right ) .
\end{equation*}
If $\int_{I_t} \MM_t (\vec S_t) \, \D m=0$
(with  $\vec S_t$ the fixed point of $\mathbb{S}_t$)  then
 a direct computation gives that $\mathbb L_t$
inherits a (second) fixed point $\psi^*_t+\chi(\vec S_t) \in \mathbb{B}_t$ from
the fixed point $\vec S_t$ of $\mathbb{S}_t$.
If $\int  \MM_t(\vec S_t)\,  \D m\ne 0$ then the eigenvalue $1$ of $\mathbb L_t$
has algebraic multiplicity two but geometric multiplicity one,
and the associated  nilpotent  $N_1$  satisfies
our claim.  In both cases, the claim on $\Pi_1$ follows.
\end{remark}

\begin{remark}\label{CEinfty}
In the Collet--Eckmann case with an infinite postcritical orbit, the finite
matrix $\mathbb S_t$ appearing in the proof of Theorem~\ref{WTF} will be replaced
by a shift to the right, also denoted $\mathbb S_t$, weighted by $s_{1,k}=\pm 1$,  acting on a space of infinite sequences (for example
$\ell^\infty(\integer_+)$). Then $\mathbb S_t$ does not have any eigenvalues,
and its spectrum is contained in the closed unit disc. Also, 
$\mathbb M_t(z):=(\id -z \mathbb S_t)^{-1}$ is the infinite matrix
with $(\mathbb M_t(z))_{j,j}=1$,  $(\mathbb M_t(z))_{\ell,j}=(-1)^{1+\ell-j}
z^{\ell-j} \prod_{k=\ell}^{j-1} s_{1,k} =(-1)^{1+\ell-j}
z^{\ell-j}  s_{j-1,\ell} $
for $j<\ell$,
and $(\mathbb M_t(z))_{\ell,j}=0$ for other $\ell$, $j$.
\end{remark}

\end{appendix}


\begin{thebibliography}{DWY}


\bibitem{Ab0} N.H. Abel, {\it Solution de quelques probl\`emes \`a l'aide 
d'int\'egrales d\'efinies,} In
Gesammelte mathematische Werke. Leipzig: Teubner, 1,  11--27 (1881). 
(First publ. in Mag.
Naturvidenkaberne, 1, Christiania, 1823).

\bibitem{Ab} N.H. Abel, {\it Aufl\"osung einer mechanischen  Aufgabe,}
J. reine und angew. Mathematik
{\bf 1} 
153--157 (1826).
 
 \bibitem{ABLP} M. Aspenberg, V. Baladi, J. Lepp\"anen, and T.  Persson,
 {\it On the fractional susceptibility function of piecewise expanding maps,}
 arXiv:1910.00369.
	


\bibitem{A}
 A. Avila,  {\it Infinitesimal perturbations of rational maps,} Nonlinearity {\bf 15}
  (2002) 695--704.

\bibitem{ALM}
 A. Avila, M.  Lyubich,   and W. de Melo, 
    {\it Regular or stochastic dynamics in real analytic families of unimodal maps,}
  Invent. Math {\bf 154}
 (2003) 451--550.
     



\bibitem{AM2} A. Avila and C. G.  Moreira,  {\it 
    Statistical properties of unimodal maps: smooth families with negative {S}chwarzian derivative,}
  Ast\'erisque {\bf 286}
(2003) 81--118.
		
\bibitem{Ba} V. Baladi, {\it On the susceptibility function of piecewise expanding interval maps,} Comm. Math. Phys. {\bf 275} (2007) 839--859.



\bibitem{BBS}
V. Baladi, M. Benedicks, and D. Schnellmann,
{\it Whitney H\"older continuity of the SRB measure for transversal families
of smooth unimodal maps,} Invent. Math. {\bf 201} (2015) 773--844.



\bibitem{BMS} V. Baladi, S. Marmi, and D. Sauzin, 
{\it Natural boundary for the susceptibility function of generic
piecewise expanding unimodal maps,}
Ergodic Theory Dyn. Sys. {\bf 10}
(2013) 1--24.


\bibitem{BS0} V. Baladi and D. Smania,
{\it Linear response formula for piecewise expanding unimodal maps,}
Nonlinearity {\bf 21} (2008) 677--711 (Corrigendum, Nonlinearity {\bf 25} (2012) 2203--2205).


\bibitem{BS2} V. Baladi and D. Smania,
{\it Linear response for smooth deformations of generic
nonuniformly hyperbolic unimodal maps,}
Ann. Sc. ENS {\bf 45} (2012) 861--926.

\bibitem{Do} N. Dobbs, {\it Visible measures of maximal entropy in dimension one,} Bull. Lond. Math. Soc. {\bf 39} (2007) 366--376.

\bibitem{DM} N. Dobbs and  N. Mihalache, {\it Diabolical entropy,} Comm. Math.  Phys. {\bf 365} 1091--1123 (2019).

\bibitem{DT} N. Dobbs and  M. Todd,
{\it Free energy and equilibrium states for families of interval maps,} arXiv:1512.09245, to 
appear Memoirs A.M.S.

\bibitem{DH} A. Douady and J. H. Hubbard, 
{\it On the dynamics of polynomial-like mappings.} 
Ann. Sci. \'Ecole Norm. Sup.  {\bf 18}  (1985) 287--343.

\bibitem{hilf} R. Hilfer, {\it Threefold introduction to fractional derivatives}
in Anomalous Transport: Foundations and Applications,
R. Klages,  G. Radons, and I.M. Sokolov 
eds.
(2008) Wiley

\bibitem{JR} Y. Jiang and D. Ruelle, 
{\it Analyticity of the susceptibility function for unimodal Markovian maps of the interval,} Nonlinearity {\bf 18} (2005) 2447--2453.

\bibitem{KeNo} G. Keller and T. Nowicki, 
{\it Spectral theory, zeta functions and the distribution of periodic points for Collet-Eckmann maps,}
Comm. Math. Phys. {\bf 149} (1992) 31--69.

\bibitem{Kilbas} A.A. Kilbas, H.M. Srivastava,  J.J. Trujillo,  Theory and Applications of Fractional Differential Equations. Amsterdam,  Elsevier (2006).



\bibitem{Levin} G. Levin,
{\it On an analytic approach to the {F}atou conjecture},
Fund. Math.
{\bf 171} (2002) 177--196.



\bibitem{LL}
E.H. Lieb and M. Loss,  Analysis. Second edition. Graduate Studies in Mathematics, {\bf 14} American Mathematical Society, Providence (2001).

 \bibitem{dLS}   A. de Lima and D. Smania,   {\it  Central limit theorem
	for the modulus of continuity of averages of observables on transversal
	families of piecewise expanding unimodal maps,} J. Institut
Math.  Jussieu  {\bf 17} (2018) 673--733.

\bibitem{lyub}
  M. Lyubich, {\it Almost every real quadratic map is either regular or stochastic,}
 Ann. of Math. {\bf 156} (2002) 1--78.
    


\bibitem{mn} M. Martens and T. Nowicki, {\it Invariant measures for typical quadratic maps,} in: G\'eom\'etrie complexe et syst\`emes dynamiques (Orsay, 1995). Ast\'erisque {\bf 261} (2000), xiii, 239--252.

\bibitem{McMillan}
H.B. Dwight, {\it Tables of integrals and other mathematical data,} Third edition, McMillan (1957).

 \bibitem{MR} K.S. Miller and  B. Ross, 
{\it An Introduction to the Fractional Calculus and Fractional Differential Equations,}
John Wiley \& Sons (1993).

\bibitem{MT}  J. Milnor and W. Thurston, 
{\it On iterated maps of the interval,}
Dynamical systems ({C}ollege {P}ark, {MD}, 1986--87), Lecture Notes in Math. {\bf 1342} 465--563, 2018.

\bibitem{Mis81} M. Misiurewicz, 
{\it Absolutely continuous measures for certain maps of an interval,}
Inst. Hautes \'Etudes Sci. Publ. Math. {\bf 53} (1981) 17--51.

\bibitem{nvs} T. Nowicki and S. van Strien,
{\it Absolutely continuous invariant measures for $C^2$ unimodal maps satisfying the Collet--Eckmann conditions,} Invent. Math. {\bf 93} (1988) 619--635.

\bibitem{Og} A.I. Ognev, {\it Metric properties of a certain  class of mappings of
a segment,} Mat. Zametki {\bf 30} (1981) 723--736, 797.

\bibitem{OM} M. D. Ortigueira and
J.A.Tenreiro Machado,  {\it What is a fractional derivative?}
J. Comput. Physics
Volume 293 (2015) 4--13.


\bibitem{Pod} I. Podlubny,
Fractional differential equations. SanDiego:AcademicPress (1998).



\bibitem{Rag} C. Ragazzo, {\it Scalar autonomous second order ordinary
differential equations,} Qualitative Theory
 Dyn. Sys. {\bf 11} (2012)  277--415.
 
        
 \bibitem{Ru00} D. Ruelle,  {\it General linear response formula in statistical mechanics, and the fluctuation-dissipation theorem far from equilibrium,} Phys. Lett. A {\bf 245} (1998) 220--224.


\bibitem{Ru0} D. Ruelle, 
{\it Differentiating the absolutely continuous invariant measure of an interval map $f$ with respect to $f$,}
Comm. Math. Phys. {\bf 258} (2005) 445--453.

\bibitem{Ru}  D. Ruelle, {\it Structure and $f$-dependence of the A.C.I.M. for a unimodal map $f$ of Misiurewicz type,} Comm. Math. Phys. {\bf 287}
(2009) 1039--1070.


\bibitem{sam} S.G. Samko, A.A. Kilbas,  and O.I. Marichev,  Fractional integrals and derivatives. Theory and applications.    Gordon and Breach Science Publishers (1993).


\bibitem{Sch0} D. Schnellmann,
{\it Positive Lyapunov exponents for quadratic skew-products over a
Misiurewicz--Thurston map,} Nonlinearity {\bf 22} (2009) 2681--2695.



\bibitem{Sedro} J. Sedro, {\it Pre-threshold fractional susceptibility functions at Misiurewicz parameters} (2020) arXiv:2011.13648.



\bibitem{Thun} H. Thunberg,  {\it Unfolding of chaotic unimodal maps and the parameter dependence of natural measures,} Nonlinearity, {\bf 14} (2001) 323--337.

\bibitem{Tsu} M. Tsujii, {\it Positive Lyapunov exponents in families of one-dimensional dynamical systems,} Invent. Math. {\bf 111} (1993) 113--137. 

\bibitem{Tsu1} M. Tsujii,
{\it On continuity of Bowen--Ruelle--Sinai measures in families of one-dimensional maps,}
Comm. Math. Phys. {\bf 177} (1996) 1--11.

\bibitem{Tsu2} M. Tsujii,  {\it A simple proof for monotonicity of entropy in the quadratic family,}
 Ergodic Theory Dynam. Systems {\bf 20} (2000) 925--933.
 
\bibitem{WG1} C.L. Wormell and G.A. Gottwald,  {\it On the validity of linear response theory in high-dimensional deterministic dynamical systems,} J. Stat. Phys. {\bf 172} (2018) 1479--1498.

\bibitem{WG2} C.L. Wormell and G.A. Gottwald, {\it Linear response for macroscopic observables in high-dimensional systems,}  Chaos {\bf 29} (2019) 113127.

\bibitem{Yo} L.S.Young, {\it Decay of correlations for certain quadratic maps,} Comm. Math. Phys. {\bf 146} (1992)  123--138.

\bibitem{Yo2} L.S.Young, {\it Recurrence times and rates of mixing,} Israel J. Math. {\bf 110} (1999) 153--188. 
\end{thebibliography}
\end{document}